\documentclass[12pt]{amsart}

\usepackage[a4paper,top=2.5cm,bottom=2.60cm,left=2.8cm,right=2.8cm]{geometry} 

\usepackage{amsmath}
\usepackage{amsfonts}
\usepackage{amsbsy}
\usepackage{amssymb}
\usepackage[normalem]{ulem}
\usepackage{times}
\usepackage{mathrsfs}
\usepackage{exscale}
\usepackage{graphicx}

\allowdisplaybreaks

\usepackage[colorlinks,citecolor=blue,linkcolor=blue,
            bookmarksopen,
            bookmarksnumbered
           ]{hyperref}
           
\usepackage{tikz}
\usetikzlibrary{patterns}

\usepackage{times}

\usepackage{color}
\definecolor{gree}{rgb}   {0.,   0.66,   0.25 }
\definecolor{marin}{rgb}   {0.,   0.4,   0.75}
\definecolor{orange}{rgb}   {0.8,   0.4,   0.}
\definecolor{brown}{rgb}   {0.4,   0.,   0.8}
\definecolor{greymg}{rgb}   {1,   0.,   0.8}
\definecolor{greygr}{rgb}   {0.5,   0.8,   0.5}

\newtheorem{theorem}{Theorem}[section]
\newtheorem{lemma}[theorem]{Lemma}
\newtheorem{proposition}[theorem]{Proposition}
\newtheorem{corollary}[theorem]{Corollary}

\theoremstyle{definition}
\newtheorem{definition}[theorem]{Definition}

\newtheorem{notation}[theorem]{Notation}
\newtheorem{assumption}[theorem]{Assumption}

\theoremstyle{remark}
\newtheorem{remark}[theorem]{Remark}
\newtheorem{example}[theorem]{Example}

\numberwithin{equation}{section}

\newcommand{\ee}{\hskip0.15ex}

\newcommand{\dd}[1]{_{\raise-0.6ex\hbox{$\scriptstyle #1$}}}
\newcommand{\di}{\displaystyle}
\newcommand{\on}[1]{\big|_{#1}}

\newcommand {\Norm}[2]{ \mathchoice
    {|\ee #1\ee|\dd{#2}\,}
    {| #1 |_{#2}}
    {| #1 |_{#2}}
    {| #1 |_{#2}} }
\newcommand {\DNorm}[2]{ \mathchoice
    {\|\ee #1\ee\|\dd{#2}\,}
    {\| #1 \|_{#2}}
    {\| #1 \|_{#2}}
    {\| #1 \|_{#2}} }
\newcommand {\Normc}[2]{ \mathchoice
    {|\ee #1\ee|\dd{#2}^2}
    {| #1 |_{#2}^2}
    {| #1 |_{#2}^2}
    {| #1 |_{#2}^2} }
\newcommand {\DNormc}[2]{ \mathchoice
    {\|\ee #1\ee\|\dd{#2}^2}
    {\| #1 \|_{#2}^2}
    {\| #1 \|_{#2}^2}
    {\| #1 \|_{#2}^2} }

\newcommand\bS{{\mathbb S}}

\newcommand\R{{\mathbb R}}
\newcommand\N{{\mathbb N}}

\newcommand\Z{{\mathbb Z}}

\newcommand{\boldone}{\text{\usefont{U}{bbold}{m}{n}1}}
\newcommand{\boldzero}{\text{\usefont{U}{bbold}{m}{n}0}}

\newcommand{\sA}{{\mathscr A}}

\newcommand{\sB}{{\mathscr B}}

\newcommand{\Dom}{\operatorname{\rm Dom}}
\newcommand{\tr}{\operatorname{\rm tr}}

\newcommand\cB{{\mathcal{B}}}
\newcommand\cC{{\mathcal{C}}}

\newcommand\cF{{\mathcal{F}}}

\newcommand\cI{{\mathcal{I}}}
\newcommand\cH{{\mathcal{H}}}

\newcommand\cL{{\mathcal{L}}}
\newcommand\cM{{\mathcal{M}}}

\newcommand\cO{{\mathcal{O}}}

\newcommand\cS{{\mathcal{S}}}

\newcommand\cU{{\mathcal{U}}}
\newcommand\cV{{\mathcal{V}}}

\newcommand{\rP}{{\mathsf{P}}}

\newcommand{\rd}{{\mathrm d}}

\newcommand{\rw}{{\mathsf{w}}}

\newcommand{\sfe}{{\mathsf e}}
\newcommand{\sfE}{{\mathsf E}}
\newcommand{\sfEp}{{\mathsf E}_{+}}
\newcommand{\sfu}{{\mathsf u}}
\newcommand{\sfU}{{\mathsf U}}

\newcommand{\ueps}{u_\varepsilon}

\newcommand {\gA}{\mathfrak{A}}
\newcommand {\gB}{\mathfrak{B}}

\newcommand {\gE}{\mathfrak{E}}
\newcommand {\gF}{\mathfrak{F}}

\newcommand {\gL}{{\mathfrak L}}

\newcommand {\gN}{{\mathfrak N}}

\newcommand {\gotf}{{\mathfrak f}}
\newcommand {\gotg}{{\mathfrak g}}
\newcommand {\ga}{{\mathfrak a}}

\newcommand {\gbB}{\boldsymbol{\mathfrak B}}
\newcommand {\gbc}{\boldsymbol{\mathfrak c}}

\newcommand {\Id}{\mathbb I}

\newcommand{\dir}{{\mathrm{dir}}}
\newcommand{\loc}{{\mathrm{loc}}}

\newcommand\DelS{\nabla_{\bS^{n-1}}}

\begin{document}

\title[Dirichlet problem on perturbed conical domains]{\bf Dirichlet problem on perturbed conical domains via converging generalized power series}

\author{Martin Costabel}

\author{Matteo Dalla Riva}

\author{Monique Dauge}

\author{Paolo Musolino}

\address{IRMAR UMR 6625 du CNRS, Universit\'{e} de Rennes, France}

\address{Dipartimento di Ingegneria, Universit\`a degli Studi di Palermo, Italy}

\address{IRMAR UMR 6625 du CNRS, Universit\'{e} de Rennes, France}

\address{Dipartimento di Matematica ``Tullio Levi-Civita'', Universit\`a degli Studi di Padova, Italy}


\date{\bf\today}



\begin{abstract}
We consider the Poisson equation with homogeneous Dirichlet conditions in a family of domains in $\R^{n}$ indexed by a small parameter $\varepsilon$. The domains depend on $\varepsilon$ only within a ball of radius proportional to $\varepsilon$ and, as $\varepsilon$ tends to zero, they converge in a self-similar way to a domain with a conical boundary singularity. We construct an expansion of the solution as a series of fractional powers of $\varepsilon$, and prove that it is not just an asymptotic expansion as $\varepsilon\to0$, but that, for small values of $\varepsilon$,  it converges normally in the Sobolev space $H^{1}$. The phenomenon that solutions to boundary value problems on singularly perturbed domains may have \emph{convergent} expansions is the subject of the Functional Analytic Approach by Lanza de Cristoforis and his collaborators. This approach was originally adopted to study small holes shrinking to interior points of a smooth domain and heavily relies on integral representations obtained  through layer potentials. To relax all regularity assumptions, we forgo boundary layer potentials and instead exploit expansions in terms of eigenfunctions of the Laplace-Beltrami operator on the intersection of the cone with the unit sphere.   Our analysis is based on a two-scale cross-cutoff ansatz for  the solution. Specifically, we write the solution as a sum of a function in the slow variable multiplied by a cutoff function depending on the fast variable, plus a function in the fast variable multiplied by a cutoff function depending on the slow variable.  While the cutoffs are considered fixed, the two unknown functions are solutions to a $2\times2$ system of partial differential equations that depend on $\varepsilon$ in a way that can be analyzed in the framework of generalized power series when the right-hand side of the Poisson equation vanishes in a neighborhood of the perturbation. In this paper, we concentrate on this case. The treatment of more general right-hand sides requires a supplementary layer in the analysis  and is postponed to a forthcoming paper.
\end{abstract}

\maketitle

{\footnotesize
\parskip 1pt
\tableofcontents
}

\section{Introduction}
\label{s:1}
\subsection{Aim}
\label{SS:situation}
Consider a bounded domain $\Omega$ in $\R^{n}$ that coincides near the origin with a cone $\Gamma$. Assume that $\varepsilon$ is a small positive parameter and that the domain is perturbed in a neighborhood of size $\varepsilon$ near the vertex of the cone. The perturbed domain is thought of as a member of a family of domains $\Omega_{\varepsilon}$, where, as $\varepsilon\to0$, the perturbation shrinks to the vertex point in a self-similar way. The goal is to compare the solution $u_{\varepsilon}$ of a boundary value problem (BVP) on  $\Omega_\varepsilon$ with the solution $u$ on the unperturbed domain $\Omega$. 

This type of singular perturbation problem is typically addressed with methods of asymptotic analysis, as shown in the books by Maz'ya, Nazarov, and Plamenevskij \cite{MaNaPl00i,MaNaPl00ii} (see especially \cite[Chapter 4]{MaNaPl00i}). In particular, there is a vast literature on problems where a small perturbation is located in the interior of a domain, a situation which corresponds, in our setting, to the case where the cone $\Gamma$ coincides with the whole space. For such problems, there exists a variety of methods that yield asymptotic expansions, mostly stemming from the method of Matched Asymptotic Expansions \cite{Il92} and its variants.

Lanza de Cristoforis \cite{La02} observed that for these interior perturbations, the dependence of $u_{\varepsilon}$ on $\varepsilon$ can often be described by an analytic function defined in a neighborhood of $\varepsilon=0$. Thus,  not only by asymptotic approximations but by \emph{convergent} power series of $\varepsilon$. To prove this, he first represented the solution using layer potentials supported on boundary components associated with different scales and then transformed the singularly perturbed BVP into a system of equations for which the limit as $\varepsilon\to0$ corresponds to a \emph{regular} perturbation. These few lines synthesize the core idea of the method he called the ``Functional Analytic Approach'' (FAA), which has since been generalized in many directions, as overviewed in \cite{DaLaMu21}. In particular, in our previous paper \cite{CoDaDaMu17}, the FAA was applied to a problem in a perturbed two-dimensional corner domain. By using a conformal mapping, we reduced the problem to one in a domain with interior holes, which was then analyzed using the standard toolkit of the FAA.

In the present paper, we introduce a generalization of the FAA that overcomes several of the drawbacks found in \cite{CoDaDaMu17}. Specifically, we do not use conformal mappings, making our new method applicable in any dimension $n\ge2$ (and not just for $n=2$), and we do not use boundary integral representations, which allows us to greatly reduce the regularity requirements for the domains and the perturbations. Instead, we exploit expansions on a basis of homogeneous harmonic functions in the cone that are related to the eigenfunctions of Laplace-Beltrami operator of the cone's base (its intersection with the unit sphere). As a result, we don't need any regularity assumptions on the base of the cone or the perturbations, and we can analyze a large variety of perturbations, such as cracks, several cones touching at their tips, holes tangent to the boundary, and ``the rounding of the corner,'' which is the approximation of the corner domain by smooth domains. Many of these were excluded in \cite{CoDaDaMu17}. The rounding of the corner, for example, could not be treated because it led to non-Lipschitz boundary components. The price to pay for this broader generality is that, for now, we only consider homogeneous Dirichlet boundary conditions, the Laplace operator, and right-hand sides that vanish in a neighborhood of the tip of the cone.

Depending on the specific problem under consideration, the application of the FAA yields different analyticity properties of the solution. In simpler cases, the solution depends analytically on $\varepsilon$, whereas in more complex scenarios, a richer analytic structure is required, involving multiple scales which can include fractional powers of $\varepsilon$ or even $1/\log\varepsilon$, $\varepsilon\log\varepsilon$, and so on.

For the problem analyzed in this paper, particularly when $n\geq3$, we encounter a situation where an infinite number of scales with fractional powers are necessary. These powers correspond to the eigenvalues of the Laplace-Beltrami operator on the base of the cone. The solutions $u_\varepsilon$ are then expressed in terms of convergent generalized power series of $\varepsilon$. This is a new level of generality compared to previous works. The underlying algebraic framework is based on the theory of generalized power series developed by Hahn and others \cite{Ha07}, which we will recall and adapt to our purposes.

\subsection{State of the art}
\label{SS:state}

Understanding how the solution of a BVP depends on perturbations of the domain finds applications in various theoretical and practical scenarios. For instance, it is relevant in the study of inverse problems (e.g., Ammari and Kang \cite{AmKa07}), topological optimization (e.g., Novotny and Soko\l owski \cite{NoSo13}), and composite materials (e.g., Movchan, Movchan, and Poulton \cite{MoMoPo02}). One common setting is that of singular perturbations, where a problem is defined on a regular domain, but as a positive parameter $\varepsilon$ tends to zero, some regularity is lost or the topology changes. Examples include domains with shrinking holes or rounded corners becoming sharp (the ``rounding of a corner'' mentioned above). The prevailing approach in the literature is that of  asymptotic analysis, which employs techniques like Matched Asymptotic Expansions (as in Il'in \cite{Il78,Il92}) or Multiscale Approximation Methods (like the Compound Approximation Method of Maz'ya, Nazarov, and Plamenevskij \cite{MaNaPl00i,MaNaPl00ii} and Kozlov, Maz'ya, and Movchan \cite{KoMaMo99}). In both cases, an iterative algorithm is used to obtain asymptotic approximations of the solution as $\varepsilon$ tends to zero. The results are usually presented as a finite sum of functions of $\varepsilon$ plus a remainder that tends to zero as $\varepsilon\to0$. The vanishing order of the remainder is known and increases when more terms are considered in the expansion. However, there is, at least in general, little information about the size of the remainders for a fixed positive value of $\varepsilon$. We may not even know if, for $\epsilon$ fixed, the remainders converge to a limit when the number of terms goes to infinity. As a consequence, we cannot usually claim the convergence of the expansion for small positive values of $\varepsilon$. For many classical singular perturbation problems, where not the domain but the coefficients of the differential operator depend on $\varepsilon$, the series do indeed not converge (see \cite{OMalley2014} for examples).

The FAA developed by Lanza de Cristoforis addresses the convergence issue by showing that solutions to perturbed BVPs can often be expressed in terms of analytic functions of the perturbation parameter $\varepsilon$. In these situations, the asymptotic expansions are, in fact, convergent power series. 

So far, the strategy of the FAA consists in reducing the problem to the boundary using integral equations. This leads to  functional equations of the form 
\begin{equation}\label{eq1}
\cM[\varepsilon](\cU) = \cF\,,
\end{equation}
where $\cM$ is a map defined on a (possibly one-sided) neighborhood of $\varepsilon=0$ with values in a suitable space of operators and such that $\cM[0]$ is the limit of $\cM[\varepsilon]$ as $\varepsilon\to0$ and is invertible. 
The solution $\cU$ of this system typically consists of several components associated with geometric objects of different scales (connected components of the boundary in the integral equation method) that can be brought to fixed sizes by rescaling. As $\varepsilon\to0$, the different scales become uncoupled, and then $\cM[0]$ corresponds to the simultaneous solution of several possibly independent problems, one of which is the original boundary value problem on the unperturbed domain. In the case of a perturbation by small holes of size $\varepsilon$, the additional problem is formulated on the ``perturbation pattern'' $\rP$, the unbounded exterior of the small holes rescaled to size one. This extension of the space of solutions is the mechanism that allows the embedding of the original singular perturbation problem in a regular perturbation problem.

The dependence of the BVP solution on $\varepsilon$ can be recovered from the dependence of $\cU$ on $\varepsilon$. In the ideal scenario, $\cM$ is analytic in $\varepsilon$ and, in particular,  it can be written as a power series of $\varepsilon$ that converges for $\varepsilon$ close to $0$. Since the inverse of a convergent power series is still a convergent power series, we can then deduce from \eqref{eq1} an analyticity result for $\cU$. 

In other applications, we do not have analyticity on $\varepsilon$, but still, we might represent the solution for $\varepsilon$ in a one-sided neighborhood $[0,\varepsilon_{0})$ of $0$ in terms of analytic functions of several variables evaluated at a vector whose entries are given by certain elementary functions of $\varepsilon$. The presence of analytic functions of several variables arises in various works, for example in \cite{La08}, where the two-dimensional Dirichlet problem forces the introduction of a scale of $1/\log\varepsilon$ alongside $\varepsilon$,  or in \cite{DaMu16, DaMu17},  where boundary value problems in a domain with moderately close holes are studied
and the size of the holes and their distance are defined by small parameters that may be of different size.  Similar results are obtained for holes approaching the outer boundary of a domain, as \cite{BoDaDaMu18, BoDaDaMu21} in collaboration with Bonnaillie-No\"el and Dambrine.

From the beginning, potential theoretic methods play a crucial role in the application of the FAA. The initial papers employing the FAA were dedicated to studying the Riemann map in planar perforated domains (cf. \cite{La02}). These works explored both singular and regular perturbations by analyzing specific properties of Cauchy and Cauchy-like integral operators. Subsequently, a similar analysis was extended to harmonic layer potentials (see Lanza de Cristoforis and Rossi \cite{LaRo04}), enabling the study of perturbation problems for the Laplace and Poisson equations. For instance, in \cite{La08}, Lanza de Cristoforis considered a Dirichlet problem for the Laplace equation in a domain with a small hole. 

Over time, the FAA has been utilized to handle various boundary conditions and different differential operators, all within the established framework of standard potential theory. Noteworthy extensions include applications to the Lamé equations \cite{DaLa10}, Stokes flow \cite{Da13}, and the Helmholtz equation \cite{AkLa22}.

Nowadays, the FAA is no longer the only method to obtain analyticity results of this kind. Other approaches have been proposed, still relying on potential theory and properties of related integral operators. For instance, Henríquez and Schwab \cite{HeSc21} use complex analysis methods to prove the ``shape holomorphy" of certain integral operators, leading to a (complex) analyticity result for the solution of a BVP in a regularly perturbed two-dimensional domain.  Additionally, Feppon and Ammari \cite{FeAm22} address the Dirichlet problem in a domain with a small hole using an approach comparable to the FAA, but with a specific focus on computational efficiency.

However, in some cases, one may not expect analytic dependence on the perturbation parameter $\varepsilon$, or even joint analytic dependence on $\varepsilon$ and elementary functions of $\varepsilon$. One example is presented in our previous paper \cite{CoDaDaMu17}. In \cite{CoDaDaMu17}, we applied the FAA to study a Dirichlet problem in a polygonal domain in the plane with holes that shrink to a corner point in a self-similar manner. There the ``natural'' scale is a fractional power of $\varepsilon$ whose exponent depends on the opening angle of the corner. When the right-hand side in the Poisson equation is not vanishing in a neighborhood of the corner, a phenomenon of small denominators appears: The integer powers of $\varepsilon$ coming from the Taylor expansion of the right-hand side may enter into resonance with (multiples of) the natural scale. In this case, the convergence of the  expansion of the solution can only be guaranteed by regrouping certain terms into packets of functions that scale according to different homogeneities and are themselves not homogeneous. The expansion does then not correspond to an analytic function of several variables applied to different powers of $\varepsilon$. 
In domains perturbed near higher-dimensional conical points, which is the subject of the present paper, the same difficulty appears with higher complexity, and therefore we avoid it for the time being, by assuming that the right-hand side is zero in a neighborhood of the tip of the cone. The analysis of the case of a general (analytic) right-hand side will be postponed to a forthcoming paper.  

In the present paper, we will introduce an ansatz on the form of the solution $\ueps$ on the perturbed domain. Specifically, we will use a two-scale  ``cross-cutoff'' representation:
\begin{equation}
\label{E:cross}
 \ueps(x) = 
  \Phi\Big(\frac{x}{\varepsilon}\Big) \,\sfu[\varepsilon](x)
   + \varphi(x) \,\sfU[\varepsilon]\Big(\frac{x}{\varepsilon}\Big)\,,
\end{equation}
where $\varphi$ and $\Phi$ are cutoff functions that are $1$ and $0$, respectively, near the origin and $0$ and $1$, respectively, near infinity. 
A similar ansatz was used by Maz'ya, Nazarov, and Plamenevskij in \cite[Chap.~4]{MaNaPl00i} in the context of corner perturbations (very similar to the one in this paper), but for the application of asymptotic approximation methods. This approach was also employed in \cite{CaCoDaVi06} with Caloz and Vial, and in \cite{DaToVi10} with Tordeaux and Vial. However, these papers did not address the convergence of the expansions.

We present now a more detailed outline of this novel way to generalize the FAA and of the results thus obtained.

\subsection{Storyline of the paper}
We consider the Dirichlet problem for the Poisson equation in an $\varepsilon$-dependent bounded domain $\Omega_\varepsilon$ in $\mathbb{R}^n$. The family $\{\Omega_\varepsilon\}_\varepsilon$ represents a ``self-similar perturbation'' of a limiting domain $\Omega$ near its conical singularity at the origin $0$. Specifically, $\Omega$ coincides with a cone $\Gamma$ near $0$, and the perturbation is defined by the scaled version $\varepsilon\rP$ of an unbounded perturbation pattern $\rP$. The three elements $\{\Omega, \Gamma, \mathcal{P}\}$ (referred to as the ``generating triple,'' see Section~\ref{s:prelim}) completely determine $\Omega_\varepsilon$ through the following conditions:
\begin{itemize}
\item[{\em (i)}] Outside a ball $\sB$ of radius $r_0$ centered at the origin, $\Omega_\varepsilon$ coincides with $\Omega$;
\item[{\em (ii)}] Inside a ball $\sB_\varepsilon$ of radius $\varepsilon R_0$ and centered at the origin, $\Omega_\varepsilon$ coincides with the scaled pattern $\varepsilon\rP$;
\item[{\em (iii)}] In the annular transition region between the concentric balls $\sB$ and $\sB_\varepsilon$, the domains $\Gamma$, $\Omega_\varepsilon$, $\Omega$, and $\varepsilon\rP$ all coincide.
\end{itemize}
Besides this, our sole geometric assumption is the {\em invertibility} of the Dirichlet Laplace-Beltrami operator $\gL^\dir_{\hat\Gamma}$ on the spherical cap $\hat{\Gamma}:=\Gamma\cap\bS^{n-1}$ of the cone $\Gamma$. In other words 
\[
   \mbox{the first eigenvalue $\mu_1$ of $\gL^\dir_{\hat\Gamma}$ is positive.}
\]
Beside this assumption, which means that the Dirichlet conditions 
on $\hat\Gamma$ do not degenerate due to a complement set of null capacity, we do not need to impose any regularity conditions on the domain $\Omega$ or the cone $\Gamma$.

In Figure~\ref{F:1} we show an example of the rounding of a corner. The perturbation pattern $\sf P$ is infinitely smooth, $\Omega_{\varepsilon}$ is also smooth, but $\Omega$ has a sharp corner. This scenario is well-known and notoriously hard to analyze. Due to certain technical restrictions, we had to exclude it from our paper \cite{CoDaDaMu17}; however, it fits well into the framework of the present paper. 

\begin{figure}[h]
\caption{Rounded corner: \ 
(a) Domain $\Omega$, \ 
(b) $\Omega_{\varepsilon}$ (for $\varepsilon=0.2$), \  
(c) Pattern $\sf P$ }
\label{F:1}
\hglue-5ex
\begin{minipage}{0.32\textwidth}
\centering
\begin{tikzpicture}[x=0.37\textwidth,y=0.37\textwidth]
\filldraw [fill=red!10,draw=red,thick](0.0,-1.0) 
  to [out=-90,in=-90] (1,-1) 
  to [out=90,in=270] (1,1) 
  to [out=90,in=0] (-1,1) 
  to [out=180,in=180] (-1.0,0.0)
  -- (0.0,0.0) 
  -- (0.0,-1.0) ;
\draw[draw=blue](0.0,-0.7) arc(-90:180:0.7) (-0.7,0.0);
\path (-0.7,-0.0) node[anchor=north] {$r_{0}$} ;
\end{tikzpicture} \\
(a)
\end{minipage}
\begin{minipage}{0.32\textwidth}
\centering
\begin{tikzpicture}[x=0.37\textwidth,y=0.37\textwidth]
\filldraw [pattern=dots,pattern color=red!128,draw=red,thick](0.0,-1.0) 
  to [out=-90,in=-90] (1,-1) 
  to [out=90,in=270] (1,1) 
  to [out=90,in=0] (-1,1) 
  to [out=180,in=180] (-1.0,0.0)
  -- (-0.1,0) to [out=0,in=90] (0,-0.1) 
  -- (0.0,-1.0) ;
\draw[draw=gree](0.0,-0.18) arc(-90:180:0.18) (-0.18,0.0); 
\draw[draw=blue](0.0,-0.7) arc(-90:180:0.7) (-0.7,0.0);
\path (-0.7,-0.0) node[anchor=north] {$r_{0}$} ;
\path (-0.2,-0.0) node[anchor=north] {$\varepsilon R_{0}$} ;
\end{tikzpicture}\\
(b)
\end{minipage} \;\; 
\begin{minipage}{0.32\textwidth}
\centering\ \\[6ex]
\begin{tikzpicture}[x=0.37\textwidth,y=0.37\textwidth]
\draw[thick] (-1.4,0) -- (-0.5,0) to [out=0,in=90]  (0,-0.5) -- (0,-1.4) ;
\fill[pattern=dots] (0,-1.4) -- (1.4,-1.4) -- (1.4,1.2) -- (-1.4,1.2) -- (-1.4,0) -- (-0.5,0) to [out=0,in=90] (0,-0.5)-- (0,-1.3)  ;
\draw[draw=gree](0.0,-0.9) arc(-90:180:0.9) (-0.9,0.0);
\path (-0.9,-0.0) node[anchor=north] {$R_{0}$} ;
\end{tikzpicture} \\
(c)
\end{minipage}

\end{figure}
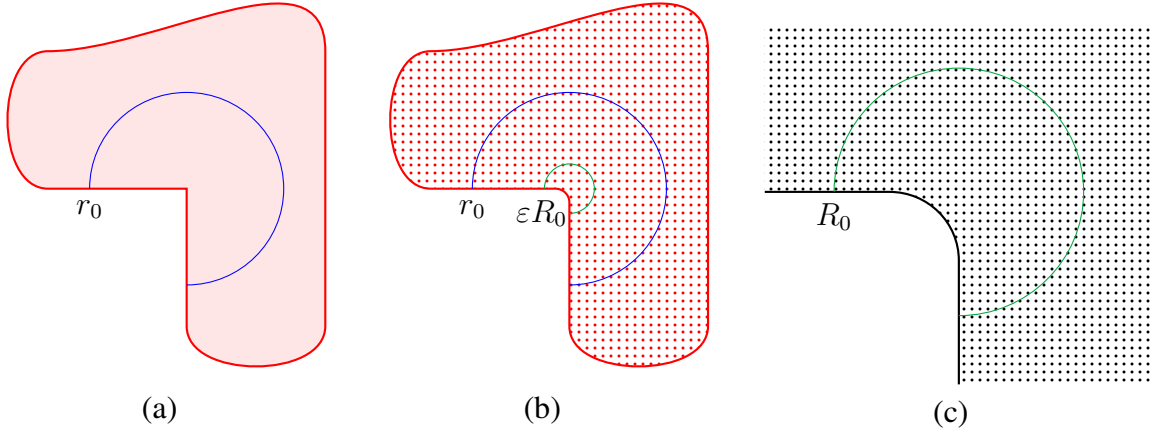
 

Given a function $f\in L^2(\R^n)$  {\em that is zero in a neighborhood of the origin, } 
we consider for each $\varepsilon$ the solution $\ueps$ of the Dirichlet problem
\begin{equation}
\label{eq:poissonf}
\begin{cases}
\begin{array}{ll}
   \ueps \in  H^1_0(\Omega_\varepsilon)\,,\\
   \Delta \ueps = f\on{\Omega_\varepsilon}\quad \text{ in }\Omega_\varepsilon\,. 
\end{array}
\end{cases}
\end{equation}
Our goal is to describe the map  $\varepsilon\mapsto\ueps$, for $\varepsilon$ close to $0$, in terms of {\em convergent series expansions} in powers of $\varepsilon$. 

The two scales of the problem are the natural (slow) variable $x$ and the scaled (rapid) variable $X:={x}/{\varepsilon}$. For the transfer between these two scales, instead of traces, we use cutoff functions $x\mapsto \varphi(x)$ (which is identically $1$ for $|x|$ small enough) and $X\mapsto \Phi(X)$ (which is identically $1$ for $|X|$ large enough). 

 We then introduce the ansatz \eqref{E:cross} for the form of the solution $\ueps$. The corresponding functions $\sfu[\varepsilon]$ and $\sfU[\varepsilon]$ are considered as independent unknowns, whereas the cutoffs $\varphi$ and $\Phi$ are fixed once and for all. By inserting the ansatz \eqref{E:cross} into problem \eqref{eq:poissonf} we construct a block $2\times 2$ operator matrix 
\begin{equation}\label{Me}
   \cM[\varepsilon]:=\
   \begin{pmatrix}
   \cM_{\Omega,\Omega} & \cM_{\Omega,\rP}[\varepsilon]\\
   \cM_{\rP,\Omega}[\varepsilon] & \cM_{\rP,\rP}
   \end{pmatrix}
\end{equation}
where the diagonal terms $\cM_{\Omega,\Omega}$ and $\cM_{\rP,\rP}$ are invertible operators in suitable function spaces on $\Omega$ and $\rP$, respectively, and are independent of $\varepsilon$.  The anti-diagonal blocks are the transfer operators.

This matrix is directly connected to problem \eqref{eq:poissonf} in the sense that, if we find solutions $u=\sfu[\varepsilon]$ and $U=\sfU[\varepsilon]$ for $\varepsilon>0$ to the system
\begin{equation}\label{sol.Me}
   \cM[\varepsilon]
   \begin{pmatrix} u \\ U \end{pmatrix}
   = \begin{pmatrix}  f \on \Omega \\ 0 \end{pmatrix}\,,
\end{equation}
then the solution $\ueps$ of the boundary value problem \eqref{eq:poissonf} is given by the cross-cutoff formula \eqref{E:cross} (see Theorem \ref{th:Meps}).

At this point, it remains to prove that $\cM[\varepsilon]$ is invertible and to determine the structure of its inverse.

To proceed, the first ingredient is a careful analysis of the transfer operators $\cM_{\Omega,\rP}[\varepsilon]$ and $\cM_{\rP,\Omega}[\varepsilon]$. Using a basis of homogeneous harmonic functions related to the eigenfunctions of the Laplace-Beltrami operator $\gL^\dir_{\hat\Gamma}$, we derive the following expansions in operator series, along with estimates for their terms:
\begin{subequations}
\begin{align}
\label{eq:MPO}
   \cM_{\rP,\Omega}[\varepsilon] &=
   \sum_{j \geq 1}\varepsilon^{\lambda_j^+}\mathcal{C}_j 
   \quad \text{with} \quad
   \|\mathcal{C}_j\| \le A\rho^{\lambda^+_j}, \\
\label{eq:MOP}
   \cM_{\Omega,\rP}[\varepsilon] &=
   \sum_{j \geq 1}\varepsilon^{-\lambda_j^-}\mathcal{B}_j
   \quad \text{with} \quad
   \|\mathcal{B}_j\| \le A\rho^{-\lambda^-_j},
\end{align}
\end{subequations}
for some constants $A$ and $\rho>0$ (cf. Theorem \ref{CjBjopnorm}).
Here the numbers $\lambda^\pm_j$ are related to the eigenvalues $\mu_j$ of $\gL^\dir_{\hat\Gamma}$ by the formula
\begin{equation}\label{lampm}
   \lambda_j^\pm:=1-\frac{n}{2} \pm \sqrt{\left(1-\frac{n}{2}\right)^2 + \mu_j}\,,
\end{equation}
which is well-known in the analysis of corner problems. The $\lambda^+_j$ are the primal singular exponents, and the $\lambda^-_j$ are the dual ones. As we assume that the first eigenvalue $\mu_1$ is positive, we have $\pm\lambda^\pm_j>0$ for all $j$.  The \emph{set} of numbers $\pm \lambda^\pm_j$ is denoted by $\sfE$ (the exponent set). The set $\sfE$ generates the \emph{monoid} $\sfE^\infty$ given as 
\begin{equation}
\label{eq:Einf}
   \sfE^\infty := \{\sfe=\sfe_1+\cdots+\sfe_k,\quad 
   \sfe_1,\ldots,\sfe_k\in \sfE\cup\{0\},\quad k\in\N^*\}.
\end{equation} 
Just like $\sfE$, the set $\sfE^\infty$ is discrete with a smallest positive element.
In the two-dimensional example of a plane sector $\Gamma$ of opening $\omega$,  the set $\sfE$ coincides with the semigroup of positive integer multiples of $\frac{\pi}{\omega}$
 and $\sfE^\infty=\sfE\cup\{0\}$.
In dimension $n\ge3$, $\sfE^\infty$ may still coincide with $\N$ as is the case when $\Gamma$ is a half-space, and there are simple examples where $\sfE^\infty$ has two distinct generators, see Example \ref{ex:sfE}. But for the general case we have to face the situation where $\sfE^\infty$ may have an arbitrary number of generators.

Based upon the expansions \eqref{eq:MPO} and \eqref{eq:MOP}, the second ingredient is the interpretation of $\cM[\varepsilon]$ as a \emph{generalized power series}  associated with the set of exponents $\sfE$.
Specifically, we write
\begin{equation}
\label{eq:Mef}
   \begin{pmatrix}
   \cM_{\Omega,\Omega}^{-1} & 0 \\
   0 & \cM_{\rP,\rP}^{-1}
   \end{pmatrix} \cM[\varepsilon] = \Id + \sum_{\sfe\in \sfE} \gA_\sfe\, \varepsilon^\sfe 
   := \Id + \gA[\varepsilon].
\end{equation}
As $\sfE$ is discrete with a positive smallest element,  the series $\Id + \gA[\varepsilon]$ is invertible in the space of {\em formal series} with exponent set $\sfE^\infty$
and its inverse is given by a formal Neumann series, see Theorem \ref{P:Neumann}:
\begin{equation}
\label{eq:Neum}
   (\Id + \gA[\varepsilon])^{-1} = \Id + \sum_{k=1}^\infty (-\gA[\varepsilon])^k
   = \Id \ + \sum_{\sfe\in \sfE^\infty\setminus\{0\}} \gB_\sfe\,\varepsilon^\sfe.
\end{equation}

Moreover, the estimates in \eqref{eq:MPO} and \eqref{eq:MOP}, along with Weyl's law for the eigenvalues of the Laplace-Beltrami operator, imply that $\cM[\varepsilon]$ is a \emph{normally convergent generalized power series}
with exponent set $\sfE^\infty$ (cf.~Theorem \ref{M-1}). Namely, there exists $\varepsilon_\star>0$ such that,  in a suitable operator norm, we have
\begin{equation}
\label{eq:normNeum}
   \sum_{\sfe\in \sfE^\infty\setminus\{0\}}\!\! \|\gB_\sfe\|\,\varepsilon^\sfe <\infty 
   \quad\mbox{for $0\le\varepsilon<\varepsilon_\star$.}
\end{equation}
This implies that we can solve equation \eqref{sol.Me} for all $\varepsilon\in(0,\varepsilon_\star)$ and find expansions of $\sfu[\varepsilon]$ and $\sfU[\varepsilon]$ as convergent generalized power series with exponent set $\sfE^\infty$:
\begin{equation}
\label{eq:soluU}
   \sfu[\varepsilon] = \sum_{\sfe\in \sfE^\infty} \varepsilon^\sfe\,\sfu_\sfe
   \quad\mbox{and}\quad
   \sfU[\varepsilon] = \sum_{\sfe\in \sfE^\infty} \varepsilon^\sfe\,\sfU_\sfe ,
\end{equation}
with $\sum_{\sfe\in \sfE^\infty}\!\! \|\sfu_\sfe\|\,\varepsilon^\sfe <\infty$ and $\sum_{\sfe\in \sfE^\infty}\!\! \|\sfU_\sfe\|\,\varepsilon^\sfe <\infty$ in energy norms.
Then we obtain a convergent multi-scale representation of the solution $\ueps$:
\begin{equation}
\label{eq:multiscale}
   \ueps(x) = \Phi\Big(\frac{x}{\varepsilon}\Big) \,
   \sum_{\sfe\in \sfE^\infty} \varepsilon^\sfe\,\sfu_\sfe(x)
   + \varphi(x) \,
   \sum_{\sfe\in \sfE^\infty} \varepsilon^\sfe\,\sfU_\sfe\Big(\frac{x}{\varepsilon}\Big),
   \quad \forall x\in\Omega_\varepsilon.
\end{equation}

Though simple and efficient, this representation is not intrinsic, as it depends on the choice of the cutoff functions. Nevertheless, it allows for finding convergent inner and outer expansions for $\ueps$, which are intrinsic. 

The traditional notion of \emph{inner expansion} refers to an expansion in rapid variables inside a near-field region, also known as \emph{microscopical expansion}. In our case, such an expansion takes the form of a converging linear combination of canonical profile functions:
\begin{equation}
\label{eq:innerX}
   \ueps(\varepsilon X) =  
   \sum_{j \ge 1} \sum_{\sfe \in \sfE^\infty} c_{j,\sfe} \, \varepsilon^{\sfe + \lambda^+_j}
   K^+_j(X)\,.
\end{equation}
Here, the $c_{j,\sfe}$ are scalar coefficients, and the $K^+_j$ are harmonic functions on the perturbation pattern $\rP$, satisfying zero Dirichlet conditions and the condition at infinity:
\begin{equation}
\label{eq:K+j}
   K^+_j(X) \to |X|^{\lambda^+_j} \psi_j\left(\tfrac{X}{|X|}\right) =: h^+_j(X)
   \quad \text{as } |X| \to \infty\,,
\end{equation}
where $\psi_j$ is an eigenfunction of the Laplace-Beltrami operator $\gL^\dir_{\hat\Gamma}$ associated with the eigenvalue $\mu_j$, cf. \eqref{lampm}. The functions $h^+_j$ are the singular functions of the Dirichlet problem on the limiting domain $\Omega$, as known from the seminal theory of Kondrat'ev \cite{Kondrat67}. The function $K^+_j$ can be represented as the difference of $\Phi h^+_j$ and a correcting function $Y^+_j$ living in a variational space on $\rP$:
\begin{equation}
\label{eq:K+jcor}
   K^+_j(X) = \Phi(X) h^+_j(X) - Y^+_j(X),\quad X \in \rP.
\end{equation}
We learn from a recent paper \cite[eq.~(6)]{Josien:2024} that such an expression can be extended to the framework of stochastic homogenization in sectors.

The sum \eqref{eq:innerX} converges for $X$ in $\rP$ inside a ball of radius $\simeq\varepsilon^{-1}$. We can rewrite it in terms of the slow variable, obtaining a sum in $\Omega_\varepsilon$ converging inside a fixed ball of radius independent of $\varepsilon$, see Theorem \ref{th:uepsinner}.

The \emph{outer (macroscopical) expansion} of $\ueps$ has the form:
\begin{equation}
\label{eq:outer}
   \ueps(x) =  u_0(x)
   + \sum_{j\ge 1} \sum_{\sfe\in\sfE^\infty} B_{j,\sfe}\,\varepsilon^{\sfe-\lambda^-_j}
    K^-_j(x)
\end{equation}
and converges for $x$ outside some ball of radius $\simeq\varepsilon$ (whereas traditionally $x$ is expected to lie outside some ball of radius $\simeq1$, see Theorem \ref{th:uepsouter}). Here, $u_0$ is the solution of the limiting Dirichlet problem \eqref{eq:poissonf}, the $B_{j,\sfe}$ are scalar coefficients, and the functions $K^-_j$ are non-variational solutions of the Dirichlet problem on the limiting domain $\Omega$, for which we have:
\begin{equation}
\label{eq:K-jcor}
   K^-_j(x) = \varphi(x) h^-_j(x) - Y^-_j(x),\quad x \in \Omega,
\end{equation}
where the $h^-_j$ are the dual singular functions (homogeneous of degree $\lambda^-_j$) and the $Y^-_j$ are corrector functions belonging to the variational space $H^1_0(\Omega)$. Note that such objects are widely used for determining coefficients of singularities, see e.g.~\  \cite{MazyaPlamenevskii:1984}, \cite{DaugeNicaise:1990I}.

A striking feature of the expansions \eqref{eq:innerX} (written in slow variables) and \eqref{eq:outer} is their convergence in a common transitional region, which forms an annulus with radii  $\simeq\varepsilon$ and $\simeq 1$. This convergence in a common region allows for the identification of certain combinations of coefficients, which coincide with those utilized in the Matched Asymptotics method, as discussed in \cite{DaToVi10}. The distinctive aspect here is that these combinations emerge naturally as a byproduct of our method, whereas in the Matched Asymptotics method, they are fundamental building blocks from the outset.

\subsection{Extensions}
 In conclusion of this introduction, we observe that extensions of our approach are possible in several directions:
\begin{enumerate}
\item We can apply our method to domains with multiple—let's say $L$—perturbed vertices. In such cases, instead of a $2 \times 2$ operator matrix forming the system $\cM[\varepsilon]$ in \eqref{eq1}, we would introduce a $(1+L) \times (1+L)$ matrix. This matrix would include one row for the equation on the limiting domain and  $L$ additional rows for the equations on the perturbed vertices. Apart from this modification, the approach remains largely unchanged from the one presented here, primarily because the perturbations do not interact with one another.
\item We can handle mixed Dirichlet-Neumann boundary conditions, given that certain regularity assumptions are met to ensure that the Laplace-Beltrami operator with induced boundary conditions on the base of the cone $\Gamma$ possesses the following properties:
\begin{enumerate}
\item It is invertible with compact resolvent;
\item Its spectral counting function grows at most polynomially.
\end{enumerate}
\item We could consider a self-similar material law within the perturbed $\varepsilon$-region. Specifically, we might replace the bilinear form associated with $\nabla \cdot \nabla$ with that associated with $a\left({x}/{\varepsilon}\right) \nabla \cdot \nabla$, where $a$ is a bounded function such that $a \ge a_0$ for some constant $a_0 > 0$. This formulation accounts for small inclusions or defects within the material.
\end{enumerate}

By contrast, addressing pure Neumann conditions or other types of partial differential equations would require more substantial modifications of the method. The case of small holes in the interior of $\Omega$ is also not directly accessible. Indeed, this scenario would correspond to setting $\Gamma=\R^{n}$, where our assumption on the sign of the first eigenvalue ($\mu_1>0$) would not hold.

Finally, if the right-hand side does not vanish identically but is
analytic inside the perturbed region, it is still possible to get
convergent expansions. This will require supplementary tools for
grouping together terms corresponding to clusters of exponents and will
be the subject of a forthcoming paper.

\subsection{Plan of the paper} The paper is organized as follows. In Section \ref{s:prelim}, we provide the preliminaries, defining the geometric and functional setting, and the family of BVPs we will study. Section \ref{s:FAAtrunc} introduces the Functional Analytic Approach with cutoffs and the operator matrix $\cM[\varepsilon]$. In Section \ref{s:specdev}, we examine the expansion of the transfer operators $\cM_{\rP,\Omega}[\varepsilon]$ and $\cM_{\Omega,\rP}[\varepsilon]$. Section \ref{s:M-1} deals with the inverse of $\cM[\varepsilon]$ and shows that it expands to a convergent generalized power series. In Section \ref{s:expansions}, we derive global, inner, and outer expansions of the solution of the BVP in terms of convergent generalized power series. At the end of the section, we compare the three expansions (global, inner, and outer) to derive a numerical iterative procedure that produces the coefficients. The paper concludes with two appendices. Appendix \ref{S:Kelvin} addresses the relationship between a Sobolev space in the intersection of the cone with a ball and a weighted Sobolev space in the intersection of the cone with the complement of a ball. Appendix \ref{S:gps} introduces a simplified and adapted version of Hahn et al.'s theory of generalized power series, and also discusses the case of convergent generalized power series.

\section{Preliminaries}
\label{s:prelim}

\subsection{Geometric setting} 

Throughout the paper we use the following notation: 
\begin{notation}
For $r>0$
\begin{itemize}
\item $\sB(r)$ is the open ball of center $0$ and radius $r$ in $\R^n$
\item $\sB^{\complement}(r)=\R^n\setminus\overline{\sB(r)}$ is the complement of the closed ball $\overline{\sB(r)}$ in $\R^n$
\item $\sA(r,r')=\sB^{\complement}(r)\cap\sB(r')$ is the open annular domain of center $0$ and radii $r, r'$.
\end{itemize} 
\end{notation}

 Here are the main assumptions that the {\em generating triple} of open sets $\Omega$, $\Gamma$, and $\rP$, involved in the definition of our self-similarly perturbed family of domains $\Omega_\varepsilon$, must satisfy.

\begin{assumption}
\label{as:GOP}\ 

\smallskip\noindent
{\em (i)} $\Gamma$ is an infinite open cone in $\R^n$ and its section $\hat\Gamma$ on the sphere $\bS^{n-1}$ satisfies
\begin{equation}\label{capacity>0}
\text{the capacity of $\bS^{n-1}\setminus\hat\Gamma$ (as a subset of $\bS^{n-1}$) is positive.}
\end{equation}

\smallskip\noindent
{\em (ii)} $\Omega$ is a bounded connected open set in $\R^n$ such that 
\begin{equation}
\label{eq:Omega}
   \Omega \cap \sB(r_0)=\Gamma \cap \sB(r_0)
\end{equation}
for some $r_0>0$.\\
\smallskip\noindent
{\em (iii)} $\rP$ is an unbounded open set in $\R^n$ such that 
\begin{equation}
\label{eq:rP}
   \rP \cap  {\sB^{\complement}(R_0)}=\Gamma \cap {\sB^\complement(R_0)}\,
\end{equation}
for some $R_0>0$.
\end{assumption}

By  Courtois \cite[Proposition 2.1]{Co95},  we can see that condition \eqref{capacity>0} is equivalent to the fact that (the space of extensions by zero of functions of) $H^1_0(\hat\Gamma)$ is a proper subspace of $H^1(\bS^{n-1})$. Moreover,  by  \cite[Theorem  1.1 (ii)]{Co95}, we find that condition \eqref{capacity>0}  is equivalent to the fact that
\begin{equation}\label{mu1>0}
   \mu_1>0
\end{equation}
where  $\mu_1$ is the first eigenvalue of the positive Dirichlet Laplace-Beltrami operator $\gL^\dir_{\hat\Gamma}$ on $\hat\Gamma$
or, equivalently, the best constant in the Poincar\'e inequality
\begin{equation}
\label{E:poincare}
 \mu_{1} \DNormc{u}{L^{2}(\hat\Gamma)} \le \Normc{u}{H^{1}(\hat\Gamma)}
 \quad\mbox{ for all }u \in H^{1}_{0}(\hat\Gamma) \,.
\end{equation}
Hence property \eqref{mu1>0} is very generally satisfied: 
as soon as $\bS^{n-1}\setminus\hat\Gamma$ has positive $(n-1)$-dimensional measure or even if $\bS^{n-1}\setminus\hat\Gamma$ is a sub-manifold of co-dimension $1$ in $\bS^{n-1}$, \eqref{capacity>0}, hence \eqref{mu1>0}, holds. This implies that cracks are allowed. However, since sub-manifolds of co-dimension $\ge 2$ have zero capacity, cf.~\cite[equation (7) and Proposition 2.4 (iii)]{Co95}, $\bS^{n-1}\setminus\hat\Gamma$ cannot be reduced to such sets (e.g.\ isolated points in dimension $n=3$).

We are ready for introducing the family of self-similarly perturbed domains $\{\Omega_\varepsilon\}$ associated to the triple $\{\Omega,\Gamma,\rP\}$.

\begin{definition}
\label{def:Omeps}
Let $\{\Omega,\Gamma,\rP\}$ be a generating triple satisfying Assumption \ref{as:GOP}.
Set 
\[
   \varepsilon_0:=\frac{r_0}{R_0}\,,
\] 
so that $r_0=\varepsilon_0 R_0$. Then the family $\{\Omega_\varepsilon\}$ is defined for any $\varepsilon \in [0,\varepsilon_0]$ by the following intersection conditions 
\[
\begin{aligned}
&\Omega_\varepsilon  \cap \sB^\complement(r_0)&&=\Omega \cap \sB^\complement(r_0)\, ,\\
&\Omega_\varepsilon  \cap \overline{\sA(\varepsilon R_0, r_0)}&&=
\Gamma \cap \overline{\sA(\varepsilon R_0, r_0)}\, \\
&\Omega_\varepsilon  \cap \sB(\varepsilon R_0)&&=\varepsilon \rP \cap \sB(\varepsilon R_0)\, .
\end{aligned}
\]
\end{definition}

\begin{remark}
\label{rem:regions}
It is obvious that the above conditions identify $\Omega_\varepsilon$ univocally. Moreover:

\smallskip\noindent
{\em (i)} \ Due to condition \eqref{eq:Omega}, $\Omega_\varepsilon$ coincides with $\Omega$ outside $\sB(\varepsilon R_0)$. Likewise, due to \eqref{eq:rP}, $\Omega_\varepsilon$ coincides with $\varepsilon\rP$ inside $\sB(r_0)$. 

\smallskip\noindent
{\em (ii)} \ Hence, if we replace $r_0$ by a smaller value $r'_0$ and $R_0$ by a larger value $R'_0$, then for $\varepsilon\le \frac{r'_0}{R'_0}$, we obtain the same domains $\Omega_\varepsilon$. 

\smallskip\noindent
{\em (iii)} \ The other consequence is that, when $\varepsilon$ tends to $0$, the domain $\Omega_\varepsilon$ tends to $\Omega$ as a set.
\end{remark}

Note that, besides the conditions in \eqref{capacity>0} on the capacity of $\bS^{n-1} \setminus \hat{\Gamma}$, we do not introduce any further regularity conditions on $\Omega$ and $\rP$. We do not even suppose that the cone $\Gamma$ or the pattern $\rP$ are connected. For instance, the union of a half-plane and a (disjoint) sector is an admissible $\Gamma$ in dimension $2$, and the pattern $\rP$ may join them or not. 

For an admissible example in dimension $2$, see Figure~\ref{F:2}, where $\Gamma$ is not connected. This example also shows features that imply that neither $\Omega_{\varepsilon}$ nor $\Omega$ are Lipschitz: cracks and holes touching the boundary.


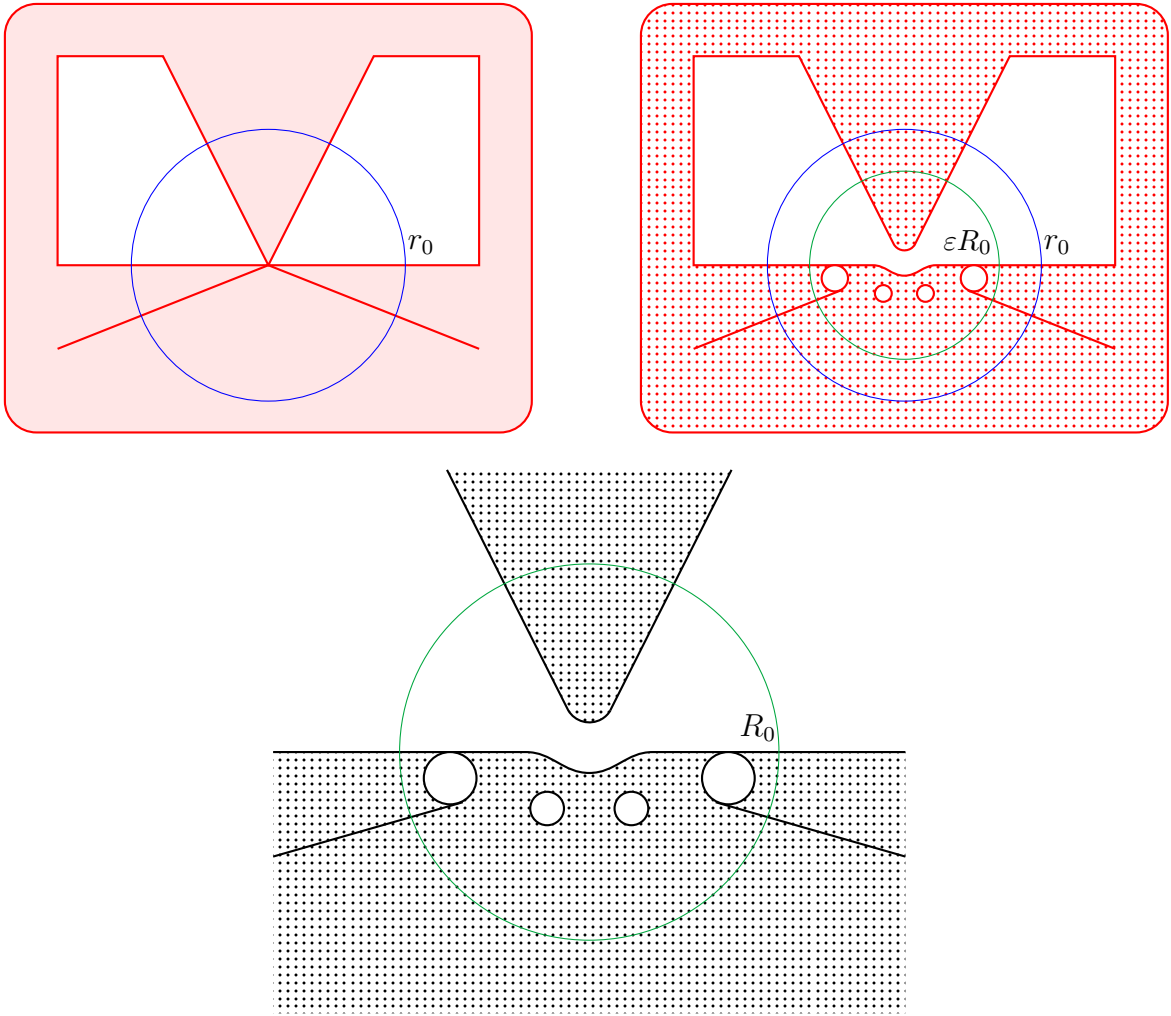
\begin{figure}[h]
\caption{$\Omega$, $\Omega_{\varepsilon}$ (for $\varepsilon=0.5$), and the perturbation pattern $\rP$}
\vskip1ex
\label{F:2}
\begin{minipage}{0.45\textwidth}
\begin{tikzpicture}[x=0.2\textwidth,y=0.2\textwidth]
\filldraw [fill=red!10,draw=red,thick]
(0,0) -- (1,2) -- (2,2) -- (2,0) -- (-2,0) -- (-2,2) -- (-1,2) -- (0,0)
(2.5,-1.3) -- (2.5,2.2) to [out=90,in=0] (2.2,2.5) 
  -- (-2.2,2.5)  to [out=180,in=90]  (-2.5,2.2) 
  -- (-2.5,-1.3) to [out=270,in=180]  (-2.2,-1.6) 
  -- (2.2,-1.6)  to [out=0,in=270] (2.5,-1.3) ; 
\draw[draw=red,thick](0,0) -- (-2,-0.8) (0,0) -- (2,-0.8); 
\draw[draw=blue] (0,0) circle(1.3);
\path (1.45,0) node[anchor=south] {$r_{0}$} ;
\end{tikzpicture}
\end{minipage}\hfill
\begin{minipage}{0.45\textwidth}
\begin{tikzpicture}[x=0.2\textwidth,y=0.2\textwidth]
\filldraw [pattern=dots,pattern color=red!128,draw=red,thick]
(0.1,0.2) -- (1,2) -- (2,2) -- (2,0) 
  -- (0.3,0) to [out=180,in=0] (0,-0.1)
  -- (0,-0.1) to [out=180,in=0] (-0.3,0)
  -- (-2,0) -- (-2,2) -- (-1,2) 
  -- (-0.1,0.2) arc(210:330:0.116) (0.1,0.2)
(2.5,-1.3) -- (2.5,2.2) to [out=90,in=0] (2.2,2.5) 
  -- (-2.2,2.5)  to [out=180,in=90]  (-2.5,2.2) 
  -- (-2.5,-1.3) to [out=270,in=180]  (-2.2,-1.6) 
  -- (2.2,-1.6)  to [out=0,in=270] (2.5,-1.3) ; 
\filldraw [fill=red!0,draw=red,thick] (-0.2,-0.27) circle(0.08);
\filldraw [fill=red!0,draw=red,thick] (0.2,-0.27) circle(0.08);
\draw [fill=red!0,draw=red,thick] (0.66,-0.125) circle(0.125);
\draw [fill=red!0,draw=red,thick] (-0.66,-0.125) circle(0.125);
\draw[draw=red,thick](-0.6,-0.24) -- (-2,-0.8) (0.6,-0.24) -- (2,-0.8); 
\draw[draw=gree] (0,0) circle(0.9);
\path (0.6,0) node[anchor=south] {\small$\varepsilon R_{0}$} ;
\draw[draw=blue] (0,0) circle(1.3);
\path (1.45,0) node[anchor=south] {$r_{0}$} ;
\end{tikzpicture}
\end{minipage}
\bigskip

\begin{center}
\begin{tikzpicture}[x=0.09\textwidth,y=0.09\textwidth]
\fill [pattern=dots]
(0.2,0.4) -- (1.35,2.7) -- (3,2.7) -- (3,0) 
  -- (0.6,0) to [out=180,in=0] (0,-0.2)
  -- (0,-0.2) to [out=180,in=0] (-0.6,0)
  -- (-3,0) -- (-3,2.7) -- (-1.35,2.7) 
  -- (-0.2,0.4) arc(210:330:0.232) (0.2,0.4)
(3,-2.5) -- (3,2.7) -- (-3,2.7) -- (-3,-2.5) -- (3,-2.5); 
\draw [thick]
(0.2,0.4) -- (1.35,2.7) ;
\draw [thick]
  (3,0) 
  -- (0.6,0) to [out=180,in=0] (0,-0.2)
  -- (0,-0.2) to [out=180,in=0] (-0.6,0)
  -- (-3,0) ;
\draw [thick]
  (-1.35,2.7) 
  -- (-0.2,0.4) arc(210:330:0.232) (0.2,0.4);
\draw [fill=red!0,thick] (-0.4,-0.54) circle(0.16);
\draw [fill=red!0,thick] (0.4,-0.54) circle(0.16);
\draw [fill=red!0,thick] (1.32,-0.25) circle(0.25);
\draw [fill=red!0,thick] (-1.32,-0.25) circle(0.25);
\draw[thick](-1.2,-0.48) -- (-3,-1) (1.2,-0.48) -- (3,-1); 
\draw[draw=gree] (0,0) circle(1.8);
\path (1.6,0) node[anchor=south] { $R_{0}$} ;
\end{tikzpicture}
\end{center}

\end{figure}


\subsection{Functional setting}
\label{ss:func}
We work in the standard framework of Sobolev spaces which allows for a clean variational formulation of our problems. For any open set $\cV$ in $\R^n$, the space $L^2(\cV)$ is the space of square integrable functions on $\cV$, and $H^1(\cV)$ is the space of $L^2(\cV)$ functions $g$ with gradient $\nabla g$ (in the distributional sense in $\cV$) belonging to $L^2(\cV)$. The space $H^1_0(\cV)$ is the closure  in $H^1(\cV)$ of the space $\cC^\infty_0(\cV)$ of smooth functions with compact support in $\cV$. The space $H^{-1}(\cV)$ is the dual space of $H^1_0(\cV)$ with the extension of $L^2(\cV)$ scalar product $\langle \cdot,\cdot\rangle_{\cV}$.
We use standard notation for the norm and semi-norm in $H^{1}(\cV)$:
$$
 \Norm{u}{H^{1}(\cV)} = \DNorm{\nabla u}{L^{2}(\cV)};\quad
 \DNorm{u}{H^{1}(\cV)} = \big(\DNormc{u}{L^{2}(\cV)} + \Normc{u}{H^{1}(\cV)}\big)^{\frac{1}{2}}\,.
$$

On any bounded domain $\cV$, the semi-norm and norm of $H^1(\cV)$ are equivalent over $H^1_0(\cV)$. From this, we deduce by the variational formulation and the Lax-Milgram theorem that the Laplace operator associated with the Dirichlet problem in variational form

\begin{equation}
\label{eq:DirVar}
\begin{array}{cccc}
    \Delta^{\sf dir}_\cV : & H^1_0(\cV) & \longrightarrow & H^{-1}(\cV) \\
    & u & \longmapsto & \Big( v \mapsto \big\langle \nabla u, \nabla v \big\rangle_\cV \Big)
\end{array} \quad \text{is an isomorphism.}
\end{equation}

As an obvious particular case, we obtain that problem \eqref{eq:poissonf} has a unique solution in $H^1_0(\Omega_\varepsilon)$ for any $f \in L^2(\Omega_\varepsilon)$, and that the same holds for the limit problem:

\begin{lemma}
\label{lem:DirOm}
The Dirichlet problem on $\Omega$ is uniquely solvable for any right-hand side $f\in L^2(\Omega)$.
\end{lemma}

By contrast, the Dirichlet problem for the Laplace operator on the unbounded domain $\rP$ is not solvable in $H^1_0(\rP)$, but in a larger weighted space $H^1_{\rw,0}(\rP)$. 
For any (possibly unbounded) domain $\cV$ in $\R^n$ we introduce the weighted $L^2$ space
$L^2_{\rw}(\cV)$ and its dual $L^2_{\rw^*}(\cV)$,
\begin{equation}
\label{E:L2w}
\begin{aligned}
   L^2_{\rw}(\cV) &= \{G\in L^2_\loc(\cV),\quad \int_{\cV} \frac{|G(X)|^2}{(1+|X|)^{2}}\;\rd X<\infty\} 
   \\
   L^2_{\rw^*}(\cV) &= \{G\in L^2(\cV),\quad \int_{\cV} {|G(X)|^2}{(1+|X|)^{2}}\;\rd X<\infty\}.
\end{aligned}
\end{equation}
Then, we define $H^1_{\rw}(\cV)$ as the space with norm and semi-norm given by
\begin{equation}
\label{eq:Hw0}
   \DNorm{G}{H^1_{\rw}(\cV)}:=\left( \DNormc{G}{L^2_{\rw}(\cV)} + \DNormc{\nabla G}{L^2(\cV)}
   \right)^{\frac{1}{2}}
   \quad\mbox{and}\quad
   \Norm{G}{H^1_{\rw}(\cV)} := \DNorm{\nabla G}{L^2(\cV)},
\end{equation}
and we define $H^1_{\rw,0}(\cV)$ as the completion of $C^\infty_0(\cV)$  with respect to the norm of $H^1_{\rw}(\cV)$.

The following result will be proved in Section~\ref{SS:H1cone} using the Poincar\'e inequality \eqref{E:PoincAr1r2}.

\begin{lemma}
\label{lem:cone-poinc}
Let $\Gamma$ be a cone satisfying the capacity condition \eqref{capacity>0}, and let $\cV$ be a domain coinciding with $\Gamma$ outside a bounded set. Then the semi-norm and norm of $H^1_{\rw}(\cV)$ are equivalent over $H^1_{\rw,0}(\cV)$. 
\end{lemma}

By the Lax-Milgram theorem for the variational formulation of the Dirichlet problem we see that the Laplace operator induces an isomorphism from $H^1_{\rw,0}(\cV)$ onto its dual. Since this dual contains the weighted space $L^2_{\rw^*}(\cV)$, as a particular case we deduce the following lemma for the pattern domain $\rP$:

\begin{lemma}
\label{lem:DirP}
For any $F$ such that $(1+|X|)F$ belongs to $L^2(\rP)$, the Dirichlet problem
\begin{equation}
\label{eq:poissonF}
\begin{cases}
\begin{array}{ll}
   U \in  H^1_{\rw,0}(\rP)\,,\\
   \Delta U = F\quad \text{ in }\;\rP\,, 
\end{array}
\end{cases}
\end{equation}
is uniquely solvable.
\end{lemma}

\subsection{Families of Dirichlet problems}
\label{ss:famdir}
With the domains $\Omega_\varepsilon$ as in Definition \ref{def:Omeps},
we consider the following family of $\varepsilon$-dependent Dirichlet problems:
\begin{equation}
\label{eq:poiseps}
\begin{cases}
\begin{array}{ll}
   \ueps \in  H^1_0(\Omega_\varepsilon)\,,\\
   \Delta \ueps = f_\varepsilon \quad \text{ in }\Omega_\varepsilon\,. 
\end{array}
\end{cases}
\end{equation}
The aim of this paper  
is to find, for small values of $\varepsilon$, representations of the solutions $u_\varepsilon$ as a convergent series in powers of $\varepsilon$, under the assumption that the right-hand sides $f_\varepsilon$ have a uniform structure.

The first level of assumption on $f_\varepsilon$ is
\begin{equation}
\label{eq:feps0}
   \exists f\in L^2(\R^n) \;\;\mbox{with}\;\; f\on{\sB(r_0)}\equiv0
   \qquad\mbox{such that}\qquad
   f_\varepsilon = f\on{\Omega_\varepsilon}\,.
\end{equation}
In fact, our method of FAA with cross-cutoff is perfectly adapted to this assumption, allowing in a very natural way for a more general right-hand side. This is the second level of assumption, for which $f_\varepsilon$ also includes some profile function $F\in L^2(\R^n)$ such that $F\equiv0$ on $\sB^\complement(R_0)$:
\begin{equation}
\label{eq:fFeps}
   f_\varepsilon(x) = f(x) + \varepsilon^{-2} F\left(\frac{x}{\varepsilon}\right),\quad 
   \forall x\in\Omega_\varepsilon \,.
\end{equation}
The third level of assumption would be to consider the case where the slow variable function $f$, instead of being $0$, is a real analytic function in the whole ball $\sB(r_0)$. This will be treated in a forthcoming paper.

\section{The Functional Analytic Approach with cross-cutoff}\label{s:FAAtrunc}
In our way to find representations of the solutions $\ueps$ of problem \eqref{eq:poiseps}, with a right-hand side satisfying \eqref{eq:feps0} or, more generally, \eqref{eq:fFeps}, we are going to translate \eqref{eq:poiseps} into a $2\times2$ system of equations in slow and rapid variables.

\subsection{Dedicated function spaces}\label{ss:prel}
Recall that $x$ and $X$ denote the slow and rapid variables, respectively.
To design our FAA with cross-cutoff we need to introduce some auxiliary functions, operators, and spaces. First of all we choose two smooth cutoff functions in relation with radii $r_0$ and $R_0$ appearing in the definition of the family of domains $\{\Omega_\varepsilon\}$:
\begin{itemize}
\item A cutoff function $\Phi:X\mapsto\Phi(X)$ that localizes in rapid variables at infinity:
\begin{equation}
\label{eq:Phi}
\Phi = 1 \quad \mbox{on }  \sB^\complement(2R_0)\, , \qquad 
\Phi = 0 \quad \mbox{on }  \sB(R_0)\, ,
\end{equation}
\item A cutoff function $\varphi:x\mapsto\varphi(x)$ that localizes in slow variables near $0$:
\begin{equation}
\label{eq:phi}
\varphi = 1 \quad \mbox{on }  \sB(r_0/2)\, , \qquad 
\varphi = 0 \quad \mbox{on }  \sB^\complement(r_0)\, .
\end{equation}
\end{itemize}

Then we denote by $\mathcal{H}_\varepsilon$ the change of variables 
\[
U\mapsto \left(\ x\mapsto U\left(\frac{x}{\varepsilon}\right)\right)
\]
(we shall write $\mathcal{H}_\varepsilon U(x)$ for $U(x/\varepsilon)$), so that its inverse is given by
\[
\mathcal{H}_\varepsilon^{-1}= 
\mathcal{H}_{1/\varepsilon}\colon u\mapsto (\ X\mapsto u(\varepsilon X)\ )
\]
(and, accordingly, we write $\mathcal{H}_{1/\varepsilon}u(X)$ for $u(\varepsilon X)$).

To deal with right-hand sides in $\Omega$ that vanish in a neighborhood of the vertex and with right-hand sides in $\rP$ that vanish outside a ball, we define the function spaces
\[
\begin{aligned}
&{\gF_\Omega} := \{f \in L^2(\Omega) ,\quad f\on{\sB(r_0/2)}= 0\}\, ,\\
&{\gF_\rP} := \{F \in L^2(\rP) ,\quad F\on{\sB^\complement(2R_0)}= 0\}\, .
\end{aligned}
\]
On $\gF_\Omega$ and $\gF_\rP$ we consider the norms induced by $L^2(\Omega)$ and $L^2(\rP)$, respectively.
Then, we introduce two function spaces $\gE_\Omega$ and $\gE_\rP$ for the solutions. The first one is defined by
\[
{\gE_\Omega}:= \{u \in H^1_0(\Omega) ,\quad \Delta u \in \gF_\Omega\}
\]
and we equip it with the graph norm
\[
\|u\|_{\gE_\Omega}^2:=\|u\|_{H^1_0(\Omega)}^2+\|\Delta u\|_{L^2(\Omega)}^2\,.
\]
To define  ${\gE_\rP}$ we use the space $H^1_{\rw,0}(\rP)$ introduced above (Section \ref{ss:func}) and set
\[
\gE_\rP := \{U \in H^1_{\rw,0}(\rP) ,\quad  \Delta U \in \gF_\rP\}
\]
with
\[
\|U\|_{\gE_\rP}^2:=\|U\|_{H^1_{\rw,0}(\rP)}^2+\|\Delta U\|_{L^2(\Omega)}^2\,.
\]

\begin{lemma}
\label{lem:isoDxDX}
The Laplace operator defines isomorphisms between the above spaces, {\it i.e.},
\begin{subequations}
\begin{align}
\label{isoDx}
&\Delta_x \colon  \gE_\Omega \to \gF_\Omega \ \mbox{is an isomorphism}\, ,\\
\label{isoDX}
&\Delta_X \colon  \gE_\rP \to \gF_\rP\  \mbox{is an isomorphism}\, .
\end{align}
\end{subequations}
\end{lemma}

\begin{proof}
As the space $\gF_\Omega$ embeds canonically in $H^{-1}(\Omega)$, \eqref{isoDx} is an obvious consequence of Lemma \ref{lem:DirOm}. Concerning \eqref{isoDX}, the argument relies on Lemma \ref{lem:DirP} if we notice that the space $\gF_\rP$ embeds in the dual of the variational space $H^1_{\rw,0}(\rP)$ (even if now the norm of the embedding depends on $R_0$).
\end{proof}

\subsection{The operator matrix $\cM[\varepsilon]$}
The equivalence between the Dirichlet problem \eqref{eq:poiseps} and a $2\times2$ system of equations on $\Omega$ or $\rP$ will be proved for right-hand sides $f_\varepsilon$ satisfying the second level of assumption \eqref{eq:fFeps}. So, let  $f\in \gF_\Omega$ and $F\in \gF_\rP$ be given. Both functions can be extended by $0$, thus defining elements of $L^2(\R^n)$ and giving sense to the equality $f_\varepsilon = f + \varepsilon^{-2}\mathcal{H}_\varepsilon F$ in $\Omega_\varepsilon$. The idea underlying the construction of the $2\times2$ system is to consider a similar ansatz for the solution and the right-hand sides. 

Let $\varepsilon \in (0,\varepsilon_0/4]$. We note that, due to the specific supports of $\Phi$ and $\varphi$, we have the equalities $f = (\mathcal{H}_\varepsilon \Phi)\, f$ and $\mathcal{H}_\varepsilon F = \varphi\,(\mathcal{H}_\varepsilon F)$. Hence, we find it convenient to write our family of Dirichlet problems as
\begin{equation}
\label{PoissonfF}
\begin{cases}
   \ueps \in  H^1_0(\Omega_\varepsilon)\, ,\\
   \Delta \ueps = (\mathcal{H}_\varepsilon \Phi)\, f + 
   \varepsilon^{-2}\varphi \,\mathcal{H}_\varepsilon F 
   \ \mbox{ in }\ \Omega_\varepsilon\, .
\end{cases}
\end{equation}
Then, for any {\em fixed} $\varepsilon$,  
we look for a representation of the solution $\ueps$ of \eqref{PoissonfF} in a form that resembles the right-hand side. Specifically, we write
\begin{equation}
\label{eq:ansatz}
   \ueps = (\mathcal{H}_\varepsilon \Phi)\,u+\varphi \,\mathcal{H}_\varepsilon U
\end{equation}
(cf.~\eqref{E:cross}), where  the functions $u:x\mapsto u(x)$ and $U:X\mapsto U(X)$ belong to  $\gE_\Omega$ and $\gE_\rP$, respectively. These functions $u$ and $U$ depend on $\varepsilon$ and will be eventually described  in terms of convergent generalized power series of $\varepsilon$ denoted by $\sfu[\varepsilon]$ and $\sfU[\varepsilon]$. For the moment being, we do not need to track the dependence in $\varepsilon$ and just write $u$ and $U$.

The next theorem relies on the calculation of the Laplace operator acting on $(\mathcal{H}_\varepsilon \Phi)\,u+\varphi \,\mathcal{H}_\varepsilon U$. Hence the commutator of $\Delta$ with cutoffs appears naturally:
\begin{notation}
\label{N:comm}
For a smooth function $\psi$, let the commutator $[\Delta,\psi]$ be defined as
\[
   [\Delta,\psi]w = \Delta(\psi w) - \psi\Delta w = 
   2\nabla\psi \cdot \nabla w + (\Delta\psi) w.
\]
\end{notation}

\begin{theorem}
\label{th:Meps}
Let $\varepsilon\in(0,\varepsilon_0/4]$ be chosen. 
Let $\cM[\varepsilon]$ be defined as the $2\times 2$ block operator matrix 
\begin{equation}\label{E:Me}
 \cM[\varepsilon]:=\
 \begin{pmatrix}
 \cM_{\Omega,\Omega}&\cM_{\Omega,\rP}[\varepsilon]\\
 \cM_{\rP,\Omega}[\varepsilon]&\cM_{\rP,\rP}
 \end{pmatrix}
\end{equation}
from $\gE_\Omega\times \gE_\rP$ to  $\gF_\Omega\times \gF_\rP$, where the entries $\cM_{\Omega,\Omega}$, $ \cM_{\Omega,\rP}$,  $\cM_{\rP,\Omega}$, and $\cM_{\rP,\rP}$ are defined by
\begin{equation}
\label{Meij}
\begin{aligned}
& \cM_{\Omega,\Omega}\ \colon &\gE_\Omega \to \gF_\Omega\, ,\quad 
   & u\longmapsto  \Delta_x u\\
& \cM_{\rP,\rP}\ \colon &\gE_\rP \to \gF_\rP\, , \quad
   & U\longmapsto  \Delta_X U\\
& \cM_{\Omega,\rP}[\varepsilon]\ \colon &\gE_\rP \to \gF_\Omega\, , \quad
   & U\longmapsto [\Delta_x,\varphi]  (\mathcal{H}_\varepsilon U) \\
& \cM_{\rP,\Omega}[\varepsilon]\ \colon &\gE_\Omega \to \gF_\rP\, , \quad
   & u\longmapsto [\Delta_X,\Phi] (\mathcal{H}_{1/\varepsilon}u)\,.
\end{aligned}   
\end{equation}
Let $f$ belong to $\gF_\Omega$ and $F$ to $\gF_\rP$. 
If we have 
\begin{equation}\label{solution.eq1}
{\cM}[\varepsilon]\left(\begin{array}{l}u\\ U \end{array}\right)=\left(\begin{array}{l}f\\F\end{array}\right)
\end{equation}
for a pair of functions $(u,U)\in \gE_\Omega\times\gE_\rP$, then the function
\[
   w_\varepsilon := (\mathcal{H}_\varepsilon \Phi) u+ \varphi (\mathcal{H}_\varepsilon U)
\] 
is a solution of the boundary value problem \eqref{PoissonfF} and, therefore, coincides with the unique solution  $\ueps$ of $\eqref{PoissonfF}$.
\end{theorem}

\begin{proof}
We have first to check that the constitutive operators of $\cM[\varepsilon]$ are well defined. Concerning the diagonal entries, this is obvious. Concerning non-diagonal entries, this is an easy consequence of the following facts
\begin{itemize}
\item The supports of $x\mapsto\nabla_x \varphi$ and $x\mapsto\Delta_x \varphi$ are contained in the annulus $\sA(\frac{r_0}{2},r_0)$, and the supports of $X\mapsto\nabla_X \Phi$ and $X\mapsto\Delta_X \Phi$ are contained in $\sA(R_0,2R_0)$,
\item The operators $\nabla_x$ and $\nabla_X$ are bounded from $H^1$ to $L^2$.
\end{itemize}
Then, we have
\begin{equation}
\label{eq:Dwe}
   \Delta w_\varepsilon = g_{\varepsilon,1} + g_{\varepsilon,2}
   \quad\mbox{with}\quad
   g_{\varepsilon,1} = \Delta_x \{ (\mathcal{H}_\varepsilon\Phi)\,u\}
   \quad\mbox{and}\quad
   g_{\varepsilon,2} = \Delta_x \{ \varphi (\mathcal{H}_\varepsilon U)\}\,,
\end{equation}
and we compute
\begin{align*}
   g_{\varepsilon,1} &= (\mathcal{H}_\varepsilon \Phi)\Delta_x u 
   + 2\nabla_x(\mathcal{H}_\varepsilon \Phi)\cdot \nabla_x u
   + (\Delta_x(\mathcal{H}_\varepsilon \Phi))u \\
   g_{\varepsilon,2} &= \varphi \Delta_x(\mathcal{H}_\varepsilon U)
   + 2 \nabla_x \varphi \cdot \nabla_x (\mathcal{H}_\varepsilon U)
   + ({\Delta_x\varphi}) \mathcal{H}_\varepsilon U\,.
\end{align*}
Since $u$ is in $\gE_\Omega$, by definition $\Delta_x u$ belongs to $\gF_\Omega$, so it is $\equiv0$ on $\sB(\frac{r_0}{2})$. As $\mathcal{H}_\varepsilon \Phi$ is $\equiv1$ on $\sB^\complement(2\varepsilon R_0)$ and $2\varepsilon R_0<\frac{r_0}{2}$, we find that the first term of $g_{\varepsilon,1}$ satisfies
\[
   (\mathcal{H}_\varepsilon \Phi)\Delta_x u = \Delta_x u.
\]
Concerning the first term of $g_{\varepsilon,2}$, we start with the identity
\[
   \varphi \Delta_x(\mathcal{H}_\varepsilon U) = 
   \varepsilon^{-2}\mathcal{H}_\varepsilon 
   \big\{ (\mathcal{H}_{1/\varepsilon}\varphi) \Delta_X U\big\}
\]
and,   since $U$ is in $\gE_\rP$, we can deduce by a similar reasoning as above that $(\mathcal{H}_{1/\varepsilon}\varphi)$ is $\equiv1$ on the support of $\Delta_X U$, ending up with
\[
   \varphi \Delta_x(\mathcal{H}_\varepsilon U) = 
   \varepsilon^{-2} \mathcal{H}_\varepsilon \big\{ \Delta_X U\big\}\,,
\]
hence
\[
   g_{\varepsilon,2} = \varepsilon^{-2} \mathcal{H}_\varepsilon \big\{ \Delta_X U\big\}
   + 2 \nabla_x \varphi \cdot \nabla_x (\mathcal{H}_\varepsilon U)
   + ({\Delta_x\varphi}) \mathcal{H}_\varepsilon U\,.
\]
Finally, noting that 
\[
   2\nabla_x(\mathcal{H}_\varepsilon \Phi)\cdot \nabla_x u
   + (\Delta_x(\mathcal{H}_\varepsilon \Phi))u =
   \varepsilon^{-2} \mathcal{H}_\varepsilon \Big\{ 
   2 \nabla_X\Phi \cdot \nabla_X (\mathcal{H}_{1/\varepsilon}u)
   +(\Delta_X \Phi) \mathcal{H}_{1/\varepsilon}u\Big\}
\]
we obtain
\[
   g_{\varepsilon,1} = \Delta_x u +
   \varepsilon^{-2} \mathcal{H}_\varepsilon \Big\{ 
   2 \nabla_X\Phi \cdot \nabla_X (\mathcal{H}_{1/\varepsilon}u)
   +(\Delta_X \Phi) \mathcal{H}_{1/\varepsilon}u\Big\}\,.
\]
Now, developing \eqref{solution.eq1} we find
\begin{align*}
   f&=\Delta_x u+ 2\nabla_x \varphi \cdot \nabla_x(\mathcal{H}_\varepsilon U)
   +(\Delta_x\varphi) \mathcal{H}_\varepsilon U\,,\\
   F&= \Delta_X U+2 \nabla_X\Phi \cdot \nabla_X (\mathcal{H}_{1/\varepsilon}u)
   +(\Delta_X \Phi) \mathcal{H}_{1/\varepsilon}u\,.
\end{align*}
So $g_{\varepsilon,1} + g_{\varepsilon,2}$ coincides with $f+\varepsilon^{-2}\mathcal{H}_\varepsilon F$, which in view of \eqref{eq:Dwe}, concludes the proof.
\end{proof}

\begin{remark}
1) From the proof above, the rationale for the construction of $\cM[\varepsilon]$ appears more clearly. The diagonal terms have a simple structure thanks to the conditions on the supports of $\Phi$ and $\varphi$, whereas the non-diagonal terms can be seen as transfer operators between slow and rapid variables: Here the cutoff dictates the choice of variables (slow when $\varphi$ is present, rapid when $\Phi$ is present). 

2) The solvability of equation \eqref{solution.eq1} will be proven later on showing that ${\cM}[\varepsilon]$ is invertible for $\varepsilon\ge 0$ small enough (cf.~Theorem \ref{M-1}).
\end{remark}

\section{The spectral expansion of transfer operators}\label{s:specdev}

The FAA strategy boils down to understanding in which sense the matrix $\cM[\varepsilon]$ can be seen as an ``analytic'' function of $\varepsilon$ as $\varepsilon$ tends to $0$. In this section we show how this is related to the spectral expansion at $0$ or at infinity of harmonic functions satisfying Dirichlet conditions in the cone $\Gamma$. Consequently we will have to extend the notion of analytic structure to generalized convergent series with non-integer exponents in $\varepsilon$.

\subsection{Energy spaces on cones}\label{SS:H1cone}

We first have to collect some basic results about Sobolev spaces and the Laplace operator in our conical domains, which, due to the very weak regularity assumptions, cannot just be quoted from the standard literature on domains with conical singularities \cite{Kondrat67,KozlovMazyaRossmann97b,KozlovMazyaRossmann01}. Recall that the spherical domain 
$\hat\Gamma=\Gamma\cap\bS^{n-1}$ is an arbitrary open subset of the unit sphere that satisfies the single condition that the capacity of its complement is non-zero. Otherwise, it can have cracks, cusps, infinitely many holes or connected components, and so on. Its boundary can be very wild. Therefore the truncated cone
$\Gamma\cap\sB(\rho)$ has a boundary with two very different parts: 
The lateral boundary $(\partial\Gamma)\cap\overline{\sB(\rho)}$, where we want to impose Dirichlet conditions, does not satisfy any of the usual conditions required for defining boundary traces in Sobolev spaces, but the remaining part is the spherical cap $\rho\hat\Gamma$ that is part of a smooth manifold, and the boundary traces on this part are well defined and will play an important role.

We begin by introducing some notation for spaces where the zero boundary condition is imposed only on parts of the boundary.

\begin{notation}
\label{not:H10barU}
Let $\rho>0$. We define
\begin{itemize}
\item[\em (i)] $C^\infty_0(\Gamma\cap\overline{\sB(\rho)})$ as the space of restrictions to $\Gamma\cap\overline{\sB(\rho)}$ of smooth functions with support in $\Gamma$, and define similarly $C^\infty_0(\Gamma\cap\overline{\sB^\complement(\rho)})$,
\item[\em (ii)] $H^1_0(\Gamma\cap\overline{\sB(\rho)})$ as the closure of $C^\infty_0(\Gamma\cap\overline{\sB(\rho)})$ in $H^1(\Gamma\cap\sB(\rho))$,
\item[\em (iii)] $H^1_{\rw,0}(\Gamma\cap\overline{\sB^\complement(\rho)})$ as the closure of $C^\infty_0(\Gamma\cap\overline{\sB^\complement(\rho)})$ in $H^1_{\rw}(\Gamma\cap\sB^\complement(\rho))$.
\end{itemize}
\end{notation}

For $x\in\R^{n}\setminus\{0\}$, we use polar coordinates $r=|x|$ and 
$\vartheta=\hat{x}=\frac{x}{|x|}$, where we write indifferently $\hat{x}$ for a unit vector in $\R^{n}$ and $\vartheta$ for the same considered as a point of the manifold $\bS^{n-1}$. The gradient $\nabla=\nabla_{x}$ splits into a radial and tangential part $\nabla=\hat{x}\partial_{r}+\nabla_{T}$, which by the orthogonality relation $\hat{x}\cdot\nabla_{T}=0$, gives the formulas
\begin{equation}
\label{E:gradtan}
 \nabla u = \partial_{r}u\,\hat{x} +\nabla_{T} u\,;\quad
 \partial_{r}u = \hat{x}\cdot\nabla u\,;\quad
 \nabla_{T}u = \nabla u - (\hat{x}\cdot\nabla u)\hat{x}\,.
\end{equation}
If we write  $\tilde{u}(r,\vartheta)=u(x)$  (the tilde may be omitted later on), we have 
$$
  \nabla_{T}u(x)=\frac1{r}\DelS\tilde{u}(r,\vartheta)\,,\quad
  \mbox{where $\DelS$ is the tangential gradient  on the unit sphere}.
$$
Pythagoras' theorem implies that
$$
 |\nabla u|^{2} = |\partial_{r}u|^{2} +\frac1{r^{2}}|\DelS u|^{2}\,,
$$
and integrating over an annular domain $\sA(r_{1}, r_{2})$ results in
\begin{equation}
\label{E:H1annulus}
  \Normc{u}{H^{1}(\Gamma\cap\sA(r_{1}, r_{2}))}
  = \int_{\sA(r_{1}, r_{2})} \!|\nabla u|^{2}\,\rd x
  = \int_{r_{1}}^{r_{2}} \!\int_{\hat\Gamma}
  \Big( |\partial_{r}u|^{2} +\frac1{r^{2}}|\DelS u|^{2}\Big)\,\rd\sigma
  \,r^{n-1}\,\rd r
\end{equation}
for any $u\in C_{0}^{\infty}(\Gamma)$. 
We can choose $r_{1}=0$ and get for any $\rho>0$
\begin{equation}
\label{E:H1finitecone}
  \Normc{u}{H^{1}(\Gamma\cap\sB(\rho))}
  = \int_{0}^{\rho} \int_{\hat\Gamma}
  \Big( r^{n-1}|\partial_{r}u|^{2} +r^{n-3}|\DelS u|^{2}\Big)\,\rd\sigma\,\rd r\,,
\end{equation} 
which is, by density, valid for all $u\in H^{1}_{0}(\Gamma\cap\overline{\sB(\rho)})$. 
Likewise, we can choose $r_{2}=\infty$ and get 
\begin{equation}
\label{E:H1complement}
  \Normc{u}{H^{1}(\Gamma\cap\sB^{\complement}(\rho))}
  = \int_{\rho}^{\infty} \int_{\hat\Gamma}
  \Big( r^{n-1}|\partial_{r}u|^{2} +r^{n-3}|\DelS u|^{2}\Big)\,\rd\sigma\,\rd r
\end{equation} 
for all 
$u\in H^{1}_{0}(\Gamma\cap\overline{\sB^{\complement}(\rho)})$.

\smallskip\noindent
Now we can employ the Poincar\'e inequality \eqref{E:poincare} to analyze the spherical traces on $|x|=r$.
\begin{lemma}
 \label{L:L2trace}
{\em (i)} \  For any $r>0$ and 
 $v\in C^\infty_0(\Gamma)$ define $\tr_r v$ as the function defined on $\hat\Gamma$ by
 \[
   \tr_r v(\hat x) = v(r\hat x),\quad \hat x\in\hat\Gamma.
\]
Let $\rho>0$. Then for $r\in(0,\rho]$, the mapping $\tr_{r}$ has an extension to a bounded operator
$$
  \tr_{r} : H^{1}_{0}(\Gamma\cap\overline{\sB(\rho)}) \to L^{2}(\hat\Gamma)
$$
satisfying an estimate
\begin{equation}
\label{E:L2trace}
\DNorm{\tr_{r}u}{L^{2}(\hat\Gamma)} \le 
C_{\hat\Gamma}\,r^{1-\frac n2}\, \Norm{u}{H^{1}(\Gamma\cap\sB(\rho))}\,.
\end{equation}
Similarly, for $r\in[\rho,\infty)$, the mapping $\tr_{r}$ has an extension to a bounded operator
$$
  \tr_{r} : H^{1}_{\rw,0}(\Gamma\cap\overline{\sB^{\complement}(\rho)}) \to L^{2}(\hat\Gamma)
$$
satisfying an estimate
\begin{equation}
\label{E:L2trcompl}
\DNorm{\tr_{r}u}{L^{2}(\hat\Gamma)} \le 
C_{\hat\Gamma}\,r^{1-\frac n2}\, \Norm{u}{H^{1}(\Gamma\cap\sB^{\complement}(\rho))}\,.
\end{equation}
Here the constant $C_{\hat\Gamma}$ depends only on the dimension $n$ and the first Dirichlet eigenvalue $\mu_{1}$ of the set $\hat\Gamma$, cf \eqref{mu1>0}.

\medskip\noindent
{\em (ii)} \ Let $u\in H^{1}_{0}(\Gamma\cap\overline{\sB(\rho)})$. 
Then the function $r\mapsto \tilde u(r)=\tr_{r}u$ belongs to the vector valued weighted $H^{1}$-Sobolev space $\mathcal{H}^1((0,\rho);L^2(\hat\Gamma))$ with norm given by 
$$
  \DNormc{\tilde u}{\mathcal{H}^1((0,\rho);L^2(\hat\Gamma))}:=
  \int_{0}^{\rho}
  \big(\DNormc{\tilde u(r)}{L^{2}(\hat\Gamma)} +
  \DNormc{\partial_{r}\tilde u(r)}{L^{2}(\hat\Gamma)}\big)\, r^{n-1}\,\rd r\,.
$$
For almost every $r\in(0,\rho]$, $\tr_{r}u$ is in $H^{1}_{0}(\hat\Gamma)$, and $\tilde u$ belongs to the vector-valued weighted $L^{2}$ space
$\mathcal{L}^{2}((0,\rho);H^{1}_{0}(\hat\Gamma))$ with norm given by
$$
  \DNormc{\tilde u}{\mathcal{L}^{2}((0,\rho);H^{1}_{0}(\hat\Gamma))} :=
  \int_{0}^{\rho} \Normc{\tilde u(r)}{H^{1}(\hat\Gamma)}\,r^{n-3}\,\rd r.
$$
Likewise, for the unbounded truncated cone $\Gamma\cap\sB^{\complement}(\rho)$ and
$u\in H^{1}_{\rw,0}(\Gamma\cap\overline{\sB^{\complement}(\rho)})$, the analogous norms
$$
   \int_{\rho}^{\infty}
  \DNormc{\frac{\tilde u(r)}{1+r}}{L^{2}(\hat\Gamma)}\, r^{n-1}\,\rd r\,,\;
   \int_{\rho}^{\infty}  
  \DNormc{\partial_{r}\tilde u(r)}{L^{2}(\hat\Gamma)}\, r^{n-1}\,\rd r\;
  \mbox{ and }\;
   \int_{\rho}^{\infty}  
   \Normc{\tilde u(r)}{H^{1}(\hat\Gamma)}\,r^{n-3}\,\rd r
$$
are finite.
\end{lemma}

\begin{proof}
{\em (i)} We first show that the bound \eqref{E:L2trace} holds for $u\in C_{0}^{\infty}(\Gamma)$ and for $r=\rho$ (from which it follows for $r\le\rho)$: We have with $\tilde u(r)=\tr_{r}u$
$$
\begin{aligned}
 \rho^{n-2}|\tilde u(\rho)|^{2} &= 
 \int_{0}^{\rho} \partial_{r}\big(r^{n-2}|\tilde u(r)|^{2}\big) \,\rd r \\ &=
 (n-2)\int_{0}^{\rho} r^{n-3}|\tilde u(r)|^{2}\,\rd r +
 2\int_{0}^{\rho}r^{\frac{n-3}2}\tilde u(r) \, r^{\frac{n-1}2}\partial_{r}\tilde u(r)\,\rd r\,.
\end{aligned}
$$
With the Young inequality, this gives for any $\eta>0$
$$
 \rho^{n-2}\DNormc{\tilde u(\rho)}{L^{2}(\hat\Gamma)} \le
 (n-2+\eta) \int_{0}^{\rho} r^{n-3}\DNormc{\tilde u(r)}{L^{2}(\hat\Gamma)}\,\rd y +
    \tfrac1\eta \int_{0}^{\rho} r^{n-1}\DNormc{\partial_{r}\tilde u(r)}{L^{2}(\hat\Gamma)}\,\rd r\,.
$$
At this point, we use the Poincar\'e inequality \eqref{E:poincare}
$\DNormc{\tilde u(r)}{L^{2}(\hat\Gamma)} \le 
  \frac1{\mu_{1}}\Normc{\tilde u(r)}{H^{1}(\hat\Gamma)}$
and compare with \eqref{E:H1finitecone} to obtain (choosing $\eta=1$)
$$
\begin{aligned}
 \DNormc{\tilde u(\rho)}{L^{2}(\hat\Gamma)} &\le
 \rho^{2-n}\int_{0}^{\rho} \big\{
   r^{n-1} \DNormc{\partial_{r}\tilde u(r)}{L^{2}(\hat\Gamma)} +
   \tfrac{n-1}{\mu_{1}} \,r^{n-3}\Normc{\tilde u(r)}{H^{1}(\hat\Gamma)} \big\}\,\rd r\\
   &\le C_{\hat\Gamma}^{2}\, \rho^{2-n}\Normc{u}{H^{1}(\Gamma\cap\sB(\rho))}\;
   \mbox{ with }\;
   C_{\hat\Gamma}=\max\{1,\sqrt{\tfrac{n-1}{\mu_{1}}}\}\,.
\end{aligned}
$$
Then, the estimate in \eqref{E:L2trace} extends by continuity to all of $H^{1}_{0}(\Gamma\cap\overline{\sB(\rho)})$. The analogous estimate for the unbounded truncated cone is obtained by integrating from $\rho$ to $\infty$. 

\smallskip\noindent
{\em (ii)\ } The proof is then easily completed: The identity \eqref{E:H1finitecone} for the $H^{1}$ seminorm can be written for the $H^{1}$ norm with the above definitions as
\begin{equation}
\label{E:H1Normfcon}
\begin{aligned}
  \DNormc{u}{H^{1}(\Gamma\cap\sB(\rho))}
  &= \int_{0}^{\rho} \int_{\hat\Gamma}
  \Big( r^{n-1}|\tilde u|^{2} + r^{n-1}|\partial_{r}\tilde u|^{2} +r^{n-3}|\DelS \tilde u|^{2}\Big)\,\rd\sigma\,\rd r\\
  &=  \DNormc{\tilde u}{\mathcal{H}^1((0,\rho);L^2(\hat\Gamma))} + 
    \DNormc{\tilde u}{\mathcal{L}^{2}((0,\rho);H^{1}_{0}(\hat\Gamma))}\,.
  \end{aligned}
\end{equation} 
The fact that $\tilde u(r)\in H^{1}_{0}(\hat\Gamma)$ for almost every $r\in(0,\rho]$ follows from Fubini's theorem for the latter integral.
\end{proof}

We can now prove the Poincar\'e inequalities that imply the proof of Lemma~\ref{lem:cone-poinc}. We choose as a domain the truncated cone $\Gamma\cap\sA(r_{1},r_{2})$, where
$0\le r_{1}<r_{2}\le+\infty$. The truncated cones $\Gamma\cap\sB(r_{2})$ and 
$\Gamma\cap \sB^\complement(r_1)$ are the special cases $r_{1}=0$ and $r_{2}=+\infty$, respectively.
From the Poincar\'e inequality on $\hat\Gamma$, we obtain by integrating with the weight $r^{n-3}$:
\begin{equation}
\label{E:PoincAr1r2}
\begin{aligned}
 \DNormc{\tfrac ur}{L^{2}(\Gamma\cap\sA(r_{1},r_{2}))} &= 
 \int_{r_{1}}^{r_{2}}r^{n-3}\DNormc{\tilde u(r)}{L^{2}(\hat\Gamma)}\rd r \le
 \frac1{\mu_{1}} \int_{r_{1}}^{r_{2}}r^{n-3}\DNormc{\DelS\tilde u(r)}{L^{2}(\hat\Gamma)}\rd r\\
 &\le \frac1{\mu_{1}} \Normc{u}{H^{1}(\Gamma\cap\sA(r_{1},r_{2}))}
\end{aligned}
\end{equation}
(see also \eqref{E:H1annulus}).
In the bounded case $r_{2}<\infty$, we can use the trivial estimate
$$
  \DNormc{u}{L^{2}(\Gamma\cap\sB(r_{2}))} 
   \le r_{2}^{2}\DNormc{\tfrac ur}{L^{2}(\Gamma\cap\sB(r_{2}))}
$$
to get rid of the weight in the $L^{2}$ norm, but in the unbounded case for $\Gamma\cap \sB^\complement(r_1)$ we need to keep it. In the latter case, the weight $\frac1r$ is equivalent to the weight $\frac1{1+r}$, which is the one used in Lemma~\ref{lem:cone-poinc}.

\subsection{Laplace-Beltrami eigenfunction expansion}
The bilinear form $\big\langle\DelS u, \DelS v\big\rangle_{\hat\Gamma}$ is, on one hand, the scalar product in the Hilbert space $H^{1}_{0}(\hat\Gamma)$ encountered in the previous section. On the other hand, it is the energy form associated with the Laplace-Beltrami operator  on the sphere $\bS^{n-1}$:
$$
  \big\langle\DelS u, \DelS v\big\rangle_{\hat\Gamma} = 
  \big\langle\gL^{\dir}_{\hat\Gamma} u, v\big\rangle_{\hat\Gamma}.
$$
This relation defines the variational form of the Laplace-Beltrami operator on $\hat\Gamma$ with Dirichlet conditions as a bounded operator 
$$
 \gL^\dir_{\hat\Gamma} :  H^1_0(\hat\Gamma)  \longrightarrow  H^{-1}(\hat\Gamma)\,.
$$
It can also be extended to an unbounded self-adjoint operator on $L^2(\hat\Gamma)$ with domain
\[
   \Dom (\gL^\dir_{\hat\Gamma}) =
   \{\hat u\in H^1_0(\hat\Gamma),\ \Delta_{\bS^{n-1}}\hat u \in L^2(\hat\Gamma)\}\,.
\]
Due to the compact embedding of $H^1_0(\hat\Gamma)$ in $L^{2}(\hat\Gamma)$, which is valid for any open subset $\hat\Gamma$ of $\bS^{n-1}$, it has a compact resolvent and therefore a discrete spectrum, and the eigenvalues can be obtained from the min-max principle applied to the quadratic energy form.

The main tool for the analysis of elements in our special spaces $\gE_\Omega$ and $\gE_\rP$ is the projection on an orthonormal basis of eigenfunctions of $\gL^\dir_{\hat\Gamma}$.
Let us therefore recall some well-known facts about such eigenfunction expansions and then apply them to our special situation of conical singularities.

The role of the assumption \eqref{capacity>0} is to ensure that $\gL^{\dir}_{\hat\Gamma}$ is positive definite, that is, its lowest eigenvalue $\mu_{1}$ is non-zero. 

\begin{notation}
Let
\[
   0<\mu_1\le\mu_2\le \ldots \mu_j \le \ldots
\]
be the non-decreasing sequence of the eigenvalues of $\gL^\dir_{\hat\Gamma}$ (with possible repetitions according to multiplicities) and let $\{\psi_j\}_{j=1}^\infty$ be an associated basis of eigenfunctions orthonormal in $L^2(\hat\Gamma)$, \ {\em i.e.\ } 
$\int_{\hat\Gamma}\psi_i\psi_j\,\rd\sigma=\delta_{i,j}$ for all $i,j\ge1$.
\end{notation}

The eigenfunctions $\psi_j$ satisfy, by definition
\[
   \forall v\in H^1_0(\hat\Gamma),\quad
   \big\langle \DelS \psi_j,\DelS v\big\rangle_{\hat\Gamma} =
   \int_{\hat\Gamma} \mu_j\psi_j\, v \;\rd\sigma\,.
\]
In particular, they are orthogonal in $H^1_0(\hat\Gamma)$:
\[
   \forall k\ge1,\quad
   \big\langle \DelS \psi_j,\DelS \psi_k \big\rangle_{\hat\Gamma} =
   \delta_{j,k} \,\mu_j\,.
\]

Any $u\in L^{2}(\hat\Gamma)$ has a convergent expansion 
\begin{equation}
\label{E:L1}
  u = \sum_{k\ge1} u_{j}\psi_{j}
  \quad\mbox{ with coefficients }
  u_{j} = \int_{\hat\Gamma} u\,\psi_j \;\rd\sigma
\end{equation}
with equality of norms
\begin{equation}
\label{eq:L2}
   \DNormc{u}{L^2(\hat\Gamma)} = \sum_{j\ge1} |u_j|^2\,.
\end{equation}
It is also well known that $u$ belongs to $H^{1}_{0}(\hat\Gamma)$ if and only if
$\sum_{j\ge1} \mu_{j}|u_j|^2<\infty$, and then the expansion \eqref{E:L1} converges in  
$H^{1}_{0}(\hat\Gamma)$ with equality of norms
\begin{equation}
\label{eq:L3}
   \Normc{u}{H^{1}(\hat\Gamma)} = \sum_{j\ge1} \mu_{j}|u_j|^2\,.
\end{equation}

We now apply this to the description of $H^{1}_{0}(\Gamma\cap\overline{\sB(\rho)})$ in polar coordinates obtained in the previous section, in particular the expression \eqref{E:H1Normfcon} for the norm. Let $u\in H^1_0(\Gamma\cap\overline{\sB(\rho)})$ and define the coefficient functions 
$r\mapsto\tilde u_j(r)$
\begin{equation}
\label{E:ujtilde}
   \tilde u_j(r) = \int_{\hat\Gamma}  \tr_{r}u(\vartheta)\,\psi_j(\vartheta) \;\rd\sigma_\vartheta \ = \int_{\hat\Gamma}  \tilde u(r)\,\psi_j \;\rd\sigma\,.
\end{equation}
From Lemma \ref{L:L2trace}, we immediately deduce that these coefficients belong to the scalar versions $\mathcal{L}^2(0,\rho):=\mathcal{L}^2((0,\rho);\mathbb{R})$ and $\mathcal{H}^1(0,\rho):=\mathcal{H}^1((0,\rho);\mathbb{R})$ of the weighted $L^{2}$ and $H^{1}$ spaces defined above. Moreover, we have the norm equalities
\begin{equation}
\label{eq:L2H1cal}
   \DNormc{\tilde u}{\mathcal{L}^2((0,\rho);H^1_0(\hat\Gamma))} =
   \sum_{j\ge1} \mu_j \DNormc{\tilde u_j}{\mathcal{L}^2(0,\rho)}
   \!\quad\mbox{and}\quad\!
   \DNormc{\tilde u}{\mathcal{H}^1((0,\rho);L^2(\hat\Gamma))} =
   \sum_{j\ge1} \DNormc{\tilde u_j}{\mathcal{H}^1(0,\rho)}\,.
\end{equation}
Hence, the following result:

\begin{proposition}
\label{P:Fourierpsij}
 For $\rho>0$ and $u\in H^{1}_{0}(\Gamma\cap\overline{\sB(\rho)})$, let the coefficients $\tilde u_j(r)$ be defined by \eqref{E:ujtilde}. Then the expansion
\begin{equation}
\label{eq:LBexp}
   u(x) = \sum_{j\ge1} \tilde u_j(r) \psi_j(\vartheta),\quad 
   x=r\vartheta\in \Gamma\cap{\sB(\rho)}
\end{equation}
converges in $H^{1}_{0}(\Gamma\cap\overline{\sB(\rho)})$ and
\begin{equation}
\label{E:FourierH1}
 \DNormc{u}{H^1(\Gamma\cap{\sB(\rho)})} =
   \sum_{j\ge1} \Big(\mu_j \DNormc{\tilde u_j}{\mathcal{L}^2(0,\rho)} +
   \DNormc{\tilde u_j}{\mathcal{H}^1(0,\rho)}\Big)\,.
\end{equation}
Here
$$
  \DNormc{\tilde u_{j}}{\mathcal{L}^{2}(0,\rho)} =
  \int_{0}^{\rho} |\tilde u_{j}(r)|^{2}\,r^{n-3}\,\rd r
  \:\:\mbox{ and }\:\:
  \DNormc{\tilde u_{j}}{\mathcal{H}^1(0,\rho)} =
  \int_{0}^{\rho}
  \big(|\tilde u_{j}(r)|^{2} +
  |\partial_{r}\tilde u_{j}(r)|^{2}\big)\, r^{n-1}\,\rd r\,.
$$
Similarly, for $u\in H^{1}_{\rw,0}(\Gamma\cap\overline{\sB^{\complement}(\rho)})$ the expansion 
\eqref{eq:LBexp} converges in $H^{1}_{\rw,0}(\Gamma\cap\overline{\sB^{\complement}(\rho)})$, and 
\begin{equation}
\label{E:FourierH1C}
 \Normc{u}{H^1(\Gamma\cap{\sB^{\complement}(\rho)})} =
   \sum_{j\ge1} \Big(\mu_j \int_{\rho}^{\infty} |\tilde u_{j}(r)|^{2}\,r^{n-3}\,\rd r
    + \int_{\rho}^{\infty}
  |\partial_{r}\tilde u_{j}(r)|^{2}\, r^{n-1}\,\rd r \Big)\,.
\end{equation}
\end{proposition}

In the proof leading to this proposition, the only thing that has not yet been mentioned is that the terms $\tilde u_j(r) \psi_j(\vartheta)$
in the expansion \eqref{eq:LBexp} actually belong to the space 
$H^{1}_{0}(\Gamma\cap\overline{\sB(\rho)})$. This is an easy consequence of the following observation, because for the approximation of $\tilde u_j(r) \psi_j(\vartheta)$ by $C_{0}^{\infty}$ functions, one can approximate both factors separately by $C_{0}^{\infty}$ functions. 

\begin{lemma}
\label{L:1DKondratev}
The subspace $C_{0}^{\infty}((0,\rho])$ is dense in $\mathcal{L}^{2}(0,\rho) \cap \mathcal{H}^1(0,\rho)$.
\end{lemma}
\begin{proof}
The norm in $\mathcal{L}^{2}(0,\rho) \cap \mathcal{H}^1(0,\rho)$ is equivalent to the Kondrat'ev-type homogeneous weighted $H^{1}$ norm $\DNorm{v}{K^{1}_{\alpha}(0,\rho)}$ with 
$\alpha=\frac{n-3}2$ defined by
$$
  \DNormc{v}{K^{1}_{\alpha}(0,\rho)} =
   \int_{0}^{\rho}\Big( r^{2\alpha}|v(r)|^{2} + r^{2\alpha+2}|v'(r)|^{2}\Big)\,\rd r\,.
$$
A short way to prove that $C_{0}^{\infty}((0,\rho])$ is dense in 
$K^{1}_{\alpha}(0,\rho)$ is the following:\\
Define $\hat v(t) = e^{-(\alpha+\frac12) t} v(\rho e^{-t})$. Then it is easy to verify that the mapping $v\mapsto\hat v$ is an isomorphism from $K^{1}_{\alpha}(0,\rho)$ onto the (unweighted) Sobolev space
$H^{1}(\R_{+})$, and $C_{0}^{\infty}(\overline{\R_{+}})$ is well known to be dense in $H^{1}(\R_{+})$.
\end{proof}

The expansion presented in Proposition \ref{P:Fourierpsij} can be used to further elucidate the structure of our energy space $H^{1}_{0}(\Gamma\cap\overline{\sB(\rho))}$. We mention two results that are valid under our general weak regularity assumptions. 

\begin{corollary}
\label{C:equaltensorprod}
For any $\rho>0$ there holds
 $$
  H^{1}_{0}(\Gamma\cap\overline{\sB(\rho)}) = 
   \mathcal{L}^{2}((0,\rho);H^{1}_{0}(\hat\Gamma)) \cap \mathcal{H}^1((0,\rho);L^2(\hat\Gamma))\,.
 $$
\end{corollary}
\begin{proof}
Lemma \ref{L:L2trace} and the equality of norms \eqref{E:H1Normfcon} imply that the left-hand space is contained as a closed subspace in the right-hand space. The reverse inclusion follows from the expansion \eqref{eq:LBexp} in Proposition~\ref{P:Fourierpsij}, because this expansion converges in the same way for all elements of the right-hand space.
\end{proof}

\begin{corollary}
 \label{C:trace0}
 Another consequence is the implication
 $$
  u\in H^{1}_{0}(\Gamma\cap\overline{\sB(\rho)}) \quad\&\quad \tr_{\rho}u=0
  \qquad\Longrightarrow\qquad
   u\in H^{1}_{0}(\Gamma\cap\sB(\rho))\,.
 $$
 \end{corollary}
\begin{proof}
Indeed, if $\tr_{\rho}u=0$, then for all $j$, $\tilde u_{j}(\rho)=0$. This implies that $\tilde u_{j}$ can be approximated in the weighted $K^{1}_{\alpha}$ norm by $C_{0}^{\infty}$ functions with support in $(0,\rho)$, and therefore every term in the expansion \eqref{eq:LBexp} is in $H^{1}_{0}(\Gamma\cap\sB(\rho))$.
\end{proof}

Analogous results hold on the unbounded truncated cone for $H^{1}_{\rw,0}(\Gamma\cap\overline{\sB^{\complement}(\rho)})$. They can be obtained without effort by using the Kelvin transformation, see Appendix \ref{S:Kelvin}.

\subsection{Expansion of harmonic functions}\label{SS:harm}

Let us assume that $u\in H^1_0(\Gamma\cap\overline{\sB(\rho)})$ is harmonic in 
$\Gamma\cap\sB(\rho)$. We will determine the form of the coefficients $\tilde u_{j}$ in the expansion \eqref{eq:LBexp} by the method of separation of variables in polar coordinates. The function $u$ satisfies $\Delta u=0$ in the distributional sense, so that
\[
   \forall v\in H^1_0(\Gamma\cap{\sB(\rho)}),\quad
   \big\langle \nabla u,\nabla v\big\rangle_{\Gamma\cap{\sB(\rho)}} = 0.
\]
Choose $j$. Taking $v(x)=\xi(r)\psi_j(\vartheta)$, we find that, for any $\xi\in C^\infty_0(0,\rho )$, we have
\[
   \int_0^\rho \int_{\hat \Gamma} \Big( \partial_r \tilde u(r,\vartheta)\, \xi'(r)\psi_j(\vartheta) 
   + \frac{1}{r^2} \DelS \tilde u(r,\vartheta) \cdot \xi(r)\DelS \psi_j \Big) 
   \rd \vartheta\; r^{n-1} \rd r = 0 \,.
\]
Combined with
$\Big\langle \DelS\tilde u(r,\cdot), \DelS\psi_{j}\Big\rangle_{\bS^{n-1}}= \mu_{j}\tilde u_{j}$,
this gives the variational form
\[
   \int_0^\rho  \Big( \tilde u'_j(r)\, \xi'(r) 
   + \frac{\mu_j}{r^2}\, \tilde u_j(r) \, \xi(r)\Big) 
   \; r^{n-1} \rd r = 0
\]
of the differential equation
\begin{equation}\label{eq:ode}
   -r^{1-n}\big(r^{n-1}\tilde u'_j(r)\big)' + \frac{\mu_j}{r^2}\, \tilde u_j(r) = 0 \,.
\end{equation}
The solutions have the form $r^\lambda$ with $\lambda$ such that
\begin{equation}
\label{eq:lam}
   -\lambda(\lambda+n-2) + \mu_j=0.
\end{equation}
Hence, there exist coefficients $c^+_j$ and $c^-_j$ such that
\begin{equation}
\label{E:ajpm}
   \tilde{u}_j(r) = c^+_j r^{\lambda^+_j} + c^-_j r^{\lambda^-_j}
\end{equation}
with 
\begin{equation}\label{lambdapm}
   \lambda_j^\pm := 1 - \frac{n}{2} \pm \sqrt{\left(1 - \frac{n}{2}\right)^2 + \mu_j}\,,
\end{equation}
the two roots of equation \eqref{eq:lam}.

We note that $\{\lambda^+_j\}$ is a non-decreasing sequence of positive numbers and $\lambda^-_j$ satisfies
\begin{equation}\label{lambda-lambda+}
   \lambda_j^- = 2 - n - \lambda_j^+.
\end{equation}
One readily checks that for $\lambda \in \mathbb{R}$, the function $r \mapsto r^{\lambda}$ belongs to our weighted space $\mathcal{L}^{2}(0,\rho) \cap \mathcal{H}^1(0,\rho)$ if and only if $\lambda > 1 - \frac{n}{2}$. Therefore, we have $c^-_j = 0$ in \eqref{E:ajpm}, and we conclude that 
\[
\tilde{u}_j(r) = c^+_j r^{\lambda^+_j}.
\]
Taking 
\begin{equation}\label{hjp}
   h^+_j(x) := r^{\lambda^+_j} \psi_j(\vartheta),
\end{equation}
we can write the expansion \eqref{eq:LBexp} as
\[
u = \sum_{j \ge 1} c^+_j h^+_j,
\]
where the sum converges in $H^1(\Gamma \cap \sB(\rho))$. We also note that for the coefficients $c^+_j$, we have the equality 
\begin{equation}\label{c+j}
c^+_j = r^{-\lambda^+_j} \tilde{u}_j(r) = r^{-\lambda^+_j} \int_{\hat{\Gamma}} \tr_{r} u(\vartheta) \, \psi_j(\vartheta) \, \rd\sigma_\vartheta.
\end{equation}

Similarly, for a function $U \in H^{1}_{\rw,0}(\Gamma \cap \overline{\sB^{\complement}(\rho)})$, we can find the expansion 
\begin{equation}\label{Uexp}
U = \sum_{j \ge 1} B^-_j h^-_j,
\end{equation}
with 
\begin{equation}\label{hjm}
   h^-_j(x) := r^{\lambda^-_j} \psi_j(\vartheta),
\end{equation}
and 
\begin{equation}\label{B-j}
B^-_j = r^{-\lambda^-_j} \int_{\hat{\Gamma}} \tr_{r} U(\vartheta) \, \psi_j(\vartheta) \, \rd\sigma_\vartheta.
\end{equation}

The series in \eqref{Uexp} converges in $H^{1}_{\text{w}}(\Gamma \cap \sB^{\complement}(\rho))$, and the expansion can be proved following the same steps as in the proof for the bounded domain. In particular,  we recover the same differential equation \eqref{eq:ode}, leading to the same solution \eqref{E:ajpm}. However, this time we must set the coefficient of $r^{\lambda^+_j}$  to zero, because $r^{\lambda^+_j}$ does not belong to the corresponding space $\mathcal{L}^{2}(\rho,+\infty) \cap \mathcal{H}^1(\rho,+\infty)$. Alternatively, a proof can be obtained using the Kelvin transform isomorphism between $H^{1}_{0}(\Gamma \cap \overline{\sB(\rho)})$ and $H^{1}_{\text{w},0}(\Gamma \cap \overline{\sB^{\complement}(\rho)})$, see Appendix \ref{S:Kelvin}.
\smallskip

For the functions $h^+_j$ and $h^-_j$, we have the following:

\begin{lemma} 
Let $\rho>0$. The homogeneous harmonic functions $h^\pm_j$ of \eqref{hjp} and \eqref{hjm} have the following properties:
\\[0.5ex]
\begin{subequations}
a) The functions $h^+_j$ belong to $H^1_0(\Gamma\cap\overline{\sB(\rho)})$ and satisfy the orthogonality relations
\begin{equation}
\label{h+inH1}
   \int_{\Gamma\cap\sB(\rho)}\nabla h^+_i\cdot \nabla h^+_j\,\rd x
   = \lambda^+_j\rho^{2\lambda^+_j+n-2}\,\delta_{i,j}\,.
\end{equation}
b) The functions $h^-_j$ belong to $H^1_{\rw,0}(\Gamma\cap\overline{\sB^\complement(\rho)})$ and satisfy the orthogonality relations
\begin{equation}
\label{h-inH1}
   \int_{\Gamma\cap{\sB^\complement(\rho)}}\nabla h^-_i\cdot \nabla h^-_j\,\rd x
   = - \lambda^-_j\rho^{2\lambda^-_j+n-2} \,\delta_{i,j}\,.
\end{equation}
\end{subequations}
\end{lemma}

\begin{proof}
The orthogonality of the families $\{h^{\pm}_{j}\}_{j}$ follows from the orthogonality of $\{\psi_{j}\}_{j=1}^{\infty}$ in both $L^{2}(\hat\Gamma)$ and $H^{1}(\hat\Gamma)$. It remains to compute the $H^{1}$ seminorms, using \eqref{E:H1finitecone}. We find with
$(\lambda^{+}_{j})^{2}+\mu_{j}=\lambda^{+}_{j}(2\lambda^{+}_{j}+n-2)$
$$
\begin{aligned}
 \int_{\Gamma\cap\sB(\rho)}|\nabla h^+_i|^{2}\,\rd x &=
 \int_{0}^{\rho} \big(
   |\lambda^{+}_{j}r^{\lambda^{+}_{j}-1}|^{2}r^{n-1}+
   |r^{\lambda^{+}_{j}}|^{2}r^{n-3}\mu_{j}
   \big)\,\rd r\\
   &= \big((\lambda^{+}_{j})^{2}+\mu_{j}\big)\int_{0}^{\rho}r^{2\lambda^{+}_{j}+n-3}\,\rd r
   = \lambda^{+}_{j}\rho^{2\lambda^{+}_{j}+n-2}\,.
\end{aligned}
$$
The proof for \eqref{h-inH1} is similar.
\end{proof}

We can now summarize the result about the expansion in the basis of the harmonic extensions $h^{\pm}_{j}$ of the Laplace-Beltrami eigenfunctions $\psi_{j}$. The proof of Theorem \ref{uUseries} follows from Proposition~\ref{P:Fourierpsij}, with the special form \eqref{c+j} and \eqref{B-j} of the coefficient functions.

\begin{theorem}
\label{uUseries}
Choose $\rho>0$ and recall that $\tr_\rho$ is the trace operator defined  in Lemma {\em\ref{L:L2trace}}. 

a) 
For all $j \geq 1$, we denote by $\mathbf{c}_j$ the functional 
\begin{equation}\label{cj}
\begin{tabular}{cccc}
   $\mathbf{c}_j \colon$ & $H^1_0(\Gamma\cap\overline{\sB(\rho)})$ &$\to$ & $\R$\\
   & $u$ &$\mapsto$& $\di\rho^{-\lambda^+_j}
   \int_{\hat{\Gamma}} \tr_{\rho}u (\vartheta) \;\psi_j(\vartheta)\, \rd\sigma_{\vartheta}\,.$
\end{tabular}
\end{equation}
If $u\in H^1_0(\Gamma\cap\overline{\sB(\rho)})$ is harmonic, 
then the sequence $\{\mathbf{c}_j(u)\}$ satisfies:
\begin{equation}\label{useries.conv}
   \sum_{j \geq 1} \lambda^+_j \rho^{2\lambda^+_j+n-2}\,|\mathbf{c}_j(u)|^2
   = \int_{\Gamma\cap\sB(\rho)} |\nabla u|^2 \,\rd x
\end{equation}
and we have the representation formula for $u$:
\begin{equation}\label{useries.eq0}
   u = \sum_{j \geq 1} \mathbf{c}_j(u) \,h^+_j \quad
   \mbox{with convergence in $H^1_0(\Gamma\cap\overline{\sB(\rho)})$.}
\end{equation}

b) 
For all $j \geq 1$, we denote by $\mathbf{B}_j$ the functional 
\begin{equation}\label{Bj}
\begin{tabular}{cccc}
   $\mathbf{B}_j \colon$ & 
   $H^1_{\rw,0}(\Gamma\cap\overline{\sB^\complement(\rho)})$ &$\to$ & $\R$\\
   & $U$ &$\mapsto$& $\di\rho^{-\lambda^-_j}
   \int_{\hat{\Gamma}} \tr_{\rho}U (\vartheta) \;\psi_j(\vartheta)\, \rd\sigma_{\vartheta}\,.$
\end{tabular}
\end{equation}
If $U\in H^1_{\rw,0}(\Gamma\cap\overline{\sB^\complement(\rho)})$ is harmonic, 
then the sequence $\{\mathbf{B}_j(U)\}$ satisfies:
\begin{equation}\label{Useries.conv}
   -\sum_{j \geq 1} \lambda^-_j \rho^{2\lambda^-_j+n-2}\,|\mathbf{B}_j(U)|^2
   = \int_{\Gamma\cap\sB^\complement(\rho)} |\nabla U|^2 \,\rd X
\end{equation}
and we have the representation formula for $U$:
\begin{equation}\label{Useries.eq0}
   U = \sum_{j \geq 1} \mathbf{B}_j(U) \,h^-_j \quad
   \mbox{with convergence in $H^1_{\rw,0}(\Gamma\cap\overline{\sB^\complement(\rho)})$.}
\end{equation}
\end{theorem}

An obvious consequence of expansions \eqref{useries.eq0} and \eqref{Useries.eq0} is that coefficients $\mathbf{c}_j(u)$ and $\mathbf{B}_j(U)$ do not depend on the choice of $\rho$ as soon as $u$ and $U$ satisfy the assumptions of Theorem \ref{uUseries}.

\subsection{Expansion of transfer operators}
\label{ss:ExpTrans}
Recall the Notation \ref{N:comm} for the commutators and the expression \eqref{Meij} of the transfer operators
\begin{subequations}
\begin{align}
   \mathcal{M}_{\rP,\Omega}[\varepsilon] &=
   [\Delta_X,\Phi] \circ \mathcal{H}_{1/\varepsilon}\,, \\
   \mathcal{M}_{\Omega,\rP}[\varepsilon] &=
   [\Delta_x,\varphi] \circ \mathcal{H}_{\varepsilon} \,.
\end{align}
\end{subequations}
Here we are going to combine these expressions with the expansions \eqref{useries.eq0} and \eqref{Useries.eq0}\smallskip

\begin{subequations}
a) Let $u\in\gE_\Omega$, set $\rho=\frac{r_0}{2}$. Then the series expansion \eqref{useries.eq0} of $u$ holds  in $\Gamma\cap\sB(\rho)$. Let $\varepsilon>0$ be given. Applying the change of variables $\mathcal{H}_{1/\varepsilon}$ we deduce from \eqref{useries.eq0}
\[
   (\mathcal{H}_{1/\varepsilon} u)(X) = 
   \sum_{j \geq 1} \mathbf{c}_j(u) \,
   \varepsilon^{\lambda^+_j} h^+_j(X) \quad
   \mbox{with convergence in $H^1_0(\Gamma\cap\overline{\sB(\frac{\rho}{\varepsilon})})$.}
\]
Applying the operator $[\Delta_X,\Phi]$ on both sides, we find the convergence in $L^2(\Gamma\cap\sB(\frac{\rho}{\varepsilon}))$ of the series 
\begin{equation}
\label{eq:MPOa}
   \mathcal{M}_{\rP,\Omega}[\varepsilon] u = 
   \sum_{j \geq 1} \mathbf{c}_j(u) \,
   \varepsilon^{\lambda^+_j} [\Delta_X,\Phi] h^+_j \,.
\end{equation}
The support of $[\Delta_X, \Phi] h^+_j$ is contained in $\Gamma \cap \sA(R_0, 2R_0)$, which is a subset of $\Gamma \cap \sB\left(\frac{\rho}{\varepsilon}\right)$ when $ 0 < \varepsilon \leq \frac{\varepsilon_0}{4}$ because, in that case, the radius $2R_0$ is smaller than $\frac{\rho}{\varepsilon} = \frac{r_0}{2\varepsilon}$. Hence, the series \eqref{eq:MPOa} converges in $L^2(\rP)$ for $\varepsilon \in (0, \frac{\varepsilon_0}{4})$.
\smallskip

b) Let $U\in\gE_\rP$, set  now $\rho=2R_0$. Then the series expansion \eqref{Useries.eq0} for $U$ holds in $\Gamma\cap\sB^\complement(\rho)$. Applying the change of variables $\mathcal{H}_{\varepsilon}$ we find for any given $\varepsilon>0$
\[
   (\mathcal{H}_{\varepsilon} U)(x) = 
   \sum_{j \geq 1} \mathbf{B}_j(U) \,
   \varepsilon^{-\lambda^-_j} h^-_j(x) \quad
   \mbox{with convergence in 
   $H^1_{\rw,0}(\Gamma\cap\overline{\sB^\complement(\rho\varepsilon)})$.}
\]
Hence with the weighted $L^{2}$ space introduced in \eqref{E:L2w} the convergence in $L^2_{\rw}(\Gamma\cap\sB^\complement(\rho\varepsilon))$ of the series
\begin{equation}
\label{eq:MOPb}
   \mathcal{M}_{\Omega,\rP}[\varepsilon] U = 
   \sum_{j \geq 1} \mathbf{B}_j(U) \,
   \varepsilon^{-\lambda^-_j} [\Delta_x,\varphi] h^-_j \,.
\end{equation}
The support of $[\Delta_x,\varphi] h^-_j$ is contained in $\Gamma\cap\sA(\frac{r_0}{2},r_0)$, which is, for any $\varepsilon\in (0,\frac{\varepsilon_0}{4})$, a subset of $\Gamma\cap\sB^\complement(\rho\varepsilon)$. Whence, finally, the convergence in $L^2(\Omega)$ for such $\varepsilon$.
\end{subequations}
\smallskip

In the following Proposition \ref{P:CjBjopnorm} and Theorem \ref{CjBjopnorm} we prove that the series \eqref{eq:MPOa} and \eqref{eq:MOPb} converge in {\em operator norms}.

\begin{proposition}
\label{P:CjBjopnorm}
For any integer $j\ge1$ define the operators $\cC_j$ and $\cB_j$ by
\begin{subequations}
\begin{equation}
\label{eq:Cjcal}
\begin{aligned}
\cC_j \colon & \gE_\Omega \to \gF_\rP\\
& u \mapsto \cC_j[u](X):=\mathbf{c}_j(u)\,[\Delta_X,\Phi] h^+_j(X), \quad X\in P
\end{aligned}
\end{equation}
and
\begin{equation}
\label{eq:Bjcal}
\begin{aligned}
\cB_j \colon & \gE_\rP \to \gF_\Omega\\
& U \mapsto \cB_j[U](x):=\mathbf{B}_j(U)\,[\Delta_x,\varphi] h^-_j(x),\quad x\in \Omega\,
\end{aligned}
\end{equation}
\end{subequations}
with the functionals $\mathbf{c}_j$ and $\mathbf{B}_j$ defined in \eqref{cj} and \eqref{Bj}, respectively. Then there exist a constant $A>0$ such that,  for all integer $j\ge1$, the operator norms $\|\mathcal{C}_j\|_{\mathcal{L}(\gE_\Omega,\gF_\rP)}$ and $\|\mathcal{B}_j\|_{\mathcal{L}(\gE_\rP,\gF_\Omega)}$ satisfy
\begin{subequations}
\begin{align}
\label{Cjopnorm}
   \DNorm{\mathcal{C}_j}{\mathcal{L}(\gE_\Omega,\gF_\rP)}&\le A
   \left(\frac{4}{\varepsilon_0}\right)^{\lambda^+_j}\\
\label{Bjopnorm}
   \DNorm{\mathcal{B}_j}{\mathcal{L}(\gE_\rP,\gF_\Omega)}&\le A
   \left(\frac{4}{\varepsilon_0}\right)^{-\lambda^-_j} .
\end{align}
\end{subequations}
\end{proposition}

\begin{proof}
a) Let $u\in\gE_\Omega$ and $\rho={r_0}/{2}$. On one hand, equality \eqref{useries.conv} can also be written as
\[
\begin{aligned}
   \sum_{j \geq 1} \DNormc{\nabla h^+_j}{L^2(\Gamma\cap\sB(r_0/2))}\,|\mathbf{c}_j(u)|^2
   &= \DNormc{\nabla u}{L^2(\Gamma\cap\sB(r_0/2))}\\[-1ex]
   &\le \DNormc{u}{\gE_\Omega},
\end{aligned}
\]
and on the other hand we have
\[
   \DNorm{[\Delta_X,\Phi] h^+_j}{\gE_\rP} \le C_\Phi\, \DNorm{h^+_j}{H^1(\Gamma\cap\sA(R_0,2R_0))}
\]
from which we draw for any integer $j\ge1$
\[
   \DNorm{\mathbf{c}_j(u)\,[\Delta_X,\Phi] h^+_j}{\gF_\rP} \le C_\Phi\,
   \frac{\DNorm{h^+_j}{H^1(\Gamma\cap\sA(R_0,2R_0))}}
   {\DNorm{\nabla h^+_j}{L^2(\Gamma\cap\sB(r_0/2))}}\;
   \DNorm{u}{\gE_\Omega}\,.
\]
To bound the $H^1$-norm of $h^+_j$, we use the Poincar\'e inequality \eqref{E:PoincAr1r2}
\[
   \DNormc{|X|^{-1}\,h^+_j}{L^2(\Gamma\cap\sB(2R_0))} \le 
   \mu_1^{-1} \DNormc{\nabla h^+_j}{L^2(\Gamma\cap\sB(2R_0))}
\]
which implies that 
\[
   \DNormc{h^+_j}{L^2(\Gamma\cap\sB(2R_0))} \le 
   \mu_1^{-1} (2R_0)^2 \DNormc{\nabla h^+_j}{L^2(\Gamma\cap\sB(2R_0))}
\]
hence
\[
   \DNorm{h^+_j}{H^1(\sA(R_0,2R_0))} \le 
   \sqrt{1+\mu_1^{-1} (2R_0)^2} \;\DNorm{\nabla h^+_j}{L^2(\Gamma\cap\sB(2R_0))}\,.
\]
Thus we have
\[
   \DNorm{\mathbf{c}_j(u)\,[\Delta_X,\Phi] h^+_j}{\gF_\rP} \le C_\Phi\,
   \sqrt{1+\mu_1^{-1} (2R_0)^2} \;
   \frac{\DNorm{\nabla h^+_j}{L^2(\Gamma\cap \sB(2R_0))}}
   {\DNorm{\nabla h^+_j}{L^2(\Gamma\cap\sB(r_0/2))}}\;
   \DNorm{u}{\gE_\Omega}\,.
\]
Using formula \eqref{h+inH1}, we find immediately
\[
   \frac{\DNorm{\nabla h^+_j}{L^2(\Gamma\cap \sB(2R_0))}}
   {\DNorm{\nabla h^+_j}{L^2(\Gamma\cap\sB(r_0/2))}} =
   \Big(\frac{2R_0}{r_0/2}\Big)^{\lambda^+_j+\frac n2-1}\,,
\]
which ends the proof of estimate \eqref{Cjopnorm}.
\medskip

b) The proof of the second estimate \eqref{Bjopnorm} is very similar: 
We take $U\in\gE_\rP$ and $\rho=2R_0$, and start with equality \eqref{Useries.conv}.
Then we use the Poincar\'e inequality \eqref{E:PoincAr1r2} at infinity for $h^-_j$ 
\[
   \DNormc{|x|^{-1}\,h^-_j}{L^2(\Gamma\cap\sB^\complement(r_0/2))} \le 
   \mu_1^{-1} \DNormc{\nabla h^-_j}{L^2(\Gamma\cap\sB^\complement(r_0/2))}
\]
which implies that (note the slight difference here)
\[
   \DNormc{h^-_j}{L^2(\Gamma\cap\sA(r_0/2,r_0))} \le 
   \mu_1^{-1} (r_0)^2 \DNormc{\nabla h^-_j}{L^2(\Gamma\cap\sB^\complement(r_0/2))}\,.
\]
We deduce
\[
   \DNorm{\mathbf{B}_j(U)\,[\Delta_x,\varphi] h^-_j}{\gF_\Omega} \le C_\varphi\,
   \sqrt{1+\mu_1^{-1} (r_0)^2} \;
   \frac{\DNorm{\nabla h^-_j}{L^2(\Gamma\cap \sB^\complement(r_0/2))}}
   {\DNorm{\nabla h^-_j}{L^2(\Gamma\cap\sB^\complement(2R_0))}}\;
   \DNorm{U}{\gE_\rP}\,,
\]
and we conclude using formula \eqref{h-inH1}.
\end{proof}

\begin{theorem}
\label{CjBjopnorm}
The transfer operators $\mathcal{M}_{\rP,\Omega}$ and ${\mathcal{M}}_{\Omega,\rP}$ are  given by the operator series  
\begin{equation}\label{M10M01series}
   \mathcal{M}_{\rP,\Omega}[\varepsilon]=
   \sum_{j \geq 1}\varepsilon^{\lambda_j^+}\mathcal{C}_j
   \quad\text{and}\quad
   {\mathcal{M}}_{\Omega,\rP}[\varepsilon] =
   \sum_{j \geq 1}\varepsilon^{-\lambda_j^-}\mathcal{B}_j\,,
\end{equation}
which, for all $\varepsilon\in[0,\frac{\varepsilon_0}{4})$, are normally converging  in $\mathcal{L}(\gE_\Omega,\gF_\rP)$ and $\mathcal{L}(\gE_\rP,\gF_\Omega)$, respectively: 
\begin{equation}\label{M10M01normal}
   \sum_{j \geq 1}\varepsilon^{\lambda_j^+} \DNorm{\mathcal{C}_j}{\mathcal{L}(\gE_\Omega,\gF_\rP)}
   \!<\infty
   \quad\text{and}\quad
   \sum_{j \geq 1}\varepsilon^{-\lambda_j^-}\DNorm{\mathcal{B}_j}{\mathcal{L}(\gE_\rP,\gF_\Omega)}
   \!<\infty, \quad 0\le\varepsilon<\tfrac{\varepsilon_0}{4}\,.
\end{equation}
\end{theorem}

In view of estimates \eqref{Cjopnorm} and \eqref{Bjopnorm}, the sums in \eqref{M10M01normal} compare to the pseudo-geometrical series
\[
   \sum_{j \geq 1}\varepsilon^{\lambda_j^+} \left(\frac{4}{\varepsilon_0}\right)^{\lambda^+_j}
   \quad\mbox{and}\quad
   \sum_{j \geq 1}\varepsilon^{-\lambda_j^-} \left(\frac{4}{\varepsilon_0}\right)^{-\lambda^-_j}.
\]
If the $\lambda^+_j$ and $-\lambda^-_j$ behave like the integer $j$, the convergence for $\varepsilon\in[0,\frac{\varepsilon_0}{4})$ is obvious, and this occurs in dimension $n=2$. But for higher dimensions $n$, the distance between consecutive $\lambda^\pm_j$ globally tends to $0$ as $j\to\infty$. This is quantified with the help of the spectral counting function (Weyl's law).

\begin{lemma}\label{mudistribution}
Let $\#\{\mu_j\le \mu\}$ denote the number of eigenvalues (counting multiplicity) of $\gL^\dir_{\hat\Gamma}$ that are  less than or equal to $\mu$, then we have 
\[
   \#\{\mu_j\le \mu\} \;\le \; d_n\,\mu^{(n-1)/2},\quad\forall\mu\ge1,
\]
for some constant $d_n>0$ that depends only on $n$.
\end{lemma}

\begin{proof}
If $N_{\bS^{n-1}}(\mu)$ is the number of eigenvalues (counting multiplicity) of  the (positive) Laplace-Beltrami operator $\gL_{\bS^{n-1}}$ on $H^1(\bS^{n-1})$ that are smaller than or equal to $\mu>0$, then we have
\[
N_{\bS^{n-1}}(\mu)\le\frac{2}{(n-1)!}\,\mu^{(n-1)/2}+c_n\mu^{(n-2)/2}+c_n
\]
for some constant $c_n>0$ that depends only on $n$ (cf., e.g., Shubin \cite[p.~172]{Sh01}).
Since the $j$-th eigenvalue of  $\gL_{\bS^{n-1}}$ on $H^1_0(\hat \Gamma)$ is bigger than or equal to the $j$-th eigenvalue of  $\gL_{\bS^{n-1}}$ on $H^1(\bS^{n-1})$, we deduce the inequality in the lemma.
\end{proof}

Now, in the same spirit, we denote by $\#\{\pm\lambda^\pm_j\le \lambda\}$ the number of $\pm\lambda^\pm_j$'s less than or equal to  $\lambda$ (with repetitions corresponding to multiple eigenvalues $\mu_j$). As a consequence of Lemma \ref{mudistribution}, using formulas \eqref{lambdapm} we find:

\begin{corollary}
\label{cor:lampol}
There exists a constant $d'_n$ depending only on $n$ such that, for all $\lambda\ge1$,
\begin{subequations}
\begin{align}
\label{N+<}
    \#\{\lambda^+_j\le \lambda\} &\le  d'_n\,\lambda^{n-1}, \\
\label{N-<}
    \#\{-\lambda^-_j\le \lambda\} &\le  d'_n\,\lambda^{n-1} .
\end{align}
\end{subequations}
\end{corollary}

The proof of normal convergence of the series \eqref{M10M01series} then relies on the fact that generalized power series $\sum_{j\ge1}\xi^{\lambda_j}$ with exponents $\{\lambda_j\}$ whose  counting function has a polynomial growth is converging for $0\le\xi<1$, just as a standard series with integer exponents. There is an analogous result in the framework of generalized power series, see Lemma~\ref{L:polgro}, but although the proofs use similar arguments, we cannot just apply that result here, because the sequences $(\pm\lambda_{j}^{\pm})_{j\ge1}$ may involve nontrivial multiplicities.

\begin{lemma}
\label{lem:convpol}
Let $\{\lambda_j\}$ be a sequence of positive exponents satisfying, for some $m>0$, the distribution law
\[
    \#\{\lambda_j\le \lambda\} \le  d\,\lambda^{m} ,\quad \forall\lambda\ge1.
\]
Then the series $\sum_{j\ge1}\xi^{\lambda_j}$ is converging for all $\xi\in[0,1)$.
\end{lemma}

\begin{proof}
We organize the sum $\sum_{j\ge1}\xi^{\lambda_j}$ by packets:
\[
   \sum_{j\ge1}\xi^{\lambda_j} = \sum_{\ell=0}^\infty
   \Big(\sum_{j,\ \ell\le\lambda_j<\ell+1} \xi^{\lambda_j}\Big).
\]
In each packet, $\xi^{\lambda_j}\le \xi^\ell$ because $0\le\xi<1$, so this packet
can be bounded by its number of elements multiplied by $\xi^\ell$, arriving at
\[
   \sum_{j,\ \ell\le\lambda_j<\ell+1} \xi^{\lambda_j} \le
   \big(\#\{\lambda_j\le\ell+1\}\big)\; \xi^{\ell} \le 
   d\, (\ell+1)^{m} \xi^{\ell}.
\]
Since the sum $\sum_{\ell=0}^\infty (\ell+1)^{m} \xi^{\ell}$ is convergent for $0\le\xi<1$, the lemma is proved.
\end{proof}
\smallskip

\begin{proof}[Proof of Theorem \ref{CjBjopnorm}]
We only have to prove that the series
\[
   \sum_{j \geq 1}\varepsilon^{\lambda_j^+} 
   \DNorm{\mathcal{C}_j}{\mathcal{L}(\gE_\Omega,\gF_\rP)}
   \quad\text{and}\quad
   \sum_{j \geq 1}\varepsilon^{-\lambda_j^-} 
   \DNorm{\mathcal{B}_j}{\mathcal{L}(\gE_\rP,\gF_\Omega)}
\]
are converging for all $\varepsilon\in[0,\frac{\varepsilon_0}{4})$. Using \eqref{Cjopnorm}-\eqref{Bjopnorm}, we have:
\[
   \sum_{j \geq 1}\varepsilon^{\lambda_j^+} 
   \DNorm{\mathcal{C}_j}{\mathcal{L}(\gE_\Omega,\gF_\rP)} \!\le
   A \sum_{j \geq 1}
   \left(\frac{4\varepsilon}{\varepsilon_0}\right)^{\lambda^+_j}
   \!\!\quad\text{and}\quad
   \sum_{j \geq 1}\varepsilon^{-\lambda_j^-} 
   \DNorm{\mathcal{B}_j}{\mathcal{L}(\gE_\rP,\gF_\Omega)} \!\le
   A \sum_{j \geq 1} 
   \left(\frac{4\varepsilon}{\varepsilon_0}\right)^{-\lambda^-_j}
\]
hence the convergence using Corollary \ref{cor:lampol} and Lemma \ref{lem:convpol}.
\end{proof}

\begin{remark}
The factor $4$ dividing $\varepsilon_0$ in the convergence range in $\varepsilon$ comes from the choice that we have made for the supports of the cutoff functions $\Phi$ and $\varphi$ in \eqref{eq:Phi} and \eqref{eq:phi}. If, instead, we impose $\Phi\equiv1$ on $\sB^\complement(\varkappa R_0)$ and $\varphi\equiv1$ on $\sB(r_0/\varkappa)$ for some $\varkappa>1$, then the convergence range would be ${\varepsilon_0}/{\varkappa^2}$, but the constant $A$ appearing in \eqref{Cjopnorm}-\eqref{Bjopnorm} would blow up as $\varkappa$ is closer to $1$.
\end{remark}

\section{Inversion by a Neumann series}\label{s:M-1}

\subsection{Expansion of the operator $\cM[\varepsilon]$} 

With the converging expansions of $\cM_{\rP,\Omega}[\varepsilon]$ and $\cM_{\Omega,\rP}[\varepsilon]$ at hand (Theorem \ref{CjBjopnorm}), we have a converging expansion of the full operator $\cM[\varepsilon]$ \eqref{Me} that we write in the form
\[
   \cM[\varepsilon] = \sum_{j\in\Z} \cM_j\, \varepsilon^{\lambda_j}
\]
with the exponents $\lambda_j$ defined as (recall that the $\mu_j$ are the eigenvalues of $\gL^\dir_{\hat\Gamma}$)
\begin{equation}
\label{eq:lamj}
\lambda_j := 
\begin{cases}
   \ \lambda^+_j = 1-\frac{n}{2} + \sqrt{\left(1-\frac{n}{2}\right)^2 + \mu_j}& \mbox{if}\; j>0,\\
   \ \ 0 & \mbox{if}\; j=0,\\
   -\lambda^-_j = n-2+\lambda_{-j} & \mbox{if}\; j<0,\\
\end{cases}
\end{equation}
and the operator matrix coefficients $\cM_j$ defined as
\[
\cM_0 = \begin{pmatrix}
        \cM_{\Omega,\Omega}&0\\
        0&  \cM_{\rP,\rP}
        \end{pmatrix}\quad\mbox{and}\quad
\cM_j = \begin{cases}
   \begin{pmatrix} 0 & 0\\ \cC_j & 0 \end{pmatrix}  & \mbox{if}\; j>0,\\[2.5ex]
   \begin{pmatrix} 0 & \cB_j\\ 0 & 0 \end{pmatrix} & \mbox{if}\; j<0.\\
\end{cases}
\]
Each operator $\cM_j$ is bounded from $\gE_\Omega\times\gE_\rP$ to $\gF_\Omega\times\gF_\rP$ and, as a direct consequence of Theorem \ref{CjBjopnorm}, the series $\sum_{j\in\Z} \cM_j\, \varepsilon^{\lambda_j}$ converges normally for all $\varepsilon\in[0,\frac{\varepsilon_0}{4})$. As $\cM_0$ is invertible we may write
\begin{equation}
\label{eq:M-1Me}
   \cM^{-1}_0\cM[\varepsilon] = 
   \Id + \sum_{j\in\Z^*} \cM^{-1}_0\cM_j\; \varepsilon^{\lambda_j}
\end{equation}
where the operators $\cM^{-1}_0\cM_j$ are bounded from $\gE_\Omega\times\gE_\rP$ in itself and satisfy the estimates
\begin{equation}
\label{eq:NormM-1Mj}
   \DNorm{\cM^{-1}_0\cM_j}{\cL(\gE_\Omega\times\gE_\rP)} \le 
   A \DNorm{\cM^{-1}_0}{\cL(\gF_\Omega\times\gF_\rP,\gE_\Omega\times\gE_\rP)}
   \left(\frac{4}{\varepsilon_0}\right)^{\lambda_j}\,,
\end{equation}
cf.~\eqref{Cjopnorm}-\eqref{Bjopnorm}.

The series \eqref{eq:M-1Me} fits within the framework of convergent generalized power series, as introduced in Definition \ref{D:cgps}. However, we need to switch from considering the exponents $\lambda_j$ as a \emph{sequence} with possible multiplicities (repeated values) to using the corresponding \emph{set} of exponents.

\begin{lemma}
\label{lem:gAe}
(i) Let $\sfE$ be the \emph{set} of exponents consisting of all the values of the sequence $\{\lambda_j\}$, with $\lambda_j$ given by \eqref{eq:lamj}:
\begin{equation}
\label{eq:sfE}
   \sfE := \{\sfe \in \R_+ ,\quad  \exists j \in \Z^*,\ \sfe = \lambda_j\}\,.
\end{equation}
Then, for all $\sfe \in \sfE$, let $\gA_\sfe$ be the operator coefficient defined by 
\begin{equation}
\label{eq:gAe}
   \gA_\sfe := \sum_{\lambda_j = \sfe} \cM^{-1}_0 \cM_j.
\end{equation}
It holds that
\begin{equation}
\label{eq:M-1MeAe}
   \cM^{-1}_0 \cM[\varepsilon] = 
   \Id + \sum_{\sfe \in \sfE} \gA_\sfe \varepsilon^\sfe,  
\end{equation}
and the series on the right-hand side converges normally in $\cL(\gE_\Omega \times \gE_\rP)$ for all $\varepsilon \in [0, \frac{\varepsilon_0}{4})$.

\medskip

(ii) Let $\sfe_*$ be the smallest element of $\sfE$. Then $\sfe_* = \lambda_1 > 0$ and
\begin{equation}
\label{eq:Ae}
   \sum_{\sfe \in \sfE} \gA_\sfe \varepsilon^\sfe =
   \varepsilon^{\sfe_*} \sum_{\sfe \in \sfE} \gA_\sfe \varepsilon^{\sfe - \sfe_*}\,,
\end{equation}
where the series on the right-hand side converges normally in $\cL(\gE_\Omega \times \gE_\rP)$ for all $\varepsilon \in [0, \frac{\varepsilon_0}{4})$.

\medskip

(iii) There exists $\varepsilon_\star > 0$ (with $\varepsilon_\star \le \frac{\varepsilon_0}{4}$) such that, for all $\varepsilon \in [0, \varepsilon_\star)$, there holds
\begin{equation}
\label{eq:Ae<1}
   \sum_{\sfe \in \sfE} \DNorm{\gA_\sfe}{\cL(\gE_\Omega \times \gE_\rP)} \varepsilon^\sfe < 1.
\end{equation}
\end{lemma}

\begin{proof}
To prove statement (i), it suffices to write:
\[
\begin{aligned}
   \sum_{\sfe\in\sfE} \DNorm{\gA_\sfe}{\cL(\gE_\Omega\times\gE_\rP)} \; \varepsilon^\sfe &\le
   \sum_{\sfe\in\sfE} \sum_{\lambda_j=\sfe} 
   \DNorm{\cM^{-1}_0\cM_j}{\cL(\gE_\Omega\times\gE_\rP)} \; \varepsilon^\sfe \\
   &=\sum_{j\in\Z^*} 
   \DNorm{\cM^{-1}_0\cM_j}{\cL(\gE_\Omega\times\gE_\rP)} \; \varepsilon^{\lambda_j}
\end{aligned}
\]
and to note that the last sum converges for $\varepsilon\in[0,\frac{\varepsilon_0}{4})$ by Theorem \ref{CjBjopnorm}.

The identity $\sfe_*=\lambda_1$ is a consequence of equality $\lambda_{-1}=n-2+\lambda_1$. The normal convergence of the series \eqref{eq:Ae} is proved as above, with the exponents $\lambda_j$ replaced with $\lambda_j-\lambda_1$. Hence, point {\em (ii)}. Finally, point {\em (iii)} is an obvious consequence of point {\em (ii)}.
\end{proof}

\subsection{Expansion of the inverse of $\cM[\varepsilon]$}
\label{sec:Meps-1}
Let
\[
   \gA[\varepsilon] = \sum_{\sfe\in\sfE} \gA_\sfe\; \varepsilon^\sfe.
\]
We are ready to invert $\Id + \gA[\varepsilon]$:
In view of Lemma \ref{lem:gAe} and Theorem \ref{T:invconv}, for all $\varepsilon\in[0,\varepsilon_\star)$ the Neumann series
\[
   (\Id + \gA[\varepsilon])^{-1} = \Id + \sum_{k=1}^\infty (-\gA[\varepsilon])^k
\]
is a convergent generalized power series with exponent set $\sfE^\infty$ defined as the {\em additive monoid generated by the set} $\sfE$ in \eqref{eq:sfE}:
\begin{equation}
\label{eq:Einfb}
   \sfE^\infty := \bigcup_{k=1}^\infty \sfE^k \, \cup\{0\} \quad\mbox{with}\quad
   \sfE^k := \big\{\sfe=\sfe_1+\cdots+\sfe_k,\quad 
   \sfe_1,\ldots,\sfe_k\in \sfE\big\}.
\end{equation} 
Then, comparing with \eqref{E:Neumann}, we see that we have
\begin{equation*}
   (\Id + \gA[\varepsilon])^{-1} 
   = \Id \ + \sum_{\sfe\in \sfE^\infty\setminus\{0\}} \gB_\sfe\,\varepsilon^\sfe,
   \quad\mbox{with}\quad
   \gB_\sfe := \sum_{k\ge1}(-1)^k 
   \!\!\sum_{\substack{\sfe_1+\dots+\sfe_k=\sfe\\ \sfe_1,\dots,\sfe_k\in \sfE}}\!\!
   \gA_{\sfe_1}\dots \gA_{\sfe_k}\,,
\end{equation*}
and the sum $\sum_{\sfe\in \sfE^\infty\setminus\{0\}} \gB_\sfe\,\varepsilon^\sfe$ converges normally for all $\varepsilon\in[0,\varepsilon_\star)$:
\begin{equation*}
   \sum_{\sfe\in \sfE^\infty\setminus\{0\}}\!\! 
   \DNorm{\gB_\sfe}{\cL(\gE_\Omega\times\gE_\rP)}\,\varepsilon^\sfe <\infty 
   \quad \forall\varepsilon\in[0,\varepsilon_\star).
\end{equation*}
As $(\Id + \gA[\varepsilon])^{-1} \cM_0^{-1} = \cM[\varepsilon]^{-1}$, we have obtained the following:

\begin{theorem}
\label{M-1}
Let $\varepsilon_\star>0$ be as in Lemma {\em\ref{lem:gAe}}. With the notations introduced above, the following holds:

\smallskip\noindent
(i) The operator $\cM[\varepsilon]$ is invertible for all $\varepsilon\in[0,\varepsilon_\star)$.

\smallskip\noindent
(ii) Let $\sfE^\infty$ be the monoid \eqref{eq:Einfb} generated by the exponents set $\sfE$ \eqref{eq:sfE}. Define
\begin{equation}
\label{eq:N0Ne}
   \gN_0 := \cM_0^{-1}=
   \begin{pmatrix}
   \cM_{\Omega,\Omega}^{-1}&0\\
   0&  \cM_{\rP,\rP}^{-1}
   \end{pmatrix}
   \quad\mbox{and}\quad
   \gN_\sfe := \gB_\sfe \gN_0,
   \quad\forall\sfe\in\sfE^\infty\setminus\{0\}\,.
\end{equation}
Then the inverse of $\cM[\varepsilon]$ can be represented as the sum of the normally convergent generalized series with exponent set $\sfE^\infty$ and coefficients $\gN_\sfe$:
\[
   \cM[\varepsilon]^{-1} = 
   \sum_{\sfe\in\sfE^\infty} \gN_\sfe\,\varepsilon^\sfe =
   \gN_0 + \sum_{\sfe\in\sfE^\infty\setminus\{0\}} \gN_\sfe\,\varepsilon^\sfe,
\]
where the series convergence normally in ${\cL(\gF_\Omega\times\gF_\rP,\gE_\Omega\times\gE_\rP)}$ for all $\varepsilon\in[0,\varepsilon_\star)$.

\smallskip\noindent
(iii) The set $\sfE^\infty$ is discrete, and 
its smallest nonzero element is 
\[
   \sfe_*=\sfe_1=\lambda_1=1-\frac{n}{2} + \sqrt{\left(1-\frac{n}{2}\right)^2 + \mu_1}>0.
\]
\end{theorem}

\begin{remark}
{\em(i)} \ In our application, $\sfE$ is a discrete subset of $\R_+$.
As stated in Appendix \ref{S:gps}, a sufficient property at this stage is that $\sfE$ is a well-ordered subset of $\R_+$ (meaning that every non-empty subset of $\sfE$ has a smallest element). This property is more general because it allows for bounded increasing subsequences.
\smallskip

{\em(ii)} \ The operators $\gN_\sfe$ have finite rank because the operators $\cC_j$ \eqref{eq:Cjcal}  and $\cB_j$ \eqref{eq:Bjcal} have rank one. 
\end{remark}

Let $\Lambda$ denote the set of exponents $\lambda^+_j$ (the exponents corresponding to the singularities of the variational solutions). We then have
\[
   \sfE = \Lambda \cup (n-2 + \Lambda).
\]
In certain cases where we have an explicit expression for $\Lambda$, we can use it to determine explicit expressions for $\sfE$ and $\sfE^\infty$. In what follows, $S_\omega$ denotes the plane sector with opening $\omega$, where $\omega \in (0, 2\pi]$, and $\alpha$ is the quotient $\frac{\pi}{\omega}$.

\begin{example}[of sets $\sfE$ and $\sfE^\infty$]
\label{ex:sfE}
\mbox{ }

{\em(i)} \ When $n=2$ and $\Gamma$ is  the cone  $S_\omega$, we have $\Lambda = \alpha\N^*$, $\sfE=\Lambda$, and $\sfE^\infty=\alpha\N$.
\smallskip

{\em(ii)} \ When $n>2$ and $\Gamma$ is a {\em wedge} with opening angle $\omega$, meaning that $\Gamma=\R^{n-2}\times S_\omega$,  we have $\Lambda = \N+\alpha\N^*$ (cf.~\cite[\S18.C]{DBook88}). Then, $\sfE=\Lambda$ and $\sfE^\infty=\N+\alpha\N^*\cup\{0\}$. As particular cases: for the {\em half-space}, for which $\omega=\pi$, we have $\sfE=\N^*$ and $\sfE^\infty=\N$, and for the {\em crack}, where $\omega=2\pi$ (note that this is a non-Lipschitz domain), we have $\sfE=\frac{1}{2}\N^*$ and $\sfE^\infty=\frac{1}{2}\N$. 
\smallskip

{\em(iii)} \ When $n=3$ and $\Gamma$ is a {\em half-wedge} of opening angle $\omega$, meaning that  $\Gamma=\R_+\times S_\omega$, we have $\Lambda=1+2\N+\alpha\N^*$  (cf.~\cite[\S18.C]{DBook88}). Thus $\sfE=\N^*+\alpha\N^*=1+\alpha+\N+\alpha\N$ and $\sfE^\infty = \sfE\cup\{0\}$. As a particular case: for the {\em octant} $(\R_+)^3$, where $\omega=\frac{\pi}{2}$, we have $\sfE=3+\N$ and $\sfE^\infty = \sfE\cup\{0\}$. 
\smallskip

{\em(iv)} \ When $n=3$ and  $\Gamma$ is a {\em circular cone} with opening angle $\zeta\in(0,\pi)$  (note that $\zeta=\frac{\pi}{2}$ gives back the half-space, and $\zeta=\pi$ is excluded by the non-zero capacity condition \eqref{capacity>0}),  $\Lambda$ consists of the roots $\nu$ of the equations
\[
   \exists m\in\N,\quad\mathsf{P}_\nu^m(\cos\zeta) = 0,
\]
where $\mathsf{P}_\nu^m$ denotes the order-$m$ associated Legendre function of the first kind (cf.~\cite[\S18.D]{DBook88}). Then, $\sfE=\Lambda\cup(1+\Lambda)$ and $\sfE^\infty$ is given by the general formula \eqref{eq:Einfb}. In contrast with the previous examples, the multiplicity of the elements of 
$\Lambda$ is generically $1$ when $\zeta\neq\frac{\pi}{2}$. 
\end{example}

\begin{remark}
We observe that $\sfE^\infty$ is contained in $\N$ when $\Gamma$ is a half-space, or when a half-space can be obtained by a finite number of reflections of $\Gamma$ (as for a sector or a wedge of opening $\frac{\pi}{k}$, where $k$ is an integer $\ge2$, an octant, etc.). In these cases, the series in powers of $\varepsilon$ is a classical power series, and its sum is real-analytic in the open sets where it converges.
\end{remark}

\section{Expansion of solutions of the Dirichlet problem in $\Omega_\varepsilon$}\label{s:expansions}

Using the ansatz \eqref{eq:ansatz} for the solution $\ueps$, the problem of solving the Dirichlet problem \eqref{PoissonfF} has been transformed into the equivalent problem of inverting the block operator matrix $\cM[\varepsilon]$, as proven in Theorem \ref{th:Meps}. Then, in Theorem \ref{M-1}, we inverted $\cM[\varepsilon]$ by means of a Neumann series that takes the form of a generalized power series with exponents in the discrete monoid $\sfE^\infty$ defined in \eqref{eq:Einfb}, generated by the set $\sfE$ in \eqref{eq:sfE}. In this section, we deduce from these results several representations for $\ueps$.  We will present a global representation that depends on the cutoffs $\Phi$ and $\varphi$, as well as intrinsic representations, which are independent of $\Phi$ and $\varphi$, and are defined either close to or away from the vertex of the cone.

\subsection{Global expansions}
\label{ss:global}

By Theorem \ref{M-1} we readily deduce the following:

\begin{theorem}\label{uepsthmfF} 
Let $f\in\gF_\Omega$ and $F\in\gF_\rP$. For any $\varepsilon\in(0,\frac{\varepsilon_0}{4})$, let $\ueps$ be the solution of problem \eqref{PoissonfF}. Let $\varepsilon_\star>0$ be as in Lemma {\em\ref{lem:gAe}} and let the operator coefficients $\gN_\sfe$ be defined as in Theorem {\em\ref{M-1}}. Then the two series 
\[
   \sfu[\varepsilon] := \sum_{\sfe\in\sfE^\infty}\varepsilon^{\sfe}\sfu_\sfe
   \;\;\mbox{in}\;\;\gE_\Omega
   \quad\mbox{and}\quad
   \sfU[\varepsilon] := \sum_{\sfe\in\sfE^\infty}\varepsilon^{\sfe}\sfU_\sfe
   \;\;\mbox{in}\;\;\gE_\rP\,,
\]
with coefficients given, for all $\sfe\in\sfE^\infty$, by the identity
\begin{equation}
\label{eq:ueUefF}
   \begin{pmatrix}
   \sfu_\sfe\\
   \sfU_\sfe
   \end{pmatrix} =
   \gN_\sfe\begin{pmatrix}
   f\\
   F
   \end{pmatrix},
\end{equation}
converge normally in $\gE_\Omega$ and $\gE_\rP$, respectively, for all $\varepsilon\in[0, \varepsilon_\star)$.
Using these series, we can write the following 2-scale representation of $\ueps$:
\begin{align}
\label{eq:uepsglo}
   \ueps(x)&=
   \Phi\left(\frac{x}{\varepsilon}\right)
   \sum_{\sfe\in\sfE^\infty} \varepsilon^{\sfe}\sfu_\sfe(x)
   +\varphi(x)\sum_{\sfe\in\sfE^\infty} \varepsilon^{\sfe}
   \sfU_\sfe\left(\frac{x}{\varepsilon}\right)\,,
\end{align}
which holds for a.e. $x\in\Omega_\varepsilon$ and all $\varepsilon\in(0,\varepsilon_\star)$.
In particular, the first coefficients $u_0$ and $U_0$ of each series are the solutions of the
limit boundary value problems
\begin{equation}
\label{eq:u0U0}
\begin{cases}
\ u_0\in H^1_0(\Omega)\,,\\
\ \Delta u_0=f\,,
\end{cases}
\quad\mbox{and}\qquad
\begin{cases}
\ U_0\in H^1_{\rw,0}(\rP)\,,\\
\ \Delta U_0=F\,,
\end{cases}
\end{equation}
respectively, and, in both series,  the next exponents with non-zero coefficients satisfy the inequality $\sfe\ge\lambda_1>0$.
\end{theorem}

\begin{proof}
Using Theorem \ref{M-1}, we introduce 
the two series $\sfu[\varepsilon]$ and $\sfU[\varepsilon]$ for $0\le \varepsilon<\varepsilon_\star$, with terms in $\gE_\Omega$ and $\gE_\rP$, respectively,  defined by 
\[
   \begin{pmatrix}
   \sfu[\varepsilon]\\
   \sfU[\varepsilon]
   \end{pmatrix} :=
   \cM[\varepsilon]^{-1} \begin{pmatrix}
   f\\
   F
   \end{pmatrix}=\sum_{\sfe\in\sfE^\infty} \varepsilon^{\sfe}\,\gN_\sfe\begin{pmatrix}
   f\\
   F
   \end{pmatrix}.
\]
Then, by Theorem \ref{th:Meps}, $\ueps(x)$ coincides with $\Phi\left(\frac{x}{\varepsilon}\right)
   \sfu[\varepsilon]
   +\varphi(x)\,
   \sfU[\varepsilon]\!\left(\frac{x}{\varepsilon}\right)$. This leads to the representation \eqref{eq:uepsglo}  with coefficients $\sfu_\sfe\in\gE_\Omega$ and $\sfU_\sfe\in\gE_\rP$ defined by \eqref{eq:ueUefF}. The expressions for $u_0$ and $U_0$ follow from the expression of $\gN_0$ in  \eqref{eq:N0Ne}.
\end{proof}

\begin{remark}
\label{rem:elemination}
The two series $\sfu[\varepsilon]$ and $\sfU[\varepsilon]$ are not independent of each other. With the exception of their principal terms $u_0$ and $U_0$, either one of them can be eliminated from the representation of $\ueps$ by taking the corresponding Schur complement. Indeed, the series $\sfu[\varepsilon]$ and $\sfU[\varepsilon]$ are, by definition, solutions of the system
\[
\begin{cases}
 \ \cM_{\Omega,\Omega} \,\sfu[\varepsilon] + \cM_{\Omega,\rP}[\varepsilon]\,\sfU[\varepsilon] = f\,,\\
 \ \cM_{\rP,\Omega}[\varepsilon]\,\sfu[\varepsilon] + \cM_{\rP,\rP}\,\sfU[\varepsilon] = F\,.
\end{cases}
\]
{\em (i)} \ Then, since  $\cM_{\rP,\rP}$ is invertible, we see that 
\[
   \sfU[\varepsilon] = 
   \cM_{\rP,\rP}^{-1}F - \cM_{\rP,\rP}^{-1} \cM_{\rP,\Omega}[\varepsilon]\,\sfu[\varepsilon]
   = U_0- \cM_{\rP,\rP}^{-1} \cM_{\rP,\Omega}[\varepsilon]\,\sfu[\varepsilon],
\]
and for the solution $\ueps$ of \eqref{PoissonfF} we obtain the formula:
\begin{equation}\label{uepsu}
   \ueps(x)=\Phi\left(\frac{x}{\varepsilon}\right)\sfu[\varepsilon](x)
   +\varphi(x)\left\{U_0\left(\frac{x}{\varepsilon}\right)
   -\big(\cM_{\rP,\rP}^{-1}\cM_{\rP,\Omega}[\varepsilon]\sfu[\varepsilon]\big)
   \left(\frac{x}{\varepsilon}\right)\right\}\,,
\end{equation}
which holds for all $0<\varepsilon<\varepsilon_\star$ and a.e. $x\in \Omega_\varepsilon$.

\smallskip\noindent
{\em (ii)} \ Similarly, we can get rid of $\sfu[\varepsilon]$ noting that
\[
   \sfu[\varepsilon] = \cM_{\Omega,\Omega}^{-1}f 
   - \cM_{\Omega,\Omega}^{-1} \cM_{\Omega,\rP}[\varepsilon]\,\sfU[\varepsilon]
   = u_0 - \cM_{\Omega,\Omega}^{-1} \cM_{\Omega,\rP}[\varepsilon]\,\sfU[\varepsilon]
\]
and writing
\begin{equation}\label{uepsU}
   \ueps(x)=\Phi\left(\frac{x}{\varepsilon}\right)
   \Big\{u_0(x) - 
   \left(\cM_{\Omega,\Omega}^{-1} \cM_{\Omega,\rP}[\varepsilon]\,\sfU[\varepsilon]\right)(x)\Big\}
   +\varphi(x)\sfU[\varepsilon]\left(\frac{x}{\varepsilon}\right)
\end{equation}
for all $0<\varepsilon<\varepsilon_\star$ and a.e. $x\in \Omega_\varepsilon$.
\end{remark}

\subsection{Intrinsic  expansions}\label{s:intrinsic}

Theorem \ref{uepsthmfF} shows an expansion of the solution $\ueps$ on the whole of $\Omega_\varepsilon$, but the form of the expansion depends on the cutoffs $\Phi$ and $\varphi$ and, in this sense, it is not intrinsic. We will see that analyzing separately the solutions in the two regions {\em ``close to the vertex of the cone''} and {\em ``away from the vertex of the cone''}, we can write expansions that are independent from $\Phi$ and $\varphi$. 

The region ``close to the vertex of the cone'' is by definition the region where the geometric perturbation is living, i.e.\ where $\Omega_\varepsilon$ coincides with $\varepsilon\rP$. Recall that $\Omega_\varepsilon$ coincides with $\varepsilon\rP$ inside $\sB(r_0)$ (see Remark \ref{rem:regions} {\em (i)}). Thus we may expect that inside $\sB(r_0)$, the solutions $\ueps$ can be described by the rapid variables $X=\frac{x}{\varepsilon}\in\rP$. Paradoxically, it is formula \eqref{uepsu} from which $\sfU[\varepsilon]$ is eliminated that will lead to such an {\em inner expansion}.

The other region ``away from the vertex of the cone'' is the region of the far field with respect to $\varepsilon$-perturbations, namely the region where $\Omega_\varepsilon$ coincides with $\Omega$. From Remark \ref{rem:regions} {\em (i)}, we recall that $\Omega_\varepsilon$ coincides with $\Omega$ outside $\sB(\varepsilon R_0)$, so the outer region extends in principle to $\Omega\cap\sB^\complement(\varepsilon R_0)$, where solutions could be described in the slow variables $x$. Such an {\em outer expansion} will be obtained through formula \eqref{uepsU} from which $\sfu[\varepsilon]$ is eliminated

Before proving the formulas for the inner and outer expansions, we introduce certain  {\em canonical harmonic functions} $K^+_j$ and $K^-_j$  satisfying homogeneous Dirichlet conditions on $\partial\rP$ and $\partial\Omega$, respectively. Since these canonical functions $K^+_j$ and $K^-_j$ are not identically $0$, they cannot belong to the variational spaces $H^1_{\rw,0}(\rP)$ and $H^1_0(\Omega)$, but to some larger ``dual'' spaces. Their definition relies on the following lemma. 

\begin{lemma}
\label{lem:K+K-}
a) Let $h^+$ be one of the homogeneous harmonic functions $h^+_j$ on $\Gamma$. Then there exists a unique function $K^+\in L^2_\loc(\rP)$ satisfying the following three conditions
\begin{subequations}
\label{eq:K+}
\begin{align}
\label{eq:K+a}
   \Delta K^+ &= 0 \quad \mbox{in} \quad \rP, \\
\label{eq:K+b}
   K^+ &\in H^1_0(\rP\cap\overline{\sB(\rho)}) \quad\mbox{for any}\quad \rho>0, \\
\label{eq:K+c}
   K^+ - h^+ &\in H^1_{\rw,0}(\rP\cap\overline{\sB^\complement(R_0)})\,. 
\end{align}
Here $H^1_0(\rP\cap\overline{\sB(\rho)})$ and $H^1_{\rw,0}(\rP\cap\overline{\sB^\complement(R_0)})$ are defined as in Notation {\em\ref{not:H10barU}} (ii) and (iii), with $\rP$ replacing $\Gamma$.
\end{subequations}
\smallskip

b) Let $h^-$ be any of the homogeneous harmonic functions $h^-_j$ on $\Gamma$. Then there exists a unique function $K^-\in L^2_\loc(\Omega)$ satisfying the following three conditions
\begin{subequations}
\label{eq:K-}
\begin{align}
\label{eq:K-a}
   \Delta K^- &= 0 \quad \mbox{in} \quad \Omega, \\
\label{eq:K-b}
   K^- &\in H^1_0(\Omega\cap\overline{\sB^\complement(\rho)}) \quad\mbox{for any}\quad \rho>0, \\
\label{eq:K-c}
   K^- - h^- &\in H^1_0(\Omega\cap\overline{\sB(r_0)})\,,
\end{align}
where $H^1_0(\Omega\cap\overline{\sB^\complement(\rho)})$ and $H^1_0(\Omega\cap\overline{\sB(r_0)})$ are also defined as in Notation {\em\ref{not:H10barU}}, but with $\Omega$ replacing $\Gamma$.
\end{subequations}
\end{lemma}

Let us note that, due to its behavior at infinity, $h^+$ does not belong to $H^1_{\rw,0}(\rP\cap\overline{\sB^\complement(R_0)})$. Condition \eqref{eq:K+c} means that $K^+(X)$ is equal to $h^+(X)$ when $|X|\to\infty$, modulo a correction in $H^1_{\rw,0}(\rP)$. In particular, it implies that $K^+\not\equiv0$. Similarly, $h^-$ does not belong to $H^1_0(\Omega\cap\overline{\sB(r_0)})$ due to its behavior at the vertex, and condition \eqref{eq:K-c} means that $K^-(x)$ is equal to $h^-(x)$ when $|x|\to0$, modulo a correction in $H^1_0(\Omega)$. Then we also deduce that $K^-\not\equiv0$.

Conditions \eqref{eq:K+b} and \eqref{eq:K-b} mean that $K^+$ and $K^-$, though not belonging to variational spaces $H^1_{\rw,0}(\rP)$ and $H^1_0(\Omega)$, satisfy Dirichlet conditions on $\partial\rP$ and $\partial\Omega$, respectively. 

We also note that in condition \eqref{eq:K+c}, the space $H^1_{\rw,0}(\rP \cap \overline{\sB^\complement(R_0)})$ can be replaced with $H^1_{\rw,0}(\rP \cap \overline{\sB^\complement(\rho_0)})$ for any $\rho_0 \ge R_0$. This is because $h^+$ belongs, by definition, to $H^1_0(\rP \cap \overline{\sA(\rho, \rho')})$ for any $\rho, \rho' \ge R_0$, and $K^+$ belongs to the same space by condition \eqref{eq:K+b}. A similar remark applies  as well to $K^-$.

\begin{proof}[Proof of Lemma {\em\ref{lem:K+K-}}]
a) Since $\Delta h^+=0$ in $\Gamma$, we can verify that the identity $\Delta(\Phi h^+) = h^+\Delta\Phi + 2\nabla h^+\cdot\nabla\Phi$ holds in $\rP$. Hence, $\Delta(\Phi h^+)$ is a function in $L^2(\rP)$ with compact support, defining an element of the dual space of $H^1_{\rw,0}(\rP)$. We introduce the corrector $Y^+$ as the  solution of the variational problem
\begin{equation}
\label{eq:Y+}
   Y^+\in H^1_{\rw,0}(\rP)\quad\mbox{and}\quad \Delta Y^+ = \Delta(\Phi h^+).
\end{equation}
Then we take
\begin{equation}
\label{eq:K+Y+}
   K^+ = \Phi h^+-Y^+.
\end{equation}
By construction, $K^+$ satisfies all conditions \eqref{eq:K+}.

Proving the uniqueness of $K^+$ boils down to showing that if a function $\widetilde K^+$ satisfies conditions \eqref{eq:K+a}, \eqref{eq:K+b}, and \eqref{eq:K+c} with $h^+=0$, then $\widetilde K^+ = 0$. This is evident because, in this case, we have $\widetilde K^+ \in H^1_{\rw,0}(\rP\cap\overline{\sB^\complement(R_0)})$ by \eqref{eq:K+c}, and $\widetilde K^+ \in H^1_0(\rP\cap\overline{\sB(2R_0)})$ by \eqref{eq:K+b}. Thus,  $\widetilde K^+$ is a harmonic function in $H^1_{\rw,0}(\rP)$, and it must be identically zero.

\smallskip
b) The proof for $K^-$ follows the same approach. In this case,  $\Delta(\varphi h^-)$ defines a function in $L^2(\Omega)$, and we can introduce the corrector $Y^-$ as the  solution of the variational problem
\begin{equation}
\label{eq:Y-}
   Y^-\in H^1_0(\Omega)\quad\mbox{and}\quad \Delta Y^- = \Delta(\varphi h^-).
\end{equation}
Then, we set
\begin{equation}
\label{eq:K-Y-}
   K^- = \varphi h^- - Y^-.
\end{equation}
The proof of uniqueness of $K^-$ follows the same lines as that for $K^+$.
\end{proof}

\begin{notation}
\label{not:Kj+Kj-}
For all $j\ge1$, we define:

a) $K^+_j$ as the solution $K^+$ given by Lemma \ref{lem:K+K-} when $h^+=h^+_j$. We have
\begin{equation}
\label{eq:K+jY+j}
   K^+_j = \Phi h^+_j-Y^+_j\,,
\end{equation}
with $Y^+_j\in H^1_{\rw,0}(\rP)$ such that $\Delta Y^+_j = \Delta(\Phi h^+_j)$.

\smallskip
b) $K^-_j$ as the solution $K^-$ given by Lemma \ref{lem:K+K-} when $h^-=h^-_j$. We have
\begin{equation}
\label{eq:K-jY-j}
   K^-_j = \varphi h^-_j-Y^-_j\,,
\end{equation}
with $Y^-_j\in H^1_0(\Omega)$ such that $\Delta Y^-_j = \Delta(\varphi h^-_j)$.
\end{notation}

Although the correctors $Y^+_j$ and $Y^-_j$ depend on the specific choice of the cutoff functions $\Phi$ and $\varphi$, the problems \eqref{eq:K+} and \eqref{eq:K-}, which define $K^+_j$ and $K^-_j$, are independent of $\Phi$ and $\varphi$. Thus, $K^+_j$ and $K^-_j$ also do not depend on the specific choice of the cutoff functions and are, therefore, intrinsic objects of the perturbation domain $\rP$ and the limit domain $\Omega$, respectively. Explicit formulas for $K^+_j$ and $K^-_j$ are available only in specific cases.

\begin{example}
\label{ex:K+K-}
a) If 
\[
   \rP = \Gamma\cap\sB^\complement(R_0)\,,
\]
which means we are removing a circular hole centered at the vertex of the cone $\Gamma$, then we have
\[
   K^+_j = h^+_j - R_0^{\lambda^+_j-\lambda^-_j}h^-_j
   \quad \mbox{in}\quad \rP
\]
for all $j\ge 1$.

\smallskip
b) Similarly, if $\Omega$ is the finite sector $\Gamma\cap\sB(r_0)$, then
\[
   K^-_j = h^-_j - r_0^{\lambda^-_j-\lambda^+_j}h^+_j 
   \quad \mbox{in}\quad \Omega
\]
for all $j\ge 1$.
\end{example}

In general, although explicit formulas may not be available, we can still obtain converging expansions for $K^+_j$ at infinity and $K^-_j$ at the origin. Specifically, by applying Theorem \ref{uUseries} b) to the corrector $Y^+_j$, we find
\begin{equation}
\label{eq:K+jseries}
   K^+_j = h^+_j - \sum_{k\ge1} \mathbf{B}_k(Y^+_j) \,h^-_k
   \quad \mbox{in}\quad \rP\cap\sB^\complement(2R_0),
\end{equation}
and, by applying Theorem \ref{uUseries} a) to the corrector $Y^-_j$, we find
\begin{equation}
\label{eq:K-jseries}
   K^-_j = h^-_j - \sum_{k\ge1} \mathbf{c}_k(Y^-_j) \,h^+_k
   \quad \mbox{in}\quad \Omega\cap\sB(r_0/2),
\end{equation}
Note that the coefficients $\mathbf{B}_k(Y^+_j)$ and $\mathbf{c}_k(Y^-_j)$ are also intrinsic; they do not depend on the choice of cutoff functions $\Phi$ and $\varphi$. The sums $\sum_{k\ge1} \mathbf{B}_k(Y^+_j) \,h^-_k$ and $\sum_{k\ge1} \mathbf{c}_k(Y^-_j) \,h^+_k$ can be considered the \emph{primal singular parts} of $K^+_j$ and $K^-_j$, respectively. In contrast, the terms $h^+_j$ and $h^-_j$ represent the \emph{dual singular parts} of $K^+_j$ and $K^-_j$, respectively.

\begin{remark}
The functions $K^+_j$ and $K^-_j$ are widely used in the literature for extracting coefficients in corner asymptotics  (see, e.g., \cite{MazyaPlamenevskii:1984, DaugeNicaise:1990I, CostabelDaugeYosibash04}). We now explain in detail how this is done for the asymptotics at the vertex of $\Omega$ for an element $u \in \gE_\Omega$. Specifically, we will use a duality formula with $K^-_j$ against $\Delta u$. Since, for $u\in\gE_\Omega$, we have $\Delta u=0$ in $\Omega\cap\overline{\sB(r_0/2)})$, formula \eqref{useries.eq0} yields the asymptotic expansion
\[
   u = \sum_{k \geq 1} \mathbf{c}_k(u) \,h^+_k\,,
\]
which convergences in $H^1_0(\Omega\cap\overline{\sB(r_0/2)})$.
Then, provided that $\Omega$ is sufficiently regular to apply the divergence theorem, we calculate
\[
\begin{aligned}
   \int_\Omega K^-_j\,\Delta u\ \rd x &=
   \int_\Omega \varphi h^-_j\,\Delta u\ \rd x - \int_\Omega Y^-_j\,\Delta u\  \rd x \\
   &= \int_\Omega \varphi h^-_j\,\Delta u\ \rd x - \int_\Omega \Delta Y^-_j\, u\ \rd x \\
   &= \int_\Omega \varphi h^-_j\,\Delta u\ \rd x - \int_\Omega \Delta (\varphi h^-_j)\, u\ \rd x \\
   &= \int_{\Omega\cap\sB^\complement(\rho)} \varphi h^-_j\,\Delta u\ \rd x 
   - \int_{\Omega\cap\sB^\complement(\rho)} \Delta (\varphi h^-_j)\, u\ \rd x \,,
   \qquad \forall\rho\in [0,\tfrac{r_0}{2}].\\
\end{aligned}
\]
Since the only part of the boundary of $\Omega\cap\sB^\complement(\rho)$ where  non-zero traces can appear is $\rho\hat\Gamma$, and  on $\rho\hat\Gamma$ the exterior normal derivative is $-\partial_r$ while $\varphi\equiv1$, we deduce that 
\[
   \int_\Omega K^-_j\,\Delta u\ \rd x =
   \int_{\rho\hat\Gamma} \partial_rh^-_j\,u - h^-_j\,\partial_r u\ \rd\sigma\,.
\]
Replacing $u$ with the convergent sum $\sum_{k \geq 1} \mathbf{c}_k(u) \,h^+_k$ and using the orthogonality of the angular functions $\psi_k$, we obtain
\[
\begin{aligned}
   \int_\Omega K^-_j\,\Delta u\ \rd x &=
   \int_{\rho\hat\Gamma} \partial_rh^-_j\,\mathbf{c}_j(u) \,h^+_j 
   - h^-_j\,\partial_r (\mathbf{c}_j(u) \,h^+_j)\ \rd\sigma \\
&=
   \mathbf{c}_j(u) \int_{\rho\hat\Gamma} \partial_rh^-_j\, \,h^+_j 
   - h^-_j\,\partial_r \,h^+_j\ \rd\sigma \\
&=
   \mathbf{c}_j(u) \,\rho^{\lambda^+_j+\lambda^-_j+n-2} (\lambda^-_j - \lambda^+_j)
   =  -2\mathbf{c}_j(u) \sqrt{\left(1-\frac{n}{2}\right)^2 + \mu_j}\,,
\end{aligned}
\]
where, for the last equality, we used formula \eqref{lambdapm}. Hence, the extraction formula
\[
   \mathbf{c}_j(u) = 
   - \tfrac{1}{2\,\sqrt{\left(1-\frac{n}{2}\right)^2 + \mu_j}}\int_\Omega K^-_j\,\Delta u\ \rd x\,.
\]
Similarly,  for the asymptotics at infinity in $\rP$ of elements $U\in\gE_\rP$, we can derive the following extraction formula:
\[
   \mathbf{B}_j(U) = 
   \tfrac{1}{2\,\sqrt{\left(1-\frac{n}{2}\right)^2 + \mu_j}}\int_\rP K^+_j\,\Delta U\ \rd X\,.
\]
\end{remark}

\subsection{Inner expansion}
We introduce the inner expansion operator
\begin{equation}
\label{eq:IOO}
   \cI_{\Omega,\Omega}[\varepsilon] : \gE_\Omega \ni u \longmapsto 
   (\cH_\varepsilon\Phi) u
   - \cH_\varepsilon\left(\cM_{\rP,\rP}^{-1}\cM_{\rP,\Omega}[\varepsilon]u\right)
   \in H^1_0((\varepsilon\rP)\cap\overline{\sB(r_0/2)}),
\end{equation}
which allows us to compactly write formula \eqref{uepsu} as
\begin{equation}
\label{eq:uepsucomp}
   \ueps = \cH_\varepsilon U_0 + \cI_{\Omega,\Omega}[\varepsilon] \sfu[\varepsilon].
\end{equation}

We recall that any function $u$ of $\gE_\Omega$ has an expansion 
\[
u = \sum_{j \geq 1} \mathbf{c}_j(u) \,h^+_j
\]
as in \eqref{useries.eq0}, and  we wish to show that we have
\begin{equation}\label{eq:IOOh}
\cI_{\Omega,\Omega}[\varepsilon]u = \cI_{\Omega,\Omega}[\varepsilon]\biggl(\sum_{j \geq 1} \mathbf{c}_j(u) \,h^+_j\biggr) = \sum_{j \geq 1} \mathbf{c}_j(u) \,\cI_{\Omega,\Omega}[\varepsilon]h^+_j\,,
\end{equation}
where the last series converges in a suitable sense. This amounts to proving that  the operator  $\cI_{\Omega,\Omega}[\varepsilon]$ expands into a convergent series of bounded operators. In addition, we will prove that the terms $\cI_{\Omega,\Omega}[\varepsilon]h^+_j$ only depend on $j$ and on the shape of the perturbation $\rP$, but not on the specific choice of the cutoffs $\Phi$ and $\varphi$.

 Then, applying \eqref{eq:IOOh} to the expansion 
\[
\sfu[\varepsilon] = \sum_{\sfe\in\sfE^\infty}\varepsilon^{\sfe}\sfu_\sfe
\]
from Theorem \ref{uepsthmfF}, we will derive an expansion of $\ueps$  from \eqref{eq:uepsucomp}.

So, we now focus on the terms $\cI_{\Omega,\Omega}[\varepsilon]h^+_j$.  Recalling that $\cM_{\rP,\Omega}[\varepsilon] = [\Delta_X,\Phi]\circ\cH_{1/\varepsilon}$, we can write
\begin{equation}
\label{eq:IOOIPP}
   \cI_{\Omega,\Omega}[\varepsilon] = \cH_\varepsilon\circ
   \big(\Phi\,\Id - \cM_{\rP,\rP}^{-1} [\Delta_X,\Phi]\big) \circ\cH_{1/\varepsilon}
\end{equation}
where $\Phi\,\Id - \cM_{\rP,\rP}^{-1} [\Delta_X,\Phi]$ is a compact notation for the operator $U\mapsto \Phi\,U - \cM_{\rP,\rP}^{-1}( [\Delta_X,\Phi]U)$.

We observe the following meaningful relation:
\[
   \Delta_X(\Phi\,U - \cM_{\rP,\rP}^{-1} [\Delta_X,\Phi]U) = 
   \Delta_X(\Phi\,U) -  [\Delta_X,\Phi]U =
   \Phi(\Delta_X U)
\]
which implies that if $U$ is harmonic in $\Gamma$ and satisfies Dirichlet conditions on $\partial\Gamma$, then $(\Phi\,\Id - \cM_{\rP,\rP}^{-1} [\Delta_X,\Phi])U$ is harmonic on $\rP$ and satisfies Dirichlet conditions on $\partial\rP$.

Taking for $U$ any of the homogeneous functions $h^+_j$, which are harmonic in $\Gamma$, we note that for the canonical harmonic functions $K^+_j$ and their correctors $Y^+_j$, we have 
\[
   Y^+_j := \cM_{\rP,\rP}^{-1}  [\Delta_X,\Phi]h^+_j \quad\mbox{and}\quad
   K^+_j = \big(\Phi\,\Id - \cM_{\rP,\rP}^{-1} [\Delta_X,\Phi]\big) h^+_j
\]
(see Notation \ref{not:Kj+Kj-} a)). We deduce that 
\[
\cI_{\Omega,\Omega}[\varepsilon]h^+_j=\varepsilon^{\lambda^+_j}\cH_\varepsilon K^+_j\,.
\]
Since the $K^+_j$ are intrinsic objects of $\rP$, the last equality implies that $\cI_{\Omega,\Omega}[\varepsilon]h^+_j$ only depends on $j$ and $\rP$.

We now proceed to prove some bounds for the $L^2$ norms of $K^+_j$ and $\nabla K^+_j$, which in turn provide us with a bound for the $H^1$ norm of $\cH_\varepsilon K^+_j$, and hence of $\cI_{\Omega,\Omega}[\varepsilon] h^+_j$. 

\begin{lemma}
\label{lem:Kj+}
Let $R\ge 2R_0$ and $j\ge1$.

(i) We have the following bounds for $K^+_j$ in $\rP\cap\sB(R)$, with a constant $C$ independent of $j$ and $R$:
\begin{subequations}
\begin{align}
\label{eq:K+jH1}
   \DNorm{\nabla K^+_j}{L^2(\rP\cap\sB(R))} &\le C
   (\lambda^+_j+1)^{1/2}\, R^{\lambda^+_j-1+\frac{n}{2}} , \\
\label{eq:K+jL2}
   \DNorm{K^+_j}{L^2(\rP\cap\sB(R))} &\le C(\lambda^+_j+1)^{1/2}\, 
   R^{\lambda^+_j+\frac{n}{2}} .
\end{align}
\end{subequations}

(ii) Let $\varepsilon\in(0,\frac{\varepsilon_0}{4})$ and $\rho\in[2\varepsilon R_0, r_0/2]$.
Then we have the following bound for $\cH_\varepsilon K^+_j$ in $\varepsilon\rP\cap\sB(\rho)$, with a constant $C$ independent of $\varepsilon$, $j$ and $\rho$:
\begin{equation}
\label{eq:HepsKj+}
   \DNorm{ \cH_\varepsilon K^+_j}{H^1(\varepsilon\rP\cap\sB(\rho))} \le C
   (\lambda^+_j+1)^{1/2}\, \left(\frac{\rho}{\varepsilon}\right)^{\lambda^+_j} .
\end{equation}
\end{lemma}

\begin{proof}
In what follows, we denote by the same letter $C$ some constants that are independent of $j$ and $R$, but may depend on $\Phi$, $R_0$, and $\rP$, and may vary across different inequalities.

{\em(i)} We consider separately the two terms of $K^+_j$, namely $\Phi h^+_j$ and $Y^+_j$.

\smallskip
$\blacktriangleright\:$ Concerning the term $\Phi h^+_j$, we observe that:
\[
   \DNorm{\Phi h^+_j}{L^2(\rP\cap\sB(R))} \le C \DNorm{h^+_j}{L^2(\Gamma\cap\sB(R))}\,.
\]
An explicit calculation yields
\[
   \DNormc{h^+_j}{L^2(\Gamma\cap\sB(R))} = \frac{1}{2\lambda^+_j+n}\, R^{2\lambda^+_j+n}
   \; (\;\le R^{2\lambda^+_j+n})\,.
\]
Hence,
\begin{equation}\label{lem:Kj+.eq(1)}
   \DNorm{\Phi h^+_j}{L^2(\rP\cap\sB(R))} \le C\, 
   R^{\lambda^+_j+\frac{n}{2}}.
\end{equation}
Similarly,
\[
   \DNorm{\nabla(\Phi h^+_j)}{L^2(\rP\cap\sB(R))} \le C 
   \big(\DNorm{\nabla h^+_j}{L^2(\Gamma\cap\sB(R))} 
   + \DNorm{h^+_j}{L^2(\Gamma\cap\sA(R_0,2R_0))}\big).
\]
Hence, using \eqref{h+inH1} and the formula for $\DNormc{h^+_j}{L^2(\Gamma\cap\sB(2R_0))}$,
\[
   \DNorm{\nabla(\Phi h^+_j)}{L^2(\rP\cap\sB(R))} \le C \left(
   \sqrt{\lambda^+_j}\, R^{\lambda^+_j-1+\frac{n}{2}} +
    (2R_0)^{\lambda^+_j+\frac{n}{2}} \right)\,.
\]
Since $R \ge 2R_0$, we can write $(2R_0)^{\lambda^+_j+\frac{n}{2}} \le
2R_0 \,R^{\lambda^+_j-1+\frac{n}{2}}$ and deduce from the previous inequality that
\begin{equation}\label{lem:Kj+.eq(2)}
   \DNorm{\nabla(\Phi h^+_j)}{L^2(\rP\cap\sB(R))} \le C 
   (\lambda^+_j + 1)^{1/2}\, R^{\lambda^+_j-1+\frac{n}{2}} .
\end{equation}

\smallskip
$\blacktriangleright\:$  Concerning the term $Y^+_j$, set $C^+_j := [\Delta_X, \Phi]h^+_j$ so that $\cM_{\rP,\rP}^{-1}C^+_j = Y^+_j$. Relying on \eqref{isoDX}, we start from the resolvent estimate:
\[
   \DNorm{Y^+_j}{H^1_\rw(\rP)} \le \DNorm{Y^+_j}{\gE_\rP} =
   \DNorm{\cM_{\rP,\rP}^{-1}C^+_j}{\gE_\rP} \le C
   \DNorm{C^+_j}{\gF_\rP} = C
   \DNorm{C^+_j}{L^2(\rP)}\,.
\]
But, since $C^+_j = [\Delta_X, \Phi]h^+_j$, we have:
\[
\begin{aligned}
   \DNorm{C^+_j}{L^2(\rP)} 
   &\le C \DNorm{h^+_j}{H^1(\Gamma \cap \sA(R_0, 2R_0))} \\
   &\le C \left(
   \sqrt{\lambda^+_j}\, (2R_0)^{\lambda^+_j-1+\frac{n}{2}} +
    (2R_0)^{\lambda^+_j+\frac{n}{2}} \right) \\
   &\le C (\lambda^+_j + 1)^{1/2}\, (2R_0)^{\lambda^+_j-1+\frac{n}{2}} .
\end{aligned}
\]
Observing that for any element $w$ of $H^1_\rw(\rP)$ there holds
\[
   \DNormc{w}{H^1_\rw(\rP)} \ge \DNormc{\nabla w}{L^2(\rP \cap \sB(R))} +
   \frac{1}{R^2 + 1} \DNormc{w}{L^2(\rP \cap \sB(R))}
\]
we deduce for $Y^+_j$:
\begin{equation}\label{lem:Kj+.eq(3)}
   \DNorm{\nabla Y^+_j}{L^2(\rP \cap \sB(R))} +
   \frac{1}{R + 1} \DNorm{Y^+_j}{L^2(\rP \cap \sB(R))} \le
   C (\lambda^+_j + 1)^{1/2}\, (2R_0)^{\lambda^+_j-1+\frac{n}{2}} .
\end{equation}

\smallskip
$\blacktriangleright\:$  Putting \eqref{lem:Kj+.eq(1)}, \eqref{lem:Kj+.eq(2)}, and \eqref{lem:Kj+.eq(3)} together, and using the fact that $R\ge 2R_0$, we obtain the bounds \eqref{eq:K+jH1} and \eqref{eq:K+jL2}.

\medskip
{\em (ii)} To obtain the bound \eqref{eq:HepsKj+} for $K^+_j$ as a function of the rapid variable ${x}/{\varepsilon}$, it suffices to use estimates \eqref{eq:K+jH1} and \eqref{eq:K+jL2} with $R = {\rho}/{\varepsilon} $, which is greater than or equal to $2R_0$ due to the assumption that $\rho \ge 2R_0\varepsilon$, and then use the change of variables $x = \varepsilon X$ in integrals and derivatives.
\end{proof}

From estimate \eqref{eq:HepsKj+}  
we readily deduce a bound for  $\cI_{\Omega,\Omega}[\varepsilon]h^+_j=\varepsilon^{\lambda^+_j}\cH_\varepsilon K^+_j$, which leads to a representation of the inner expansion operator $\cI_{\Omega,\Omega}$ as a normally convergent series:

\begin{proposition}
\label{th:IOO}
Let $\mathbf{c}_j$ be the trace operator defined in \eqref{cj}.  
The inner expansion operator $\cI_{\Omega,\Omega}[\varepsilon]$ \eqref{eq:IOO}--\eqref{eq:IOOIPP} satisfies 
\begin{equation}
\label{eq:IOOu}
   \forall u\in\gE_\Omega,\quad
   \cI_{\Omega,\Omega}[\varepsilon] u =  
   \sum_{j\ge 1} \varepsilon^{\lambda^+_j} \mathbf{c}_j(u)\; \cH_\varepsilon K^+_j,
\end{equation}
with normal convergence in $H^1$-norm in the region $\Omega_\varepsilon\cap\sB({r_0}/{2})$ in the following sense: 
For any $\varepsilon_1\in(0,\frac{\varepsilon_0}{4})$, for any $r_1\in(0,\frac{r_0}{2})$, there exists a constant $C$ such that
\begin{equation}
\label{eq:IOOunorm}
   \forall\varepsilon\in(0,\varepsilon_1],\quad
   \forall u\in\gE_\Omega,\quad
   \sum_{j\ge 1} \varepsilon^{\lambda^+_j} |\mathbf{c}_j(u)|\; 
   \DNorm{\cH_\varepsilon K^+_j}{H^1(\Omega_\varepsilon\cap\sB(r_1))}  \le
   C \DNorm{u}{\gE_\Omega}.
\end{equation}
\end{proposition}

\begin{proof}
By Theorem \ref{uUseries} {\em a)}, the harmonic function $u$ can be expanded in $H^1_0(\Gamma \cap \overline{\sB(r_0/2)})$ as a series of the homogeneous harmonic functions $h^+_j$, as follows:
\[
   u = \sum_{j \ge 1} \mathbf{c}_j(u) h^+_j.
\]
On one hand, using the homogeneity of the $h^+_j$, we obtain
\[
   (\cH_{1/\varepsilon} u)(X) = 
   \sum_{j \ge 1} \mathbf{c}_j(u) h^+_j(\varepsilon X) = 
   \sum_{j \ge 1} \varepsilon^{\lambda^+_j} \mathbf{c}_j(u) h^+_j(X).
\]
On the other hand, we can use Theorem \ref{CjBjopnorm} to write, for $\varepsilon \in \left(0, \frac{\varepsilon_0}{4}\right)$,
\[
   \cM_{\rP,\Omega}[\varepsilon]u =
   \sum_{j \ge 1} \varepsilon^{\lambda^+_j} \mathbf{c}_j(u) [\Delta_X,\Phi] h^+_j.
\]
This leads immediately to the expansion \eqref{eq:IOOu} of $\cI_{\Omega,\Omega}[\varepsilon]u$:
\[
\begin{aligned}
   \cI_{\Omega,\Omega}[\varepsilon] u &=  
   \sum_{j\ge 1} \varepsilon^{\lambda^+_j} \mathbf{c}_j(u)\;
   \cH_\varepsilon\Big(\Phi h^+_j - \cM_{\rP,\rP}^{-1}  [\Delta_X,\Phi]h^+_j\Big) \\ &=  
   \sum_{j\ge 1} \varepsilon^{\lambda^+_j} \mathbf{c}_j(u)\;
   \cH_\varepsilon  K^+_j.
\end{aligned}
\]
Let us check the normal convergence.
Using \eqref{useries.conv} (with the specific choice $\rho=r_0/2$), we deduce that
\[
   |\mathbf{c}_j(u)| \le (\lambda^+_j)^{-1/2} \left(\frac{r_0}{2}\right)^{-\lambda^+_j-1+n/2}
   \DNorm{\nabla u}{L^2(\Gamma\cap\sB(r_0/2))}\,.
\]
Let $\varepsilon\in(0,\frac{\varepsilon_0}{4})$. 
Combining the inequality above with \eqref{eq:HepsKj+} for $\rho\in[2\varepsilon R_0, r_0/2)$ leads to
\[
\begin{aligned}
   \sum_{j\ge 1} \varepsilon^{\lambda^+_j} |\mathbf{c}_j(u)|\; &
   \DNorm{\cH_\varepsilon K^+_j}{H^1(\varepsilon\rP\cap\sB(\rho))} \\ 
   &\le
   \left(\frac{r_0}{2}\right)^{-1+n/2} \DNorm{\nabla u}{L^2(\Gamma\cap\sB(r_0/2))}
   \sum_{j\ge 1} \varepsilon^{\lambda^+_j}
   \frac{(\lambda^+_j+1)^{1/2}}{(\lambda^+_j)^{1/2}}\, 
   \left(\frac{r_0}{2}\right)^{-\lambda^+_j}
   \left(\frac{\rho}{\varepsilon}\right)^{\lambda^+_j} \\
   &\le
   C \DNorm{u}{\gE_\Omega}
   \sum_{j\ge 1} \left(\frac{2\rho}{r_0}\right)^{\lambda^+_j} .
\end{aligned}
\]
As $\rho<r_0/2$, Lemma \ref{lem:convpol} yields the convergence of the latter series, from which we deduce inequality \eqref{eq:IOOunorm}.
\end{proof}

\begin{remark}
From Proposition \ref{th:IOO}, we see more clearly how a  function $u$, harmonic near the vertex of $\Gamma$ and with Dirichlet conditions on $\partial\Gamma$, is transformed by $\cI_{\Omega,\Omega}[\varepsilon]$ into another harmonic function with Dirichlet conditions on $\varepsilon\partial\rP$ through the formulas
\[
   u = \sum_{j\ge 1} \varepsilon^{\lambda^+_j} \mathbf{c}_j(u) \;\cH_\varepsilon h^+_j
   \quad\mbox{and}\quad
   \cI_{\Omega,\Omega}[\varepsilon] u =
   \sum_{j\ge 1} \varepsilon^{\lambda^+_j} \mathbf{c}_j(u)\; \cH_\varepsilon K^+_j.
\]
\end{remark}

We are now ready to prove Theorem \ref{th:uepsinner}, where we finally apply $\cI_{\Omega,\Omega}[\varepsilon]$ to  the function $\sfu[\varepsilon]=\sum_{\sfe\in\sfE^\infty}\varepsilon^{\sfe}\sfu_\sfe$ from Theorem \ref{uepsthmfF}.

\begin{theorem}\label{th:uepsinner} 
Let $f\in\gF_\Omega$, $F\in\gF_\rP$, and let $\ueps$ be the solution of problem \eqref{PoissonfF}.
Under the assumptions of Theorem {\em\ref{uepsthmfF}}, and with $K^+_j$ being the canonical harmonic function introduced in Definition {\em\ref{not:Kj+Kj-}} a),
we have for all $0<\varepsilon<\varepsilon_\star$:
\begin{equation}\label{uepsinner.eq1}
   \ueps(x) = U_0\left(\frac{x}{\varepsilon}\right)
   + \sum_{j\ge 1} \sum_{\sfe\in\sfE^\infty} \varepsilon^{\sfe+\lambda^+_j}
   \mathbf{c}_j(\sfu_\sfe)\, K^+_j\left(\frac{x}{\varepsilon}\right)\,,
   \quad \mbox{a.e.}\;\; x\in\Omega_\varepsilon\cap\sB(r_0/2)
\end{equation}
where $\sfu_\sfe$ are the coefficients in the series $\sum_{\sfe\in\sfE^\infty}\varepsilon^{\sfe}\sfu_\sfe = \sfu[\varepsilon]$, which is the slow part in the global expansion \eqref{eq:uepsglo}.  Moreover, for all  $r_1\in(0,r_0/2)$  and all $\varepsilon_1\in(0,\varepsilon_\star)$  the series \eqref{uepsinner.eq1} converges normally in $H^1(\Omega_\varepsilon\cap\sB(r_1))$ with the estimates
\begin{equation}\label{uepsinner.eq2}
   \sum_{j\ge 1} \sum_{\sfe\in\sfE^\infty} 
   \varepsilon^{\sfe+\lambda^+_j}|\mathbf{c}_j(\sfu_\sfe)| 
   \left\|K^+_j\left(\frac{\cdot}{\varepsilon}\right)\right\|
   _{H^1(\Omega_\varepsilon\cap{\sB(r_1)})}  \le
   C \left(\DNorm{f}{\gF_\Omega} + \DNorm{F}{\gF_\rP}\right)\,,
\end{equation}
with $C$ a constant independent of $\varepsilon$, $f$, and $F$.
\end{theorem}

\begin{proof}
Recall that from \eqref{eq:uepsucomp} we have 
$\ueps = \cH_\varepsilon U_0 + \cI_{\Omega,\Omega}[\varepsilon] \sfu[\varepsilon]$,
with the series $\sfu[\varepsilon]=\sum_{\sfe\in\sfE^\infty}\varepsilon^{\sfe}\sfu_\sfe$ normally converging in $\gE_\Omega$. Then, the result of the Theorem follows by linearity from Proposition \ref{th:IOO}. Also note that the series \eqref{uepsinner.eq2} is bounded by $C\sum_{\sfe\in\sfE^\infty}\varepsilon^{\sfe}\DNorm{\sfu_\sfe}{\gE_\Omega}$.
\end{proof}

\begin{remark}
Owing to the assumptions on the support of the right-hand side of equation \eqref{PoissonfF}, the inner expansion is rather simple: $\ueps - \cH_\varepsilon U_0$ appears as a sum in $j$ of the canonical harmonic functions $K^+_j$ with scalar coefficients $\sum_{\sfe\in\sfE^\infty} \varepsilon^{\sfe+\lambda^+_j} \mathbf{c}_j(\sfu_\sfe)$ that are generalized power series. Analogous structures can be more or less easily recognized in other works. In \cite[Sec.\,5]{DaToVi10}, which considers the same problem in dimension $2$, a reorganization of the constitutive terms of the inner expansion would yield similar formulas. In \cite[Th.\,4.7]{CaCoDaVi06}, which considers an angle with a coating of width $\varepsilon$, canonical piecewise harmonic profiles denoted there by $\mathfrak{K}^\lambda$ are a building block of asymptotic expansions. In \cite{Josien:2024}, which addresses stochastic homogenization in a plane sector, this structure of canonical profile functions associated with standard singular functions (denoted there $\bar\tau_n$) appears clearly as early as eqs.\,(6) and (7) in the form of the sum $\bar\tau_n+\phi^{\mathfrak C}_n$ with the corrector $\phi^{\mathfrak C}_n$. By contrast, the very broad generality of the impressive \cite[Chap.\,4]{MaNaPl00i} makes it difficult to identify constitutive structures inside asymptotic expansions.
\end{remark}

\subsection{Outer expansion}\label{ss:outer}
The strategy to obtain the outer expansion, away from the vertex of the cone, is similar to the one used for the inner expansion. However, we adopt a different perspective. Instead of starting from \eqref{uepsU}, we begin with formula \eqref{uepsu}, with $x$ in the region $\Omega_\varepsilon \cap \sB^\complement(\varepsilon R_0) = \Omega \cap \sB^\complement(\varepsilon R_0)$ (where $\cH_\varepsilon\Phi \equiv 1$):
\begin{equation}
\label{eq:uepsU}
   \ueps(x) = u_0(x)
   +\varphi(x) \cH_\varepsilon \sfU[\varepsilon](x)
   -\left(\cM_{\Omega,\Omega}^{-1}\cM_{\Omega,\rP}[\varepsilon]\sfU[\varepsilon]\right)(x).
\end{equation}
Recalling that $\cM_{\Omega,\rP}[\varepsilon] = [\Delta,\varphi]\circ\cH_\varepsilon$, we introduce the outer expansion operator
\begin{equation}
\label{eq:OOP}
   \cO_{\Omega,\rP}[\varepsilon] = (\varphi\Id 
   - \cM_{\Omega,\Omega}^{-1} [\Delta,\varphi]) \circ \cH_\varepsilon,
\end{equation}
so that 
\begin{equation}
\label{eq:uepsUcomp}
   \ueps = u_0 + \cO_{\Omega,\rP}[\varepsilon] \sfU[\varepsilon].
\end{equation}
We wish to prove that $\cO_{\Omega,\rP}[\varepsilon]$ expands into a converging series of bounded operators and apply it to the series expansion 
\[
\sfU[\varepsilon]=\sum_{\sfe\in\sfE^\infty}\varepsilon^{\sfe}\sfU_\sfe
\]
from Theorem \ref{uepsthmfF} to obtain an expansion of $\ueps$. Since $\Delta(\varphi\Id - \cM_{\Omega,\Omega}^{-1}[\Delta,\varphi])=\varphi\Delta$, for the canonical harmonic functions $K^-_j$ and their correctors $Y^-_j$, we find:
\[
   Y^-_j = \cM_{\Omega,\Omega}^{-1}[\Delta,\varphi] h^-_j \quad\mbox{and}\quad
   K^-_j = \big(\varphi\Id - \cM_{\Omega,\Omega}^{-1}[\Delta,\varphi]\big) h^-_j 
\]
(see Notation \ref{not:Kj+Kj-} b)).

The proof of the following estimates for $K^-_j$ follows the same lines as in Lemma \ref{lem:Kj+} and is left to the reader. (It is even simpler here because the variational space is $H^1_0(\Omega)$, so we do not have to deal with weights.)

\begin{lemma}
\label{lem:Kj-}
Let $r\le r_0/2$ and $j\ge1$ be an integer.
We have the following bounds for $K^-_j$ in $\Omega\cap\sB^\complement(r)$, with a constant $C$ independent of $j$ and $R$:
\begin{equation}
\label{eq:normKj-}
   \DNorm{K^-_j}{H^1(\Omega\cap\sB^\complement(r))} \le C
   (|\lambda^-_j|+1)^{1/2}\, r^{\lambda^-_j-1+\frac{n}{2}} .
\end{equation}
\end{lemma}

We are now ready to show that the outer expansion operator $\cO_{\Omega,\rP}[\varepsilon]$ can be expressed as a convergent series of bounded operators.

\begin{proposition}
\label{th:OOP}
Let $\mathbf{B}_j$ be the trace operator defined in \eqref{Bj}. 
The outer expansion operator $\cO_{\Omega,\rP}[\varepsilon]$ \eqref{eq:OOP} satisfies 
\begin{equation}
\label{eq:OOPU}
   \forall U\in\gE_\rP,\quad
   \cO_{\Omega,\rP}[\varepsilon] U =  
   \sum_{j\ge 1} \varepsilon^{-\lambda^-_j} \mathbf{B}_j(U)\; K^-_j,
\end{equation}
with normal convergence in $H^1$-norm in the region $\Omega_\varepsilon\cap\sB^\complement(2\varepsilon R_0)$ in the following sense: 
For any $\varepsilon_1\in(0,\varepsilon_0/4)$ and for any $R_1>2R_0$, there exists a constant $C$ such that
\begin{equation}
\label{eq:OOPUnorm}
   \forall\varepsilon\in(0,\varepsilon_1],\quad
   \forall U\in\gE_\rP,\quad
   \sum_{j\ge 1} \varepsilon^{-\lambda^-_j} |\mathbf{B}_j(U)|\; 
   \DNorm{K^-_j}{H^1(\Omega_\varepsilon\cap\sB^\complement(\varepsilon R_1))}  \le
   C \DNorm{U}{\gE_\rP}.
\end{equation}
\end{proposition}

\begin{proof}
Though similar to that of Proposition \ref{th:IOO}, the proof is given for completeness:
By part {\em b)} of Theorem \ref{uUseries}, the harmonic function $U$ can be expanded in $H^1_{\rw,0}(\Gamma \cap \overline{\sB^\complement(2R_0)})$ as a series of the homogeneous harmonic functions $h^-_j$ according to:
\[
   U = \sum_{j\ge 1} \mathbf{B}_j(U) h^-_j.
\]
On one hand, using the homogeneity of the functions $h^-_j$, we obtain
\[
   (\cH_{\varepsilon} U)(x) = 
   \sum_{j\ge 1} \mathbf{B}_j(u) h^-_j\left(\frac{x}{\varepsilon}\right) = 
   \sum_{j\ge 1} \varepsilon^{-\lambda^-_j} \mathbf{B}_j(U) h^-_j(x).
\]
On the other hand, we use Theorem  \ref{CjBjopnorm} to write, for $\varepsilon\in(0,\frac{\varepsilon_0}{4})$,
\[
   \cM_{\Omega,\rP}[\varepsilon]U =
   \sum_{j\ge 1} \varepsilon^{-\lambda^-_j} \mathbf{B}_j(U) [\Delta_x,\varphi]h^-_j.
\]
This leads to the expansion \eqref{eq:OOPU} of $\cO_{\Omega,\rP}[\varepsilon]U$:
\[
\begin{aligned}
   \cO_{\Omega,\rP}[\varepsilon] U &=  
   \sum_{j\ge 1} \varepsilon^{-\lambda^-_j} \mathbf{B}_j(U)\;
   \Big(\varphi h^-_j - \cM_{\Omega,\Omega}^{-1}  [\Delta_x,\varphi]h^-_j\Big) \\ &=  
   \sum_{j\ge 1} \varepsilon^{-\lambda^-_j} \mathbf{B}_j(U)\;
   K^-_j.
\end{aligned}
\]
Let us check the normal convergence.
Using \eqref{Useries.conv} (with $\rho=2R_0$) we deduce that
\[
   |\mathbf{B}_j(U)| \le |\lambda^-_j|^{-1/2} (2R_0)^{-\lambda^-_j-1+n/2}
   \DNorm{\nabla U}{L^2(\Gamma\cap\sB^\complement(2R_0))}.
\]
Let $\varepsilon\in(0,\frac{\varepsilon_0}{4})$. 
Combining the inequality above with \eqref{eq:normKj-} for $\rho\in(2\varepsilon R_0, r_0/2]$ leads to
\[
\begin{aligned}
   \sum_{j\ge 1} \varepsilon^{-\lambda^-_j} |\mathbf{B}_j(U)|\; &
   \DNorm{K^-_j}{H^1(\Omega\cap\sB^\complement(\rho))} \\ 
   &\hskip-1em\le
   \DNorm{\nabla U}{L^2(\Gamma\cap\sB^\complement(2R_0))}
   \sum_{j\ge 1} \varepsilon^{-\lambda^-_j}
   |\lambda^-_j|^{-1/2} (2R_0)^{-\lambda^-_j-1+n/2}
   \DNorm{K^-_j}{H^1(\Omega\cap\sB^\complement(\rho))}\\
   &\hskip-1em\le
   \DNorm{U}{\gE_\rP}
   \sum_{j\ge 1} \varepsilon^{-\lambda^-_j}
   \frac{(|\lambda^-_j|+1)^{1/2}}{|\lambda^-_j|^{1/2}}\, 
   (2R_0)^{-\lambda^-_j-1+n/2} \rho^{\lambda^-_j-1+\frac{n}{2}}\\
   &\hskip-1em\le
   C (2R_0\rho)^{-1+\frac{n}{2}} \DNorm{U}{\gE_\rP}
   \sum_{j\ge 1} \left(\frac{2R_0\varepsilon}{\rho}\right)^{-\lambda^-_j} .
\end{aligned}
\]
Since $\rho>2\varepsilon R_0$, Lemma \ref{lem:convpol} yields the convergence of the latter series, from which we deduce inequality \eqref{eq:OOPUnorm}.
\end{proof}

Equality \eqref{eq:uepsUcomp} can be written as $\ueps = u_0+\cO_{\Omega,\rP}\left( \sum_{\sfe\in\sfE^\infty}\varepsilon^\sfe \sfU_\sfe \right)$, which, thanks to the description of the operator $\cO_{\Omega,\rP}$ as a normally converging series in Proposition \ref{th:OOP}, yields the following outer expansion for $\ueps$:

\begin{theorem}\label{th:uepsouter}
Let $f\in\gF_\Omega$, $F\in\gF_\rP$, and let $\ueps$ be the solution of problem \eqref{PoissonfF}.
Under the assumptions of Theorem {\em\ref{uepsthmfF}}, and with $K^-_j$ being the canonical harmonic function introduced in Definition {\em\ref{not:Kj+Kj-}} b),
we have for all $0<\varepsilon<\varepsilon_\star$:
\begin{equation}\label{uepsouter.eq1}
   \ueps(x) = u_0(x)
   + \sum_{j\ge 1} \sum_{\sfe\in\sfE^\infty} \varepsilon^{\sfe-\lambda^-_j}
   \mathbf{B}_j(\sfU_\sfe)\, K^-_j(x)\,,
   \quad \mbox{a.e.}\;\; x\in\Omega_\varepsilon\cap\sB^\complement(2R_0\varepsilon)
\end{equation}
where $\sum_{\sfe\in\sfE^\infty}\varepsilon^{\sfe}\sfU_\sfe = \sfU[\varepsilon]$ is the rapid part in the global expansion \eqref{eq:uepsglo}.  Moreover, for all  $R_1>2R_0$  and all $\varepsilon_1\in(0,\varepsilon_\star)$  the series \eqref{uepsouter.eq1} converges normally in $H^1(\Omega_\varepsilon\cap{\sB^\complement(\varepsilon R_1)})$ with the estimates
\begin{equation}\label{uepsouter.eq2}
   \sum_{j\ge 1} \sum_{\sfe\in\sfE^\infty} 
   \varepsilon^{\sfe-\lambda^-_j}|\mathbf{B}_j(\sfU_\sfe)| 
   \DNorm{K^-_j}
   {H^1(\Omega_\varepsilon\cap{\sB^\complement(\varepsilon R_1)})} \le
   C \left(\DNorm{f}{\gF_\Omega} + \DNorm{F}{\gF_\rP}\right)\,,
\end{equation}
with $C$ a constant independent of $\varepsilon$, $f$, and $F$.
\end{theorem}

\subsection{Conclusion: Comparison of expansions}
In conclusion of this paper, we now compare the three expansions \eqref{eq:uepsglo} (global), \eqref{uepsinner.eq1} (inner), and \eqref{uepsouter.eq1} (outer). As a result, we will demonstrate that the coefficients of these expansions can be computed through a recursive procedure involving a finite number of numerical operations, starting from the solutions $u_0$ and $U_0$ of the limiting problems \eqref{eq:u0U0}. While recursive procedures of this kind are familiar in the framework of the Matched Asymptotic Expansion method, our approach is different: we first establish a convergent global expansion and then derive the recursive algorithm. This method has the advantage of identifying interaction matrices that remain the same at each step of the procedure.

The comparison of the global, inner, and outer expansions is made possible by their convergence in the common annular transition region
\[
   \sA_\varepsilon\cap \Gamma\quad\mbox{with}\quad
   \sA_\varepsilon := \sA(2R_0\varepsilon ,r_0/2)\,.
\]
Indeed, for any $\varepsilon\in(0,\varepsilon_0/4]$ we have identities
\[
   \sA_\varepsilon \cap\Gamma = 
   \sA_\varepsilon \cap\Omega = 
   \sA_\varepsilon \cap\varepsilon\rP = 
   \sA_\varepsilon \cap\Omega_\varepsilon\,,
\]
and this transition region is the intersection between the domains of validity of inner and outer expansions.

 Let $\varepsilon_1\in(0,\varepsilon_\star)$ and let $\rho\in(2R_0\varepsilon_1,r_0/2)$ (so that $\rho\in(2R_0\varepsilon,r_0/2)$ for all $\varepsilon\in[0,\varepsilon_1]$). Choose $\varepsilon\in(0,\varepsilon_1]$ and
let $x\in\rho\hat\Gamma$. Then $x$ belongs to the transition region and moreover,  $\varphi(x)=\Phi(\frac{x}{\varepsilon})=1$.  With $f\in\gF_\Omega$ and $F\in\gF_\rP$, the global, inner, and outer expansions \eqref{eq:uepsglo}, \eqref{uepsinner.eq1}, and \eqref{uepsouter.eq1} hold simultaneously at $x$. Accordingly, we have 
\[
\begin{aligned}
   \ueps(x) &=
   \sum_{\sfe\in\sfE^\infty} \varepsilon^{\sfe}\sfu_\sfe(x)
   +\sum_{\sfe\in\sfE^\infty} \varepsilon^{\sfe}
   \sfU_\sfe\left(\frac{x}{\varepsilon}\right)\\
    &= U_0\left(\frac{x}{\varepsilon}\right)
   + \sum_{j\ge 1} \sum_{\sfe\in\sfE^\infty} \varepsilon^{\sfe+\lambda^+_j}
   \mathbf{c}_j(u_{\sfe})\, K^+_j\left(\frac{x}{\varepsilon}\right)\,,\\
    &= u_0(x)
   + \sum_{j\ge 1} \sum_{\sfe\in\sfE^\infty} \varepsilon^{\sfe-\lambda^-_j}
   \mathbf{B}_j(U_{\sfe})\, K^-_j(x)\,.
\end{aligned}
\]
We can compare these three representations by expanding each of them in the family of homogeneous harmonic functions 
\[
   \{h^+_k,\quad k\ge1\} \cup \{h^-_k,\quad k\ge1\}\,.
\]
 Since the $h^+_k$ and $h^-_k$ are independent in $L^2(\sA_\varepsilon\cap\Gamma)$, we can identify the coefficients of the $h^+_k$ and $h^-_k$ and obtain relations between the coefficients $\mathbf{c}_j(u_{\sfe})$ and $\mathbf{B}_j(U_{\sfe})$ of the original series (see Lemma \ref{lem:idcoef} below).
Owing to Theorem \ref{uUseries} and expansions \eqref{eq:K+jseries} and \eqref{eq:K-jseries} of $K^\pm_j$, the functions $\sfu_\sfe$, $\sfU_\sfe$, $K^+_j$, and $K^-_j$ can be expanded 
as follows:
\[
\begin{aligned}
   \sfu_\sfe(x) 
   &= \sum_{k \geq 1} \mathbf{c}_k(\sfu_\sfe) \,h^+_k(x) 
   \\
   \sfU_\sfe\left(\frac{x}{\varepsilon}\right) 
   &= \sum_{k \geq 1} \mathbf{B}_k(\sfU_\sfe) \,h^-_k\left(\frac{x}{\varepsilon}\right) 
   = \sum_{k \geq 1} \varepsilon^{-\lambda^-_k}\,\mathbf{B}_k(\sfU_\sfe) \,h^-_k(x) 
   \\
   K^+_j\left(\frac{x}{\varepsilon}\right)
   &= h^+_j\left(\frac{x}{\varepsilon}\right)
   - \sum_{k \geq 1} \mathbf{B}_k(Y^+_j) \,h^-_k\left(\frac{x}{\varepsilon}\right) 
   = \varepsilon^{-\lambda^+_j}\,h^+_j(x)
   - \sum_{k \geq 1} \varepsilon^{-\lambda^-_k}\,\mathbf{B}_k(Y^+_j) \,h^-_k(x) 
   \\
   K^-_j(x) 
   &= h^-_j(x)
   - \sum_{k \geq 1} \mathbf{c}_k(Y^-_j) \,h^+_k(x) .
\end{aligned}
\]
Coming back to $\ueps$ we obtain
\[\small
\begin{aligned}
   \ueps(x) &=
   \sum_{\sfe\in\sfE^\infty} 
   \left(\sum_{k \geq 1} \varepsilon^{\sfe}\mathbf{c}_k(\sfu_\sfe) \,h^+_k(x) 
   +\sum_{k \geq 1} \varepsilon^{\sfe-\lambda^-_k}\,\mathbf{B}_k(\sfU_\sfe) \,h^-_k(x)\right)
   \\
    &= \sum_{k \geq 1} \varepsilon^{-\lambda^-_k}\,\mathbf{B}_k(U_0) \,h^-_k(x) 
   + \sum_{j\ge 1} \sum_{\sfe\in\sfE^\infty} \varepsilon^{\sfe+\lambda^+_j}
   \mathbf{c}_j(u_{\sfe})\, \left(
   \varepsilon^{-\lambda^+_j}\,h^+_j(x)
   - \sum_{k \geq 1} \varepsilon^{-\lambda^-_k}\,\mathbf{B}_k(Y^+_j) \,h^-_k(x)  \right)\\
    &= \sum_{k \geq 1} \mathbf{c}_k(u_0) \,h^+_k(x)
   + \sum_{j\ge 1} \sum_{\sfe\in\sfE^\infty} \varepsilon^{\sfe-\lambda^-_j}
   \mathbf{B}_j(U_{\sfe})\, \left(h^-_j(x)
   - \sum_{k \geq 1} \mathbf{c}_k(Y^-_j) \,h^+_k(x) \right)\,.
\end{aligned}
\]
Recalling that all these sums are normally convergent, we can reorganize them as follows:
\[\small
\begin{aligned}
   \ueps(x) &=
   \sum_{k \geq 1} 
   \left(\sum_{\sfe\in\sfE^\infty} 
   \varepsilon^{\sfe}\mathbf{c}_k(\sfu_\sfe) \right) h^+_k(x) 
   +  \sum_{k \geq 1}
   \left(\sum_{\sfe\in\sfE^\infty} 
   \varepsilon^{\sfe-\lambda^-_k}\,\mathbf{B}_k(\sfU_\sfe)\right) h^-_k(x)
   \\
   &= \sum_{j\ge 1} \left(\sum_{\sfe\in\sfE^\infty} \varepsilon^{\sfe}
   \mathbf{c}_j(u_{\sfe})\right) h^+_j(x)
   +
   \sum_{k \geq 1} \left(\varepsilon^{-\lambda^-_k}\,\mathbf{B}_k(U_0) 
   - \sum_{j\ge 1} \sum_{\sfe\in\sfE^\infty} \varepsilon^{\sfe+\lambda^+_j-\lambda^-_k}
   \mathbf{c}_j(u_{\sfe})\, 
   \mathbf{B}_k(Y^+_j)\right) \,h^-_k(x)  
   \\
   &= \sum_{k \geq 1} \left(\mathbf{c}_k(u_0) 
   - \sum_{j\ge 1} \sum_{\sfe\in\sfE^\infty} \varepsilon^{\sfe-\lambda^-_j}
   \mathbf{B}_j(U_{\sfe})\,
   \mathbf{c}_k(Y^-_j) \right) h^+_k(x) 
   + \sum_{j\ge 1} \left(\sum_{\sfe\in\sfE^\infty} \varepsilon^{\sfe-\lambda^-_j}
   \mathbf{B}_j(U_{\sfe})\right)h^-_j(x)\, .
\end{aligned}
\]
 Since the above identities are valid for all $x\in \rho\hat\Gamma$ and all $\rho\in(2R_0\varepsilon_1,r_0/2)$ (and thus for all $x\in \sA_{\varepsilon_1}\cap\Gamma$), and since the functions $h^\pm_j$ are independent in $L^2(\sA_{\varepsilon_1}\cap\Gamma)$, we can identify the coefficients of the $h^\pm_j$ and we find the following two independent relations:
\[
   \sum_{\sfe\in\sfE^\infty} 
   \varepsilon^{\sfe}\mathbf{c}_k(\sfu_\sfe) =
   \mathbf{c}_k(u_0) 
   - \sum_{j\ge 1} \sum_{\sfe\in\sfE^\infty} \varepsilon^{\sfe-\lambda^-_j}
   \mathbf{B}_j(U_{\sfe})\,
   \mathbf{c}_k(Y^-_j), \qquad\forall k\ge1
\] 
and
\[
   \sum_{\sfe\in\sfE^\infty} 
   \varepsilon^{\sfe-\lambda^-_k}\,\mathbf{B}_k(\sfU_\sfe) =
   \varepsilon^{-\lambda^-_k}\,\mathbf{B}_k(U_0) 
   - \sum_{j\ge 1} \sum_{\sfe\in\sfE^\infty} \varepsilon^{\sfe+\lambda^+_j-\lambda^-_k}
   \mathbf{c}_j(u_{\sfe})\, 
   \mathbf{B}_k(Y^+_j), \qquad\forall k\ge1\,.
\]
 Multiplying the second equality by $\varepsilon^{\lambda^-_k}$, we obtain
\[
   \sum_{\sfe\in\sfE^\infty} 
   \varepsilon^{\sfe}\,\mathbf{B}_k(\sfU_\sfe) =
   \mathbf{B}_k(U_0) 
   - \sum_{j\ge 1} \sum_{\sfe\in\sfE^\infty} \varepsilon^{\sfe+\lambda^+_j}
   \mathbf{c}_j(u_{\sfe})\, 
   \mathbf{B}_k(Y^+_j), \qquad\forall k\ge1\,.
\]
Then, subtracting $\mathbf{c}_k(u_0)$ from the first equation and $\mathbf{B}_k(U_0)$ from the second,  we arrive at
\[
   \sum_{\sfe\in\sfE^\infty\setminus\{0\}} 
   \varepsilon^{\sfe}\mathbf{c}_k(\sfu_\sfe) =
   - \sum_{j\ge 1} \sum_{\sfe'\in\sfE^\infty} \varepsilon^{\sfe'-\lambda^-_j}
   \mathbf{B}_j(U_{\sfe'})\,
   \mathbf{c}_k(Y^-_j), \qquad\forall k\ge1
\] 
and
\[
   \sum_{\sfe\in\sfE^\infty\setminus\{0\}} 
   \varepsilon^{\sfe}\,\mathbf{B}_k(\sfU_\sfe) =
   - \sum_{j\ge 1} \sum_{\sfe'\in\sfE^\infty} \varepsilon^{\sfe'+\lambda^+_j}
   \mathbf{c}_j(u_{\sfe'})\, 
   \mathbf{B}_k(Y^+_j), \qquad\forall k\ge1\,.
\]
 We now identify the powers of $\varepsilon$ and prove the following:
\begin{lemma}
\label{lem:idcoef}  With the assumptions and notations of Theorem \ref{uepsthmfF}, and with $\mathbf{c}_j$ and $\mathbf{B}_j$ as defined in Theorem \ref{uUseries}, we have
\begin{subequations}
\begin{equation}
\label{eq:ckue}
   \mathbf{c}_k(\sfu_\sfe) =
   - \!\!\!\sum_{j\ge 1\;\;{\rm s.t.}\; \sfe+\lambda^-_j\in\sfE^\infty} 
   \!\!\!\mathbf{B}_j(U_{\sfe+\lambda^-_j})\,
   \mathbf{c}_k(Y^-_j), \qquad\forall\sfe\in\sfE^\infty\setminus\{0\},\quad\forall k\ge1
\end{equation}
and
\begin{equation}
\label{eq:BkUe}
   \mathbf{B}_k(\sfU_\sfe) =
   - \!\!\!\sum_{j\ge 1\;\;{\rm s.t.}\; \sfe-\lambda^+_j\in\sfE^\infty} 
   \!\!\!\mathbf{c}_j(u_{\sfe-\lambda^+_j})\, 
   \mathbf{B}_k(Y^+_j), \qquad\forall\sfe\in\sfE^\infty\setminus\{0\},\quad\forall k\ge1.
\end{equation}
\end{subequations}
\end{lemma}

The terms present in identities \eqref{eq:ckue} and \eqref{eq:BkUe} play distinctive roles: whereas the coefficients $\mathbf{c}_k(Y^-_j)$ and $\mathbf{B}_k(Y^+_j)$ only depend on $\rP$ and $\Omega$, and can, in this sense, be considered characteristic of the problem (see the comment after \eqref{eq:K-jseries}), the other coefficients $\mathbf{c}_k(\sfu_\sfe)$ and $\mathbf{B}_k(\sfU_\sfe)$ can be considered as unknowns in a recursive system. This is clarified by writing identities \eqref{eq:ckue} and \eqref{eq:BkUe} in matrix form for each $\sfe\in\sfE^\infty\setminus\{0\}$. To do so, we need some additional notation.

\begin{notation}[Interaction matrices]
We introduce the semi-infinite square matrices $\cS_\Omega$ and $\cS_\rP$ as follows:
\begin{itemize}
\item $\cS_\Omega$ is the matrix with entries $(\cS_\Omega)_{kj}:=\mathbf{c}_k(Y^-_j)$, for  $k\ge1$ and $j\ge1$;
\item $\cS_\rP$ is the matrix with entries $(\cS_\rP)_{kj}:=\mathbf{B}_k(Y^+_j)$, for  $k\ge1$ and $j\ge1$.
\end{itemize}
\end{notation}

Once more, we recall that the entries of the matrices $\cS_\Omega$ and $\cS_\rP$ can be obtained from solutions of Dirichlet problems in $\Omega$ and $\rP$, respectively, and that they are intrinsic.

The concept of these matrices evokes the idea of scattering, if we read the expansion \eqref{eq:K+jseries} as describing the ``total'' field $K^{+}_{j}$ as the sum of the non-variational ``incoming'' field $h^{+}_{j}$ and a ``scattered'' field that has an expansion at infinity in the basis of the ``outgoing'' fields $h^{-}_{j}$ with expansion coefficients constituting the matrix $\cS_\rP$. 
Similarly, $\cS_\Omega$ is related via \eqref{eq:K-jseries} to a scattering problem at the vertex of the cone, with the roles of $h^{+}_{j}$ and $h^{-}_{j}$ inverted.

\begin{notation} Let $C$ be the vector space of real semi-infinite column vectors: $v:=\big(v_{j}\big)_{j\ge1}$. Let $C[[\Xi^{\sfE^\infty}]]$ be the (vector) space of formal series with exponent set $\sfE^\infty$ and values in $C$, so that any element of  $C[[\Xi^{\sfE^\infty}]]$ is a formal series $\ga := \sum_{\sfe\in\sfE^\infty} \ga_\sfe \,\Xi^\sfe$ with $\ga_\sfe = \big(\ga_{\sfe;j}\big)_{j\ge1}$ (cf.~Appendix \ref{S:gps}). We introduce the operators $\Pi^+$ and $\Pi^-$ on $C[[\Xi^{\sfE^\infty}]]$ that take a formal series $\ga := \sum_{\sfe\in\sfE^\infty} \ga_\sfe \,\Xi^\sfe$ to the series $\ga^+=\Pi^+\ga$ and $\ga^-=\Pi^-\ga$, respectively. These operators are defined as follows:
\begin{itemize}
\item $\ga^+$ is the formal series $\sum_{\sfe\in\sfE^\infty} \ga^+_\sfe\, \Xi^\sfe$ with coefficients $\ga^+_\sfe = \big(\ga^+_{\sfe;j}\big)_{j\ge1}$ given by 
\[
   \ga^+_{\sfe;j} = 
   \begin{cases}
   \ga_{\sfe-\lambda^+_j;\,j} & \mbox{if}\quad \sfe-\lambda^+_j\in\sfE^\infty\,,\\
   0 & \mbox{else.}
   \end{cases}
\]
\item $\ga^-$ is the formal series $\sum_{\sfe\in\sfE^\infty} \ga^-_\sfe\, \Xi^\sfe$ with coefficients $\ga^-_\sfe = \big(\ga^-_{\sfe;j}\big)_{j\ge1}$ given by
\[
   \ga^-_{\sfe;j} = 
   \begin{cases}
   \ga_{\sfe+\lambda^-_j;\,j} & \mbox{if}\quad \sfe+\lambda^-_j\in\sfE^\infty\,,\\
   0 & \mbox{else.}
   \end{cases}
\]
\end{itemize}
\end{notation}

Then, the identities \eqref{eq:ckue} and \eqref{eq:BkUe} can be written in matrix-vector products as
\begin{equation}
\label{eq:matvecprod}
   \gbc_{\sfe} = \cS_\rP\, \gbB^-_{\sfe}
   \quad\mbox{and}\quad
   \gbB_{\sfe} = \cS_\Omega\,  \gbc^+_{\sfe},
   \quad\forall\sfe\in\sfE^\infty\setminus\{0\}\,,
\end{equation} 
where $\gbc$ and $\gbB$ are the formal series with coefficients $\gbc_\sfe = \big(\mathbf{c}_j(\sfu_\sfe)\big)_{j\ge1}$ and $\gbB_\sfe = \big(\mathbf{B}_j(\sfU_\sfe)\big)_{j\ge1}$, respectively. 

The system of equations \eqref{eq:matvecprod} displays the following algorithmic feature:   For any $j\ge1$, we have $\sfe-\lambda^+_j<\sfe$ and $\sfe+\lambda^-_j<\sfe$. Therefore, the values of $\gbB^-_{\sfe}$ and $\gbc^+_{\sfe}$ only depend on coefficient of $\gbc$ and $\gbB$ with exponents $\sfe'$ that are strictly smaller than $\sfe$, making \eqref{eq:matvecprod} a system that can be solved recursively. This leads to the following statement: 

\begin{theorem}
\label{th:algo}
Let $f\in\gF_\Omega$ and $F\in\gF_\rP$. Under the assumptions of Theorem {\em\ref{uepsthmfF}}, the solution $\ueps$ of problem \eqref{PoissonfF} can be obtained for all $\varepsilon\in(0,\varepsilon_\star)$ using the following algorithm:

Define the formal series $\gbc$ with coefficients $\gbc_{\sfe}$ and $\gbB$ with coefficients $\gbB_{\sfe}$ recursively (with increasing $\sfe\in\sfE^\infty$) as follows:
\begin{enumerate}
\item For $\sfe=0$, set $\gbc_{0} = \big(\mathbf{c}_j(u_0)\big)_{j\ge1}$ and $\gbB_0 = \big(\mathbf{B}_j(U_0)\big)_{j\ge1}$, with $u_0$ and $U_0$ solutions of the limit problems \eqref{eq:u0U0}.
\item For $\sfe>0$, determine $\gbc_{\sfe}$ and $\gbB_{\sfe}$ using \eqref{eq:matvecprod}, where the values of $\gbB^-_{\sfe}$ and $\gbc^+_{\sfe}$ only depend on exponents $\sfe'<\sfe$.
\end{enumerate}
Then the inner and outer expansions of $\ueps$ are given by
\[
  U_0\left(\frac{x}{\varepsilon}\right)
   + \sum_{j\ge 1} \sum_{\sfe\in\sfE^\infty} \varepsilon^{\sfe+\lambda^+_j}
   \gbc_{\sfe;j}\, K^+_j\left(\frac{x}{\varepsilon}\right)\,,
\]
and
\[
   u_0(x)
   + \sum_{j\ge 1} \sum_{\sfe\in\sfE^\infty} \varepsilon^{\sfe-\lambda^-_j}
   \gbB_{\sfe;j}\, K^-_j(x)\,,
\]
respectively, leading to the full knowledge of $\ueps$.
\end{theorem}

We also note that, since the sequences $\{-\lambda^+_j\}_j$ and $\{\lambda^-_j\}_j$ tend to $-\infty$, for a fixed $\sfe\in\sfE^\infty$ the vectors $\gbB^-_{\sfe}$ and $\gbc^+_{\sfe}$ have only a finite number of nonzero entries. Consequently, the entries $\gbc_{\sfe;j}$ and $\gbB_{\sfe;j}$ of $\gbc_{\sfe}$ and $\gbB_{\sfe}$ can be determined with a finite number of numerical operations.

\newpage

\appendix

\section{Kelvin transformation}
 \label{S:Kelvin}
 The parallel statements for the bounded and unbounded truncated cones $\Gamma\cap\sB(\rho)$ and $\Gamma\cap\sB^{\complement}(\rho)$ that we have seen in Section~\ref{s:specdev} were proved using similar reasonings and computations for both cases. There exists a clearcut method to obtain the second case directly as a corollary from the first one, by means of the \emph{Kelvin transform}. This transform is defined using the reflection at the unit sphere $\bS^{n-1}$:
$$
 x \mapsto X = |x|^{-2}x.
$$
The Kelvin transformation $\mathcal{K}$  maps a function $u$ to the function $U = \mathcal{K}[u]$, defined by
\begin{equation}
\label{E:Kelvintrf}
 U(X) = \mathcal{K}[u](X) := |X|^{2-n} u(|X|^{-2}X) = |x|^{n-2} u(x).
\end{equation}
For $\rho>0$, we can use the dilation operator $\mathcal{H}_{\rho}$ introduced in Section \ref{ss:prel}:
$$
 \mathcal{H}_{\rho}u(x) = u(\tfrac x\rho)\,,
$$
to define the corresponding transformation $\mathcal{K}_{\rho}$ associated with the reflection at the sphere of radius $\rho$:
$$
  \mathcal{K}_{\rho} := \mathcal{H}_{\rho} \circ\mathcal{K}\circ \mathcal{H}_{\frac1\rho}\,.
$$
This gives the formula
$$
  \mathcal{K}_{\rho}[u](X) = 
  (\tfrac\rho{|X|})^{n-2} u((\tfrac\rho{|X|})^{2}X)\,.
$$
Here is the basic fact:
\begin{lemma}
 \label{L:Kelvin}
 Let $\rho>0$.
The Kelvin transformation $\mathcal{K}_{\rho}$ associated with the reflection at the sphere of radius $\rho$ is an isomorphism between 
$H^{1}_{0}(\Gamma\cap\overline{\sB(\rho)})$  and 
$H^{1}_{\rw,0}(\Gamma\cap\overline{\sB^{\complement}(\rho)})$.
\end{lemma}
\begin{proof}
Because $\mathcal{H}_{\rho}$ provides isomorphisms between 
$H^{1}_{0}(\Gamma\cap\overline{\sB(1)})$ and $H^{1}_{0}(\Gamma\cap\overline{\sB(\rho)})$ 
as well as between
$H^{1}_{\rw,0}(\Gamma\cap\overline{\sB^{\complement}(1)})$ and
$H^{1}_{\rw,0}(\Gamma\cap\overline{\sB^{\complement}(\rho)})$, 
it is sufficient to give the proof for $\rho=1$.

Writing $X=R\vartheta$ and $x=r\vartheta$, with $\vartheta\in\bS^{n-1}$, we easily find that
$$
\begin{aligned}
  |U(X)|^{2}R^{n-3}\rd R &= |u(x)|^{2}r^{n-3}\rd r\,,\\
  \partial_{R}U(X) &= r^{n}\Big(\tfrac{2-n}{r}u(x) - \partial_{r}u(x) \Big)\,,\\
  \nabla_{T}U(X) &= r^{n} \nabla_{T}u(x)\\
  |\nabla U(X)|^{2}R^{n-1}\rd R
  &= |\nabla u(x)|^{2}r^{n-1}\rd r  + (n-2)\partial_{r}\big(r^{n-2}u(x)^{2}\big) \rd r\,.
\end{aligned}
$$
Integrating these identities for $u\in C_{0}^{\infty}(\Gamma)$, we obtain
\begin{align}
\label{E:U/R}
  \DNorm{\frac UR}{L^{2}(\Gamma\cap\sB^{\complement}(1))} 
  &= \DNorm{\frac ur}{L^{2}(\Gamma\cap\sB(1))}\,,\\
\label{E:H1GammaB1}
  \DNormc{\nabla U}{L^{2}(\Gamma\cap\sB^{\complement}(1))}
  &= \DNormc{\nabla u}{L^{2}(\Gamma\cap\sB(1))}
   + (n-2) \DNormc{\tr_{1}u}{L^{2}(\hat\Gamma)}\,,\\
\label{E:H1Gamma}
  \DNormc{\nabla U}{L^{2}(\Gamma)}
  &= \DNormc{\nabla u}{L^{2}(\Gamma)}\,.
\end{align}
Then, combining \eqref{E:H1GammaB1} with the estimate \eqref{E:L2trace} from Lemma~\ref{L:L2trace}, we obtain that 
$$
 \DNormc{\nabla u}{L^{2}(\Gamma\cap\sB(1))} \le
 \DNormc{\nabla U}{L^{2}(\Gamma\cap\sB^{\complement}(1))} \le
 (1+(n-2)C_{\hat \Gamma}^{2}) \DNormc{\nabla u}{L^{2}(\Gamma\cap\sB(1))}\,,
$$
which implies the equivalence of the $H^{1}$ seminorms of  $u\in C_{0}^{\infty}(\Gamma)$ and its transform $U$. Observing that $\mathcal{K}$ maps $C_{0}^{\infty}(\Gamma)$ to itself, and using a standard density argument,  we obtain the equivalence of the norms in $H^{1}_{0}(\Gamma\cap\overline{\sB(1)})$ of $u$ and in $H^{1}_{\rw,0}(\Gamma\cap\overline{\sB^{\complement}(1)})$ of $\mathcal{K}[u]$, thus proving that 
$\mathcal{K}$ is an isomorphism between 
$H^{1}_{0}(\Gamma\cap\overline{\sB(1)})$  and $H^{1}_{\rw,0}(\Gamma\cap\overline{\sB^{\complement}(1)})$, with inverse given by the same formula.
\end{proof}

For example, Lemma \ref{L:Kelvin} can be used to deduce the trace estimate \eqref{E:L2trcompl} directly from \eqref{E:L2trace}. It suffices to notice that the traces on $\rho\hat\Gamma$ of $u$ and its transform $\mathcal{K}_{\rho}[u]$ are the same.

Another straightforward consequence of equation \eqref{E:H1Gamma} is the well-known classical relation:
$$
 \Delta \mathcal{K}[u] = \mathcal{K}[|x|^{4}\Delta u]\,,
$$
which shows that $\mathcal{K}$ preserves harmonic functions.

For the special harmonic functions $h^{\pm}_{j}$,  which form the bases for the expansion of  harmonic functions near the origin or near infinity, we have the relation
\begin{equation}
\label{E:h-=Kh+}
  h^{-}_{j} = \mathcal{K}[h^{+}_{j}]\,.
\end{equation}


\section{Generalized power series}\label{S:gps}

\subsection{Formal generalized power series}\label{SS:formal}
In this paper, we come across series of the form
\[
\sum_{j=0}^\infty c_j \xi^{\sfe_j}
\]
where $\{\sfe_j\}_{j=0}^\infty$ is  a sequence of nonnegative real exponents and $\{c_j\}_{j=0}^\infty$ is a sequence of coefficients in a suitable space of functions or operators.

The study of such series has a rich history in algebra, with early contributions by Hahn \cite{Ha07}, MacLane \cite{Ma39}, Mal'tsev \cite{Ma48}, and Neumann \cite{Ne49}. Over time, the theory surrounding these series has evolved and branched out into various directions, connecting with concepts such as formal Laurent series, surreal numbers, and the Levi-Civita field, among other algebraic structures. In particular, the generalization has concerned the set of exponents, that can be an unspecified ordered monoid (i.e.~ordered additive semigroup with a zero element). 

For this paper, we do not require the full generality provided in the algebraic theory.  Only a few elementary facts will be sufficient.  In particular, we will confine ourselves to generalized power series with exponents in $[0,+\infty)$, which for this purpose we will denote by $\sfEp$. 
Although the results in this section are well known, for the sake of convenience of the reader, we give some relevant proofs. 

A {\em generalized power series} 
\[
   \gotf:=\sum_{\sfe\in \sfE} c_\sfe \Xi^{\sfe}
\]
with exponents in a set $\sfE$, coefficients in $C$ and indeterminate $\Xi$ is simply a mapping $\sfe\mapsto c_{\sfe}$ from $\sfE$ to $C$. When the set $C$ has additional structure, such as that of a group or ring, the set of such series will acquire an analogous structure, as suggested by the writing as sums. 
In general, we assume that $C$ is an additive group with neutral element $\boldzero$. 
The subset
\[
\sfE_\gotf:=\{\sfe\in \sfE,\quad c_\sfe\neq \boldzero\}
\] 
is called the support of $\gotf$. Let us first recall:
\begin{definition}
\label{def:wellord}
A totally ordered set $\sfE$ is said {\em well-ordered} if every nonempty subset of $\sfE$ has a smallest element.
\end{definition}

\begin{remark}
\label{rem:wo}\
The proofs of the following statements are easy exercises.\\
(i) Any well-ordered subset of $\sfEp$ is countable. \\[0.5ex] 
(ii) Any discrete subset of $\sfEp$ (i.e.~a set without finite accumulation points) is well-ordered.\\[0.5ex]
(iii) There exist well-ordered sets that are not discrete (for example containing an increasing converging sequence).\\[0.5ex]
(iv) A non-empty totally ordered set $\sfE$ is well-ordered if and only if it does not contain an infinite strictly decreasing sequence or equivalently, every sequence in $\sfE$ contains an increasing subsequence.
\end{remark}

\begin{definition}
 \label{D:C[[]]}
 Let $C$ be an additive group. The set of generalized power series $\gotf$ with exponents in $\sfE\subset\sfEp$ and coefficients in $C$ such that 
 \begin{equation}
 \label{E:suppwo}
   \mbox{ the support }\sfE_{\gotf} \mbox{ is well-ordered}
 \end{equation}
 is denoted by $C[[\Xi^{\sfE}]]$. 
\end{definition}

If $\sfE\subset\sfE'$, then $C[[\Xi^{\sfE}]]$ is in a natural way contained in $C[[\Xi^{\sfE'}]]$ (extension by zero), so that all the generalized power series we will consider here belong to $C[[\Xi^{\sfEp}]]$. 

The group structure of $C$ gives $C[[\Xi^{\sfEp}]]$ a group structure:
If $\gotf:=\sum_{\sfe\in \sfE} c_\sfe \Xi^{\sfe}$ and 
$\gotf':=\sum_{\sfe\in \sfE'} c_\sfe' \Xi^{\sfe}$, then
$$
  \gotf+\gotf' := \sum_{\sfe\in \sfE\cup\sfE'} (c_\sfe+c_{\sfe}') \Xi^{\sfe}\,.
$$
The important observation is that $\sfE_{\gotf+\gotf'}\subset\sfE_{\gotf}\cup\sfE_{\gotf'}$, and this is well-ordered if $\sfE_{\gotf}$ and $\sfE_{\gotf'}$ are well-ordered. The neutral element of $C[[\Xi^{\sfEp}]]$ for the addition is the null series with empty support.

Similarly, the structure of a vector space over a field $K$ carries over from $C$ to $C[[\Xi^{\sfEp}]]$.

If $C$ is a ring, then $C[[\Xi^{\sfEp}]]$ has a ring structure, too: The product is defined by
$$
  \gotf\,\gotf' = \sum_{\sfe\in \sfE+\sfE'}\Big(
  \sum_{\substack{\sfe_{1}+\sfe_{2}=\sfe\\ \sfe_{1}\in\sfE, \sfe_{2}\in \sfE'}}
    c_{\sfe_{1}}c_{\sfe_{2}}'\Big)\Xi^\sfe\,.
$$
For the supports one has
$$
  \sfE_{\gotf\,\gotf'} \subset \sfE_{\gotf}+\sfE_{\gotf'}
  = \{\sfe_{1}+\sfe_{2} ,\quad \sfe_{1}\in\sfE_{\gotf},\; \sfe_{2}\in \sfE_{\gotf'} \}\,.
$$
At this point, the condition of the well-ordering of the supports becomes essential. 
\begin{lemma}
 \label{L:suppprod}
 Let $\sfE_{1},\sfE_{2}$ be well-ordered subsets of $\sfEp=[0,\infty)$. Then
 $\sfE_{1}+\sfE_{2}$ is well-ordered, and for any $\sfe\in\sfE_{1}+\sfE_{2}$ the set
 $$
   \{ (\sfe_{1},\sfe_{2})\in \sfE_{1}\times\sfE_{2} ,\quad
    \sfe_{1}\in\sfE_{1},\; \sfe_{2}\in \sfE_{2},\; \sfe_{1}+\sfe_{2}=\sfe \} 
 $$
 is finite.
\end{lemma}

\begin{proof} 
We use Remark \ref{rem:wo}, (iv).
Suppose that $(\sfe_{j})_{j\in\N}$ is a sequence in $\sfE_{1}+\sfE_{2}$ with 
$\sfe_{j}=\sfe_{j1}+\sfe_{j2}$, $\sfe_{jk}\in\sfE_{k}$. The sequence $(\sfe_{j1})_{j}$ contains an increasing subsequence $(\sfe_{j_{\ell}1})_{\ell}$. If $(\sfe_{j})_{j\in\N}$ were strictly decreasing, then $(\sfe_{j_{\ell}2})_{\ell}$ would also be strictly decreasing, in contradiction with the well-orderedness of $\sfE_{2}$. Hence $\sfE_{1}+\sfE_{2}$ is well-ordered. 
Suppose now there exists $\sfe\in\sfE_{1}+\sfE_{2}$ with an infinite number of different decompositions
 $\sfe=\sfe_{j1}+\sfe_{j2}$, $j\in\N$, $\sfe_{jk}\in\sfE_{k}$. We may assume that $(\sfe_{j1})_{j}$ is strictly increasing, implying the contradiction that $(\sfe_{j2})_{j}$ is strictly decreasing.
\end{proof}

The ring $C$ is naturally embedded in $C[[\Xi^{\sfEp}]]$ by identifying $c_{0}\in C$ with the series $c_{0}\Xi^{0}$, written simply as $c_{0}$. If $C$ has a left and right multiplicative identity $\boldone$, then this yields a multiplicative identity for $C[[\Xi^{\sfEp}]]$ as well, which we denote with the same symbol $\boldone$.
One can then talk about (left and right) inverses of formal generalized power series.

Before stating the main result of this section, we need a lemma that makes use of the archimedian property of $\sfEp=[0,\infty)$.
\begin{lemma}
 \label{L:Einf}
 Let $\sfE$ be a well-ordered subset of $\sfEp$ and let 
 $$
    \sfE^\infty = \{\sfe=\sfe_1+\ldots+\sfe_k ,\quad 
   \sfe_1,\ldots,\sfe_k\in \sfE\cup\{0\},\quad k\in\N^*\}.
 $$
 be the monoid generated by $\sfE$.
 Then $\sfE^\infty$ is well-ordered.
\end{lemma}
\begin{proof}
Let $\sfe_*$ be the smallest nonzero element of $\sfE$. If no such element exists, then $\sfE^{\infty}=\{0\}$, which is well-ordered.  
Let $S\subset\sfE^\infty$ and $\sfe_0\in S$. To prove that $S$ has a smallest element, it suffices to prove the same for $S':=\{\sfe\in S,\quad \sfe\le\sfe_0\}$.
Since $\sfEp$ is archimedian, there exists $k\in\N$ such that $k\sfe_*>\sfe_0$. 
Let
$$
 \sfE^{k}=\{\sfe_{1}+\ldots+\sfe_k ,\quad \sfe_1,\ldots,\sfe_k\in \sfE\cup\{0\}\,\}\,.
$$
Then $S'\subset\sfE^{k}$, and since $\sfE^{k}$ is well-ordered according to Lemma~\ref{L:suppprod}, $S'$ has a smallest element.
\end{proof} 

\begin{theorem} 
 \label{P:Neumann}
 Let $C$ be a ring with identity $\boldone$. Let $\gotf\in C[[\Xi^{\sfEp}]]$ with
 $0\not\in\sfE_{\gotf}$. Then $\boldone+\gotf$ is invertible in $C[[\Xi^{\sfEp}]]$ with inverse given by the formal Neumann series 
\begin{equation}
\label{E:Neumann}
(\boldone+\gotf)^{-1}:=\boldone+\sum_{k=1}^\infty(-\gotf)^k
 = \boldone + \sum_{\sfe\in\sfE_{\gotf}^{\infty}}
   \Big( \sum_{k=1}^{\infty}(-1)^{k}
      \sum_{\substack{\sfe_1+\dots+\sfe_k=\sfe\\ \sfe_1,\dots,\sfe_k\in \sfE_{\gotf}}}
         c_{\sfe_1}\dots c_{\sfe_k}\Big)\Xi^\sfe
\end{equation}
with support contained in the monoid $\sfE_{\gotf}^{\infty}$.
\end{theorem}
\begin{proof}
This is a well-defined formal generalized power series, because for any $\sfe\in\sfE_{\gotf}^{\infty}$, the sum defining the coefficient of $\Xi^{\sfe}$ is finite, see Lemma~\ref{L:suppprod}. The proof is then the usual identity: 
$$
  \gotf\big(\boldone+\sum_{k=1}^\infty(-\gotf)^k\big)
  = \big(\boldone+\sum_{k=1}^\infty(-\gotf)^k\big)\gotf
  = \boldone -\big(\boldone+\sum_{k=1}^\infty(-\gotf)^k\big)\,.
$$
\end{proof}
\begin{corollary}
 \label{C:inv}
 Let $\gotf=\sum_{\sfe}c_{\sfe}\Xi^{\sfe}\in C[[\Xi^{\sfEp}]]$. Then $\gotf$ is invertible in $C[[\Xi^{\sfEp}]]$ if and only if $c_{0}$ is invertible in $C$.
\end{corollary}
\begin{proof}
If $\gotf$ is invertible, then the constant term in $\gotf^{-1}$ is the inverse of $c_{0}$.
If $c_{0}$ is invertible, write 
$c_{0}^{-1}\gotf=\boldone+\sum_{\sfe\ne0}c_{0}^{-1}c_{\sfe}\Xi^{\sfe}$
and apply the theorem.
\end{proof}
\begin{remark}
 \label{R:valu}
The infinite series $\sum_{k=1}^{\infty}(-\gotf)^{k}$ converges formally in the sense that for exponents below a given finite threshold, only a finite number of terms in the series are contributing. This type of formal convergence can be made into a convergence with respect to a metric if one introduces the \emph{valuation} $v(\gotf)$ of a generalized power series 
$\gotf\in C[[\Xi^{\sfEp}]]$ and then the \emph{modulus} $|\gotf|$ (not a norm!) by looking at the smallest exponent with non-zero coefficient.
\begin{equation}
\label{E:valu}
 v(\gotf):=\min \sfE_{\gotf}\,;\quad
 |\gotf|:=e^{-v(\gotf)}\mbox{ if }\gotf\ne\boldzero\,,\; |\boldzero|:=0\,.
\end{equation}
One has $|\gotf+\gotg|\le\max\{|\gotf|,|\gotg|\}$, and with the metric $d(\gotf,\gotg)=|\gotf-\gotg|$, $C[[\Xi^{\sfEp}]]$ becomes a \emph{complete ultra-metric space}. In particular, if $0\not\in\sfE_{\gotf}$ as in Theorem~\ref{P:Neumann}, then $|\gotf|<1$, and because of $|\gotf\,\gotg|\le|\gotf||\gotg|$ and hence $|(-\gotf)^{k}|\le|\gotf|^{k}$, the convergence of the series $\sum_{k=1}^{\infty}(-\gotf)^{k}$ follows.

This type of convergence is, however, \emph{not} the convergence that we will study in the next section. 
\end{remark}

\subsection{Convergent generalized power series}\label{SS:conv}
If the coefficients $C$ are nonnegative reals, $C=\sfEp$, a formal power series 
$\gotf=\sum_{\sfe}c_{\sfe}\Xi^{\sfe}\in C[[\Xi^{\sfEp}]]$ can be evaluated at $\xi>0$:
$$
 \gotf(\xi) =\sum_{\sfe\in\sfE_{\gotf}}c_{\sfe}\xi^{\sfe} \in [0,\infty].
$$
The value of the sum does not depend on the order of summation, and the function
$\xi\mapsto\gotf(\xi)$ is monotone. In particular, if $\gotf(\xi)<\infty$ for some $\xi>0$, then $\gotf(\eta)<\infty$ for $0<\eta\le\xi$. Additionally, in this case, $\eta\mapsto\gotf(\eta)$  is  continuous on $(0,\xi]$ and has a continuous extension to $[0,\xi]$, because
$$
  \lim_{\eta\to0}\gotf(\eta)=c_{0}\,.
$$
 Note that $\sfE_{\gotf}$ is a well-ordered subset of $[0,\infty)$ and therefore countable, so that the series defining $\gotf(\xi)$ can be understood in the usual way as the limit of finite partial sums. It is also normally convergent on $[0,\xi]$ in the sense that 
$$
  \sum_{\sfe\in\sfE_{\gotf}}\sup_{\eta\in[0,\xi]}c_{\sfe}\eta^{\sfe} < \infty\,,
$$
because $\sup_{\eta\in[0,\xi]}c_{\sfe}\eta^{\sfe}=c_{\sfe}\xi^{\sfe}$.

These properties can be used to bound series  with coefficients in a normed vector space.
Let $C$ be a normed vector space with norm $\DNorm{\cdot}{}$. The mapping $N$ defined by
$$
 N\Big(\sum_{\sfe}c_{\sfe}\Xi^{\sfe}\Big) = \sum_{\sfe} \|c_{\sfe}\|\Xi^{\sfe}
$$
associates a generalized power series $\gotf=\sum_{\sfe}c_{\sfe}\Xi^{\sfe}$ in $C[[\Xi^{\sfEp}]]$ with a generalized power series $N(\gotf)\in\sfEp[[\Xi^{\sfEp}]]$. The support of $N(\gotf)$ is the same as the support of $\gotf$.

\begin{definition}
\label{D:cgps}
We call $\gotf\in C[[\Xi^{\sfEp}]]$ a \emph{convergent generalized power series} if there exists $\xi>0$ such that
$N(\gotf)(\xi)<\infty$. 
The set of convergent generalized power series with exponents in $\sfE$ is denoted by $C[[\Xi^{\sfE}]]_{*}$, and for $\xi>0$, we write $C[[\Xi^{\sfE}]]_{\xi}$ for the subset of those $\gotf$ for which $N(\gotf)(\xi)<\infty$.
\end{definition}

An example of a convergent generalized power series is a series that satisfies a polynomial growth estimate for the exponents and an exponential growth estimate for the coefficients. In what follows, we write  
$\#\{\sfe<b,\quad \sfe\in\sfE\}$ for the number of exponents in $\sfE\subset\sfEp$ that are smaller than $b$. This number is finite for all $b>0$ if and only if $\sfE$ is discrete.

\begin{lemma}
 \label{L:polgro}
 Let $\gotf=\sum_{\sfe\in\sfE}c_{\sfe}\Xi^{\sfe}\in C[[\Xi^{\sfEp}]]$ and assume that there are positive constants $a$, $d$, $\rho$, $m$ such that
\begin{equation}
\label{E:polgro}
 \#\{\sfe<b,\quad \sfe\in\sfE\} \le a(b+1)^{m}
 \:\mbox{ for all } b>0 \quad\mbox{ and }\quad
 \DNorm{c_{\sfe}}{}\le d\, \rho^{\sfe}
 \:\mbox{ for all } \sfe\in\sfE\,.
\end{equation}
Then $\gotf\in C[[\Xi^{\sfE}]]_{\xi}\,$ for all $\xi\in(0,\frac1\rho)$.
\end{lemma}
\begin{proof}
 For $k\in\N$, the number of exponents in the interval $[k,k+1)$ is bounded by 
 $\#\{\sfe<k+1,\quad \sfe\in\sfE\}\le a(k+2)^{m}$ and the norm $\DNorm{c_{e}}{}$ of the corresponding coefficients by $d\,\max\{\rho^k,\rho^{k+1}\}$. Therefore, for any $\xi>0$ we have
$$
  N(\gotf)(\xi) = \sum_{k=0}^{\infty}\sum_{k\le \sfe< k+1}\DNorm{c_{e}}{}\xi^{e}
  \le \sum_{k=0}^{\infty} a\,d\,
  \max\{1,\tfrac1\rho\}(k+2)^{m}(\rho\xi)^{k+1}\,.
$$
This is convergent as soon as $\rho\xi<1$.
 \end{proof}

If $C$ is a Banach space and $\gotf\in C[[\Xi^{\sfE}]]_{\xi}$, then
the series $\gotf(\xi) =\sum_{\sfe\in\sfE_{\gotf}}c_{\sfe}\xi^{\sfe}$  converges in $C$, and 
one has $\DNorm{\gotf(\xi)}{}\le N(\gotf)(\xi)$.
Then, the series
$$
  \gotf(\eta)= \sum_{\sfe}c_{\sfe}\eta^{\sfe}
$$  
converges normally in $C$ for all $\eta\in[0,\xi]$. It follows that the sum does not depend on the order of summation,  that $\DNorm{\gotf(\eta)}{}\le N(\gotf)(\xi)$, and that $\eta\mapsto\gotf(\eta)$ is a continuous function on $[0,\xi]$ with values in $C$. In particular, $\gotf(\eta)\to c_{0}$ in $C$ as $\eta\to0$.

It is not hard to see that  
$$
  \DNorm\gotf\xi := N(\gotf)(\xi)
$$
is a norm that makes $C[[\Xi^{\sfE}]]_{\xi}$ a normed vector space. If $\sfE$ is well-ordered, then one can readily show that $C[[\Xi^{\sfE}]]_{\xi}$ is a Banach space, and if $C$ is a Banach algebra, then $\DNorm\cdot\xi$ makes $C[[\Xi^{\sfE}]]_{\xi}$ a Banach algebra.
In particular, we have
\begin{equation}
\label{E:normprod}
  \DNorm{\gotf\,\gotg}{\xi} \le \DNorm{\gotf}{\xi}\DNorm{\gotg}{\xi}.
\end{equation}

We observe that to prove the completeness of $C[[\Xi^{\sfE}]]_{\xi}$, we cannot relax the assumption that $\sfE$ is well-ordered. As an example, consider  $\sfE=\sfEp$ and let $\gotf_N:=\sum_{j=1}^N 2^{-j}\Xi^{1/j}$. We can verify that $\{\gotf_N\}_{N=1}^\infty$ is a Cauchy sequence in $\R[[\Xi^{\sfEp}]]_{\xi}$ for any given $\xi>0$, but the support of the limiting series $\sum_{j=1}^\infty2^{-j}\Xi^{1/j}$ is not well-ordered.

For our applications, we need two further results on convergent generalized power series, for which we will provide proofs.

The first is the \emph{identity principle}, that says that the coefficients of a convergent generalized power series $\gotf$ can be recovered from the function $\xi\mapsto\gotf(\xi)$.
\begin{proposition}
\label{P:identityprinciple}
Let $C$ be a Banach space. Let $\gotf,\gotg \in C[[\Xi^\sfE]]_*$. If there exists a sequence  $\{\xi_i\}_i$ of positive real numbers such that $\xi_i\to 0$ and $\gotf(\xi_i)=\gotg(\xi_i)$ for all $i\ge 1$, then $\gotf=\gotg$.
\end{proposition}
\begin{proof}
By considering the difference $\gotf-\gotg$, we see that it suffices to give the proof for the case $\gotg=\boldzero$. So let $\gotf(\xi_{i})=\boldzero$ for all $i$. If $\gotf\ne0$, its support $\sfE_{\gotf}$ is not empty and has a smallest element $\sfe_{1}$. Let $c_{1}$ be the corresponding coefficient. Set $\gotf_{1}:= \sum_{\sfe}c_{\sfe}\Xi^{\sfe-\sfe_{1}}$. Then $\gotf_{1}\in C[[\Xi^{\sfEp}]]_{*}$ with $f_{1}(\xi_{i})=\xi_{i}^{-\sfe_{1}}\gotf(\xi_{i})=0$ for all $i$, and we arrive at the contradiction
$$
  c_{1}=\lim_{\xi\to0}\gotf_{1}(\xi)=\boldzero\,.
$$
\end{proof}

The second and most important result concerns the inverse of a convergent generalized power series.
\begin{theorem}
 \label{T:invconv}
 Let $C$ be a Banach algebra and let $\gotf=\sum_{\sfe\in\sfE}c_{\sfe}\Xi^{\sfe}\in C[[\Xi^{\sfEp}]]_{*}$ be
 a convergent generalized power series with $\sfE\subset\sfEp$. Assume that $\min\sfE_{\gotf}>0$ and 
 $$
   \DNorm{\gotf}{\xi}<1 \quad\mbox{ for some }\xi>0\,.
 $$
 Then the formal power series $(\boldone+\gotf)^{-1}$, which exists according to Theorem {\em \ref{P:Neumann}}, is convergent and thus an element of $C[[\Xi^{\sfE_{\gotf}^{\infty}}]]_{\xi}$.
\end{theorem}
\begin{proof}
 The Neumann series 
 $$
  (\boldone+\gotf(\xi))^{-1} = \sum_{k=0}^{\infty} \big(-\gotf(\xi)\big)^{k}
$$
converges in $C$, because we have $\DNorm{\gotf(\xi)}{}\le \DNorm{\gotf}{\xi} <1$ and
$\DNorm{(-\gotf(\xi))^{k}}{}\le \|\gotf\|_{\xi}^{k}$ (see \eqref{E:normprod}). According to the expression for the powers $\gotf^{k}$ in \eqref{E:Neumann}, we have 
$$
  \big(\gotf(\xi)\big)^{k} = (\gotf^{k})(\xi)\,.
$$
It follows that
$$
 (1+\gotf(\xi))^{-1} = \big((\boldone+\gotf)^{-1}\big)(\xi)
$$
and
$$
  \DNorm{\big((\boldone+\gotf)^{-1}\big)}{\xi} \le \sum_{k=0}^{\infty} \|\gotf\|_{\xi}^{k}
  = \big(1-\DNorm{\gotf}{\xi}\big)^{-1}\,.
$$
Hence, $(\boldone+\gotf)^{-1}\in C[[\Xi^{\sfE_{\gotf}^{\infty}}]]_{\xi}$\,.
\end{proof}

\begin{corollary}
 \label{C:invconv}
Let $C$ be a Banach algebra and let $\gotf=\sum_{\sfe}c_{\sfe}\Xi^{\sfe}\in C[[\Xi^{\sfEp}]]_{*}$ 
be a convergent generalized power series such that $c_{0}$ is invertible. Then the generalized power series $\gotf^{-1}$,  which exists according to Corollary {\em \ref{C:inv}}, is convergent.
\end{corollary}
\begin{proof}
 Let $\gotg=c_{0}^{-1}\gotf-\boldone$. Then 
$$
  \gotf = c_{0}(\boldone+\gotg) \quad\mbox{ and } 0\not\in \sfE_{\gotg}\,.
$$
We know that $\gotg(\xi)$ exists for sufficiently small $\xi$, and that $\gotg(\xi)\to\boldzero$ in $C$ as $\xi\to0$. Hence, there exists $\xi>0$ such that $\DNorm\gotg\xi<1$. 
We can then apply Theorem~\ref{T:invconv} to conclude that $\boldone+\gotg$ is invertible in
$C[[\Xi^{\sfEp}]]_{\xi}$. The inverse of $\gotf$ is given by
$$
  \gotf^{-1} = (\boldone+\gotg)^{-1}c_{0}^{-1}
    = \sum_{k=0}^{\infty} \big(c_{0}^{-1}(c_{0}-\gotf)\big)^{k}\,c_{0}^{-1}\,.
$$
\end{proof}

\bibliographystyle{siam}
\bibliography{faaextbiblio}

\begin{thebibliography}{10}

\bibitem{AkLa22}
{\sc T.~Akyel and M.~Lanza~de Cristoforis}, {\em Asymptotic behavior of the
  solutions of a transmission problem for the {H}elmholtz equation: a
  functional analytic approach}, Math. Methods Appl. Sci., 45 (2022),
  pp.~5360--5387.

\bibitem{AmKa07}
{\sc H.~Ammari and H.~Kang}, {\em Polarization and moment tensors}, vol.~162 of
  Applied Mathematical Sciences, Springer, New York, 2007.
\newblock With applications to inverse problems and effective medium theory.

\bibitem{BoDaDaMu18}
{\sc V.~Bonnaillie-No\"{e}l, M.~Dalla~Riva, M.~Dambrine, and P.~Musolino}, {\em
  A {D}irichlet problem for the {L}aplace operator in a domain with a small
  hole close to the boundary}, J. Math. Pures Appl. (9), 116 (2018),
  pp.~211--267.

\bibitem{BoDaDaMu21}
\leavevmode\vrule height 2pt depth -1.6pt width 23pt, {\em Global
  representation and multiscale expansion for the {D}irichlet problem in a
  domain with a small hole close to the boundary}, Comm. Partial Differential
  Equations, 46 (2021), pp.~282--309.

\bibitem{CaCoDaVi06}
{\sc G.~Caloz, M.~Costabel, M.~Dauge, and G.~Vial}, {\em Asymptotic expansion
  of the solution of an interface problem in a polygonal domain with thin
  layer}, Asymptot. Anal., 50 (2006), pp.~121--173.

\bibitem{CoDaDaMu17}
{\sc M.~Costabel, M.~Dalla~Riva, M.~Dauge, and P.~Musolino}, {\em Converging
  expansions for {L}ipschitz self-similar perforations of a plane sector},
  Integral Equations Operator Theory, 88 (2017), pp.~401--449.

\bibitem{CostabelDaugeYosibash04}
{\sc M.~Costabel, M.~Dauge, and Z.~Yosibash}, {\em A quasi-dual function method
  for extracting edge stress intensity functions}, SIAM J. Math. Anal., 35
  (2004), pp.~1177--1202 (electronic).

\bibitem{Co95}
{\sc G.~Courtois}, {\em Spectrum of manifolds with holes}, J. Funct. Anal., 134
  (1995), pp.~194--221.

\bibitem{Da13}
{\sc M.~Dalla~Riva}, {\em Stokes flow in a singularly perturbed exterior
  domain}, Complex Var. Elliptic Equ., 58 (2013), pp.~231--257.

\bibitem{DaLa10}
{\sc M.~Dalla~Riva and M.~Lanza~de Cristoforis}, {\em A singularly perturbed
  nonlinear traction boundary value problem for linearized elastostatics. {A}
  functional analytic approach}, Analysis (Munich), 30 (2010), pp.~67--92.

\bibitem{DaLaMu21}
{\sc M.~Dalla~Riva, M.~Lanza~de Cristoforis, and P.~Musolino}, {\em Singularly
  perturbed boundary value problems---a functional analytic approach},
  Springer, Cham, [2021] \copyright 2021.

\bibitem{DaMu16}
{\sc M.~Dalla~Riva and P.~Musolino}, {\em A mixed problem for the {L}aplace
  operator in a domain with moderately close holes}, Comm. Partial Differential
  Equations, 41 (2016), pp.~812--837.

\bibitem{DaMu17}
\leavevmode\vrule height 2pt depth -1.6pt width 23pt, {\em The {D}irichlet
  problem in a planar domain with two moderately close holes}, J. Differential
  Equations, 263 (2017), pp.~2567--2605.

\bibitem{DBook88}
{\sc M.~Dauge}, {\em Elliptic boundary value problems on corner domains},
  vol.~1341 of Lecture Notes in Mathematics, Springer-Verlag, Berlin, 1988.
\newblock Smoothness and asymptotics of solutions.

\bibitem{DaugeNicaise:1990I}
{\sc M.~Dauge, S.~Nicaise, M.~Bourlard, and J.~M.-S. Lubuma}, {\em Coefficients
  des singularit\'{e}s pour des probl\`emes aux limites elliptiques sur un
  domaine \`a points coniques. {I}. {R}\'{e}sultats g\'{e}n\'{e}raux pour le
  probl\`eme de {D}irichlet}, RAIRO Mod\'{e}l. Math. Anal. Num\'{e}r., 24
  (1990), pp.~27--52.

\bibitem{DaToVi10}
{\sc M.~Dauge, S.~Tordeux, and G.~Vial}, {\em Selfsimilar perturbation near a
  corner: matching versus multiscale expansions for a model problem}, in Around
  the research of {V}ladimir {M}az'ya. {II}, vol.~12 of Int. Math. Ser. (N.
  Y.), Springer, New York, 2010, pp.~95--134.

\bibitem{FeAm22}
{\sc F.~Feppon and H.~Ammari}, {\em High order topological asymptotics:
  reconciling layer potentials and compound asymptotic expansions}, Multiscale
  Model. Simul., 20 (2022), pp.~957--1001.

\bibitem{Ha07}
{\sc H.~Hahn}, {\em {\"U}ber die nichtarchimedischen
  {Gr{\"o}{{\ss}}ensysteme}.}, Wien. Ber., 116 (1907), pp.~601--655.

\bibitem{HeSc21}
{\sc F.~Henr\'{\i}quez and C.~Schwab}, {\em Shape holomorphy of the
  {C}alder\'{o}n projector for the {L}aplacian in {${\Bbb R}^2$}}, Integral
  Equations Operator Theory, 93 (2021), pp.~Paper No. 43, 40.

\bibitem{Il78}
{\sc A.~M. Il'in}, {\em {A} boundary value problem for a second-order elliptic
  equation in a domain with a narrow slit. {I}. {T}he two-dimensional case},
  Math. USSR Sb., 28 (1978), pp.~459--480.

\bibitem{Il92}
\leavevmode\vrule height 2pt depth -1.6pt width 23pt, {\em Matching of
  asymptotic expansions of solutions of boundary value problems}, vol.~102 of
  Translations of Mathematical Monographs, American Mathematical Society,
  Providence, RI, 1992.
\newblock Translated from the Russian by V. Minachin [V. V. Minakhin].

\bibitem{Josien:2024}
{\sc M.~Josien, C.~Raithel, and M.~Sch\"{a}ffner}, {\em Stochastic
  homogenization and geometric singularities: A study on corners}, SIAM Journal
  on Mathematical Analysis, 56 (2024), pp.~2395--2455.

\bibitem{Kondrat67}
{\sc V.~A. Kondrat'ev}, {\em Boundary value problems for elliptic equations in
  domains with conical or angular points}, Trudy Moskov. Mat. Ob\v{s}\v{c}., 16
  (1967), pp.~209--292.

\bibitem{KoMaMo99}
{\sc V.~Kozlov, V.~Maz'ya, and A.~Movchan}, {\em Asymptotic analysis of fields
  in multi-structures}, Oxford Mathematical Monographs, The Clarendon Press,
  Oxford University Press, New York, 1999.
\newblock Oxford Science Publications.

\bibitem{KozlovMazyaRossmann97b}
{\sc V.~A. Kozlov, V.~G. Maz'ya, and J.~Rossmann}, {\em Elliptic boundary value
  problems in domains with point singularities}, Mathematical Surveys and
  Monographs, 52, American Mathematical Society, Providence, RI, 1997.

\bibitem{KozlovMazyaRossmann01}
\leavevmode\vrule height 2pt depth -1.6pt width 23pt, {\em Spectral Problems
  Associated with Corner Singularities of Solutions to Elliptic Equations},
  Mathematical Surveys and Monographs, 85, American Mathematical Society,
  Providence, RI, 2001.

\bibitem{La02}
{\sc M.~Lanza~de Cristoforis}, {\em Asymptotic behaviour of the conformal
  representation of a {J}ordan domain with a small hole in {S}chauder spaces},
  Comput. Methods Funct. Theory, 2 (2002), pp.~1--27.

\bibitem{La08}
\leavevmode\vrule height 2pt depth -1.6pt width 23pt, {\em Asymptotic behavior
  of the solutions of the {D}irichlet problem for the {L}aplace operator in a
  domain with a small hole. {A} functional analytic approach}, Analysis
  (Munich), 28 (2008), pp.~63--93.

\bibitem{LaRo04}
{\sc M.~Lanza~de Cristoforis and L.~Rossi}, {\em Real analytic dependence of
  simple and double layer potentials upon perturbation of the support and of
  the density}, J. Integral Equations Appl., 16 (2004), pp.~137--174.

\bibitem{Ma39}
{\sc S.~MacLane}, {\em The universality of formal power series fields}, Bull.
  Amer. Math. Soc., 45 (1939), pp.~888--890.

\bibitem{Ma48}
{\sc A.~I. Mal'tsev}, {\em On the imbedding of group algebras in division
  algebras}, Dokl. Akad. Nauk SSSR, n. Ser., 60 (1948), pp.~1499--1501.

\bibitem{MaNaPl00i}
{\sc V.~Maz'ya, S.~Nazarov, and B.~Plamenevskij}, {\em Asymptotic theory of
  elliptic boundary value problems in singularly perturbed domains. {V}ol.
  {I}}, vol.~111 of Operator Theory: Advances and Applications, Birkh{\"a}user
  Verlag, Basel, 2000.
\newblock Translated from the German by Georg Heinig and Christian Posthoff.

\bibitem{MaNaPl00ii}
\leavevmode\vrule height 2pt depth -1.6pt width 23pt, {\em Asymptotic theory of
  elliptic boundary value problems in singularly perturbed domains. {V}ol.
  {II}}, vol.~112 of Operator Theory: Advances and Applications, Birkh{\"a}user
  Verlag, Basel, 2000.
\newblock Translated from the German by Plamenevskij.

\bibitem{MazyaPlamenevskii:1984}
{\sc V.~G. Maz'ya and B.~A. Plamenevskij}, {\em On the coefficients in the
  asymptotics of solutions of elliptic boundary value problems in domains with
  conical points}, Transl., Ser. 2, Am. Math. Soc., 123 (1984), pp.~57--88.

\bibitem{MoMoPo02}
{\sc A.~B. Movchan, N.~V. Movchan, and C.~G. Poulton}, {\em Asymptotic models
  of fields in dilute and densely packed composites}, Imperial College Press,
  London, 2002.

\bibitem{Ne49}
{\sc B.~H. Neumann}, {\em On ordered division rings}, Trans. Amer. Math. Soc.,
  66 (1949), pp.~202--252.

\bibitem{NoSo13}
{\sc A.~A. Novotny and J.~Soko{\l}owski}, {\em Topological derivatives in shape
  optimization}, Interaction of Mechanics and Mathematics, Springer,
  Heidelberg, 2013.

\bibitem{OMalley2014}
{\sc R.~E. O'Malley}, {\em Historical Developments in Singular Perturbations},
  Springer International Publishing, Cham, 2014.

\bibitem{Sh01}
{\sc M.~A. Shubin}, {\em Pseudodifferential operators and spectral theory},
  Springer-Verlag, Berlin, second~ed., 2001.
\newblock Translated from the 1978 Russian original by Stig I. Andersson.

\end{thebibliography}

\end{document}